\documentclass[]{amsart}
\usepackage{amsmath,amsthm,amsfonts,amssymb,amscd,latexsym}
\usepackage[utf8]{inputenc}
\usepackage[dvipsnames]{xcolor}
\usepackage{tikz}
\usepackage{mathrsfs,bbm}
\usepackage{mathtools}
\usepackage[shortlabels]{enumitem}
\setenumerate{label={\arabic*.}, ref={\arabic*}}
\usepackage{adjustbox}
\usepackage{comment}
\usepackage{thmtools,mathtools}

\usepackage{nameref,%\nameref
hyperref}%autoref,
\usepackage[nameinlink,capitalise]{cleveref}
\definecolor{webgreen}{rgb}{0,.5,0}
\definecolor{webbrown}{rgb}{.6,0,0}
\hypersetup{%
    colorlinks=true, linktocpage=true, pdfstartpage=3, pdfstartview=FitV,%
    breaklinks=true, pdfpagemode=UseNone, pageanchor=true, pdfpagemode=UseOutlines,%
    plainpages=false, bookmarksnumbered, bookmarksopen=true, bookmarksopenlevel=1,%
    hypertexnames=true, pdfhighlight=/O,%nesting=true,%frenchlinks,%
    urlcolor=webbrown, linkcolor=RoyalBlue, citecolor=webgreen, %pagecolor=RoyalBlue,%
    pdfsubject={},%
    pdfkeywords={},%
    pdfcreator={pdfLaTeX},%
}   
\usepackage[textsize=scriptsize]{todonotes}

\newtheorem{theorem}{Theorem}[section]
\newtheorem{lemma}[theorem]{Lemma}
\newtheorem{proposition}[theorem]{Proposition}
\newtheorem{corollary}[theorem]{Corollary}

\theoremstyle{definition}
\newtheorem{definition}[theorem]{Definition}
\newtheorem{example}[theorem]{Example}

\newtheorem{remark}[theorem]{Remark}
\newtheorem{assumption}[theorem]{Assumption}

\definecolor{webgreen}{rgb}{0,.5,0}

\newcommand{\N}{\ensuremath{\mathbb{N}}}
\newcommand{\Z}{\ensuremath{\mathbb{Z}}}
\newcommand{\R}{\ensuremath{\mathbb{R}}}
\newcommand{\Q}{\ensuremath{\mathbb{Q}}}
\newcommand{\Comp}{\ensuremath{\mathbb{C}}}
\newcommand{\U}{\ensuremath{\mathcal{U}}}
\newcommand{\Uz}{\ensuremath{\U_{o}}}
\newcommand{\en}[1]{\prescript{* \!}{}{#1}}

\newcommand{\tsse}{\text{ if and only if }}

\newcommand{\orig}{{\mathbf{0}}}% Bold 1
\renewcommand{\geq}{\geqslant}

\newcommand{\simo}{\mathord{\sim}}

\newcommand{\Va}{\ensuremath{\mathbf{V}}} %VARIETY
\DeclareMathOperator{\V}{\mathsf{V}} % VANISHING LOCUS
\DeclareMathOperator{\C}{\mathsf{I}} % IDEAL (=CONGRUENCE) vanishing on a set
\DeclareMathOperator{\VI}{\mathsf{V}_{\R}} % VANISHING LOCUS relative to the real unit interval
\DeclareMathOperator{\CI}{\mathsf{I}_{\R}} % IDEAL (=CONGRUENCE) vanishing on a set relative to the real unit interval
\DeclareMathOperator{\VU}{\mathsf{V}} % VANISHING LOCUS relative to the hyperreal unit interval
\DeclareMathOperator{\CU}{\mathsf{I}} % IDEAL (=CONGRUENCE) vanishing on a set relative to the hyperreal unit interval
\DeclareMathOperator{\CC}{\mathscr{I}} % I-Functor: Free/I(\theta) 
\DeclareMathOperator{\VV}{\mathscr{V}} % V-Functor:  V(\theta)
\DeclareMathOperator{\F}{\mathscr{Free}} % FREE ALGEBRA
\DeclareMathOperator{\Max}{\mathtt{Max}} % MAXIMAL SPECTRUM
\DeclareMathOperator{\UU}{\mathscr{U}} % homeomorphism of quotient and Spec
\newcommand{\A}{A}% the algebra A
\newcommand{\seq}{\subseteq}% subset
\newcommand{\imm}{\iota}
\newcommand{\df}{\coloneqq}
\newcommand{\eps}{\varepsilon}

\DeclareMathOperator{\spec}{\mathtt{Spec}}
\DeclareMathOperator{\st}{\mathtt{st}}
\DeclareMathOperator{\free}{\mathscr{F}}
\newcommand{\freen}{\mathscr{\free}^{\ell}_n}
\newcommand{\freersn}{\mathscr{\free}^{v}_n} %FREE L-GROUP n-GENERATED
\newcommand{\freek}{\mathscr{\free}^{\ell}_{\kappa}} %FREE L-GROUP K-GENERATED
\newcommand{\freersk}{\mathscr{\free}^{v}_{\kappa}} %FREE riesz space K-GENERATED

\newcommand{\genid}[1]{\langle#1\rangle}

\newcommand{\env}[1]{\langle #1 \rangle}

\newcommand{\ort}[1]{#1^{\perp}}

\newcommand{\ortp}[1]{\langle#1\rangle^{\perp}}

\newcommand{\lr}{L^{n}_\mathbb{R}}
\newcommand{\lz}{L^{n}_\mathbb{Z}}

\newcommand{\cone}{\mathsf{Cone}}
\DeclareMathOperator{\red}{red}

\newcommand{\apal}{\mathsf{G^\alpha}}
\newcommand{\vpal}{\mathsf{R^\alpha}}
\newcommand{\sdefzal}{\mathsf{S^{\alpha}_{def \, \mathbb{Z}}}}
\newcommand{\sdefral}{\mathsf{S^{\alpha}_{def \, \mathbb{R}}}}

\DeclareMathOperator{\VUst}{\VU^o}
\newcommand{\indu}{\mathsf{u}}
\newcommand{\indv}{\mathsf{v}}
\newcommand{\indw}{\mathsf{w}}
\newcommand{\pfin}{\mathcal{P}_{\text{fin}}}
\newcommand{\pfink}{\pfin(\kappa)}
\newcommand{\pfinom}{\pfin(\omega)}

\begin{document}

\title{Baker-Beynon duality beyond semisimplicity}

\author{Luca Carai}
\address{University of Milan, Department of Mathematics, via Cesare Saldini, 50 20133 Milan (MI), Italy.}
\email{luca.carai.uni@gmail.com}
\author{Serafina Lapenta}
\address{University of Salerno, Department of Mathematics, Piazza Renato Caccioppoli, 2 84084 Fisciano (SA), Italy.}
\email{slapenta@unisa.it}
\author{Luca Spada} 
\address{University of Salerno, Department of Mathematics, Piazza Renato Caccioppoli, 2 84084 Fisciano (SA), Italy.}
\email{lspada@unisa.it}

 \begin{abstract}
Combining tools from category theory, model theory, and non-standard analysis we extend Baker-Beynon dualities to the classes of \emph{all} Abelian $\ell$-groups and \emph{all}  Riesz spaces (also known as vector lattices).  The extended dualities have a strong geometrical flavor, as they involve a non-standard version of the category of polyhedral cones and piecewise (homogeneous) linear maps between them. We further show that our dualities are induced by the functor $\spec$, once it is understood how to endow it with ``coordinates'' in some ultrapower $\U$ of $\R$. This also allows us to characterize the topological spaces arising as spectra of Abelian $\ell$-groups and Riesz spaces as certain subspaces of $\U^{\kappa}$  endowed with the Zariski topology given by definable functions in their respective languages. Furthermore, we provide some applications of the extended duality by characterizing, in geometrical terms, semisimplicity, Archimedeanity, and the existence of weak and strong order-units.  Finally, we show that our dualities afford a neat and simpler proof of Panti's celebrated characterization of the prime ideals in free Abelian $\ell$-groups and Riesz spaces.
\end{abstract}
\maketitle

%%%%%%%%%%% %%%%%%%%%%% SECTION %%%%%%%%%%%%%%%%%%%%%% 
\section{Introduction}
\label{sec:intro}
%%%%%%%%%%%%%%%%%%%%%%%%%%%%%%%%%%%%%%%%%%%%%%%%%% 
The idea of representing an abstract structure using its \emph{spectrum} of maximal or prime ideals has a long and interesting history. Providing a complete account would take far more space than the one available for this introduction.  In the next few paragraphs we briefly recall some of its crucial advancements which are relevant to this paper, highlighting the interplay between functional analysis and algebraic geometry.
  
For an algebraically closed field $k$, Hilbert's \emph{Nullstellensatz} establishes a bijection between the points of $k^{n}$ and the \emph{maximal} spectrum of the polynomial ring $k[x_{1},\dots,x_{n}]$. As a consequence, any affine variety can be viewed as the set of maximal ideals of its coordinate ring and \emph{a posteriori} this gives a way to represent finitely generated commutative complex algebras without nilpotents as algebras of functions.

In his time-honored \cite{Gelfand:1941}, Gelfand took the above idea one step further providing a way to represent any \emph{abstract} (commutative, complex, and unital) Banach algebra as an algebra of continuous functions over its maximal spectrum, endowed with the weak topology. A key step in the proof is the observation that the quotient of any Banach algebra over a maximal ideal is isomorphic to $\Comp$.  
Similar ideas were developed independently by Stone \cite{Stone:1937} and Yosida \cite{yosida1941vector} in the case of divisible lattice ordered groups, and Kakutani \cite{Kakutani:1941} and the Krein brothers \cite{Krein:1940} for Riesz spaces. In all these cases a crucial property is that the structures are complete with respect to a norm, which is either derived from the algebraic structure or primitively assumed to exist.  In the 1970s, in a series of papers Baker and Beynon \cite{baker68,MR332617,beynon1974,beynon77} succeeded in dropping the assumption of norm-completeness by restricting to the class of finitely generated ---still Archimedean--- lattice ordered groups and Riesz spaces. Their representation is written in the language of algebraic geometry, making no reference to the maximal spectrum.  

The next breakthrough in this line of research was again made in algebraic geometry: starting from ideas developed in the 1950s by Chevalley, Nagata, and Serre\footnote{In the spirit of this introduction, it seems relevant to observe that, according to \cite[Footnote 29]{Cartier}, it was Martineau ---an expert in analysis!--- that suggested to Serre the possibility of using the spectrum of an arbitrary commutative ring as a foundation for algebraic geometry.}, Grothendieck shows that the spectrum of \emph{prime} ideals can be used to represent any commutative ring with unit as an algebra of functions. The main difference with the previous representation is that the quotients of a ring over its prime ideals need not have the same isomorphism type. This led Grothendieck to consider the notion of \emph{sheaf}. To each point $p$ of the spectrum is attached a local ring, called the \emph{stalk} at $p$ and the elements of the ring are represented as \emph{global sections} over the spectrum.

In this paper we are concerned with similar questions in the case of lattice ordered Abelian groups and Riesz spaces.  Recall that an \emph{Abelian}\footnote{All groups appearing in this paper are tacitly assumed to be Abelian.} lattice ordered group (\emph{$\ell$-group}, for short) is a  group $(G,+,0)$ endowed with a lattice order that is compatible with $+$, i.e., if $x\le y$ then $x+z\le y+z$. A \emph{Riesz space} (also called \emph{vector lattice}) $V$ is a vector space over $\R$ that is endowed with a partial order that makes it an $\ell$-group and that is compatible with the scalar multiplication, namely, if $0\le v\in V$ and $0\le a\in \R$, then $av\ge 0$.  An $\ell$-group (or a Riesz space) $G$ is called \emph{Archimedean} if for every $x,y\in G$, $x\le 0$ whenever $nx\le y$ for all $n\in\N $.  For any $x\in G$, we set $\lvert x\rvert=x\vee (-x)$.  An \emph{ideal} of $G$ is a subgroup $J$ of $G$ such that $x\in J$ and $\lvert y\rvert \le \lvert x\rvert$ imply $y\in J$. If $G$ is a Riesz space, it is easy to see that any ideal of $G$ is a vector subspace. 
The lattice of congruences of both $\ell$-groups and Riesz spaces is isomorphic to the lattice of ideals. Hence, we freely speak of quotients over an ideal.
A proper ideal $P$ of an $\ell$-group or Riesz space $G$ is called \emph{prime} if the quotient $G/P$ is linearly ordered. Both for $\ell$-groups and Riesz spaces, being a prime ideal is equivalent to the following property: $a\wedge b\in P$ implies $a\in P$ or $b\in P$ (see \cite[Proposition 2.4.3]{BKW} and \cite[Theorem 33.2]{RS1}). The set of prime ideals of $G$ is called \emph{spectrum} and denoted by $\spec(G)$. An ideal is called \emph{maximal} if it is proper and maximal with respect to the inclusion order, the set of maximal ideals of $G$ is denoted by $\Max(G)$.  It follows that $M$ is a maximal ideal if and only if the quotient over $M$ is simple (and non-trivial), that is, without non-trivial ideals.
An algebra is called \emph{semisimple} if it is a subdirect product of simple algebras.  In $\ell$-groups and Riesz spaces semisemplicity implies Archimedeanity and they are equivalent in the finitely generated setting.  However, in general Archimedeanity does not imply semisemplicity \cite[Example 27.8]{RS1}.

Baker-Beynon duality (see \cite[Theorem~1.1 and subsequent comments]{beynon77} and \cite[Corollary 1 to Proposition 2.1]{beynon1974}) provides a geometrical description of finitely generated Archimedean $\ell$-groups and Riesz spaces.  To precisely state the content of these results we need to recall some concepts from polyhedral geometry.

Let $\kappa$ be a cardinal. A continuous function $f\colon\mathbb{R}^\kappa \to \mathbb{R}$ is  \emph{piecewise linear} if there exists a finite set $\{l_1, \ldots, l_m\}$ of linear (homogeneous) polynomials with variables ranging in  $\{ x_\beta \mid \beta < \kappa\}$, called \emph{components}, such that for each $x \in \mathbb{R}^\kappa$ one has $f(x)=l_i(x)$ for some $i=1, \ldots, m$. The function $f$ is said to have \emph{integer coefficients} if the components can be chosen to have integer coefficients.  Notice that, although $\kappa$ might be infinite, piecewise linear functions only depend on finitely many variables.
By a \emph{cone} we mean a subset of $\R^\kappa$  closed under multiplication by non-negative scalars. A \emph{closed cone} is a cone which is also closed in the Euclidean topology of $\R^\kappa$.

Both $\ell$-groups and Riesz spaces are equationally definable classes of algebras, i.e., they form \emph{varieties}. Consequently, free objects in both categories exist; we write $\freek$ for the free $\kappa$-generated $\ell$-group and $\freersk$ for the free $\kappa$-generated Riesz space.  
In \cite[Theorem 2.4 and Section 7]{baker68}, Baker proved that $\freek$ is isomorphic to the $\ell$-group of continuous piecewise linear functions with integer coefficients on $\R^{\kappa}$, and $\freersk$ is isomorphic to the Riesz space of continuous piecewise linear functions with real coefficients on $\R^{\kappa}$.

In \cite{beynon1974}, Beynon defines a Galois connection between the powersets of $\freersk$  and of $\R^{\kappa}$, for any finite $\kappa$:
\begin{equation}\label{eq:BBoperators}
\begin{split}
 \text{for }T&\seq \freersk\qquad \V_{\R}(T)\df\{ x \in {\R}^{\kappa} \mid t(x)=0 \mbox{ for all } t \in T \},\\
\text{for }S&\seq \R^{\kappa}\hspace{0.6cm}\quad \C_{\R}(S)\df\{ t \in \freersk \mid t(x)=0 \mbox{ for all } x \in S \} .
\end{split}
\end{equation}
The fixed points of the above connection are the intersections of maximal ideals, on the one side, and the closed cones on the other. It follows that the pair of operators in \eqref{eq:BBoperators} lifts to a contravariant equivalence between the category of finitely generated Archimedean Riesz spaces and the category of \emph{closed cones} and piecewise linear functions. By \cite[Theorem~1.1 and subsequent comments]{beynon77}, analogous results hold for $\ell$-groups.\footnote{Observe that the results of Beynon straightforwardly generalize beyond the finitely generated case to any semisimple algebra.}

It is well known that any non-trivial $\ell$-group that embeds into $\mathbb{R}$ is simple. Moreover, by Hölder Theorem (see, e.g., \cite[Théorème 2.6.3]{BKW}), any simple $\ell$-group embeds into $\R$ and the embedding is unique up to multiplication by a positive real number.\footnote{For Riesz spaces a stronger version of this holds: any simple Riesz space is isomorphic to $\R$ \cite[Theorem 26.11]{RS1}.}
Now, let $I$ be an arbitrary ideal of $\freek$ such that $G=\freek/I$ is semisimple and let $(g_\beta)_{\beta<\kappa}$ the equivalence classes in $G$ of the free generators.  Since quotients over maximal ideals are simple, the composition of the projection $G\to G/M$ with an embedding of $G/M$ into $\mathbb{R}$ provides a map $\pi_{M}\colon G\to \R$,
 for each $M \in \Max(G)$. In turn these functions determine  a map $e \colon \Max(G) \to \mathbb{R}^\kappa$ defined by
\begin{equation*}\label{eq:embeddingMax}
e(M) = \big(\pi_{M}(g_{\beta})\big)_{\beta < \kappa}.
\end{equation*}
If $\Max(G)$ is endowed with the Zariski topology, then $e$ becomes a topological embedding of $\Max(G)$ into $\R^{\kappa}$ with the Euclidean topology. Thus, $\Max(G)$ is homeomorphic to its image through $e$ which lays inside $\V_{\R}(I)\subseteq\R^\kappa$.
This observation, combined with the duality induced by \eqref{eq:BBoperators}, should not lead the reader to believe that $\Max(G)$ contains enough information to reconstruct $G$, as the example of the $\ell$-groups $\R$ and $\Z$ easily shows: $\Max(\R)$ and $\Max(\Z)$ are both singletons.  Indeed, it is the accompanying embedding of $\Max(G)$ into $\R^{\kappa}$ (which we call the \emph{coordinatization} of $\Max(G)$) that completely characterizes $G$ up to isomorphism. Observe that in the finitely generated setting the embedding $e \colon \Max(G) \to \mathbb{R}^n$ can be defined in a canonical way: if $g_1, \dots
,g_n$ are the generators of $G$, the embedding of $G/M$ into $\R$ can be chosen such that $e(M)=\big(\pi_{M}(g_{1}),\dots,\pi_{M}(g_{n})\big)$ lies on the hypersphere $\mathbb{S}^{n-1}$ for each $M \in \Max(G)$. It follows that, when $G$ is finitely generated, $\Max(G)$ is homeomorphic to $\V_{\R}(I) \cap  \mathbb{S}^{n-1}$.

It has been clear that to go beyond Baker-Beynon duality maximal ideals should be replaced by prime ideals. However, the previous example of the $\ell$-groups $\R$ and $\Z$ again shows that $\spec$ alone does not suffice to characterize an $\ell$-group: $\spec(\R)$ and $\spec(\Z)$ are both singletons. Some sort of coordinatization is needed also in this case. In \cite{keimel1968representation}, Keimel showed that any  $\ell$-group can be represented as the algebra of global sections of a sheaf based on its spectrum of prime ideals.  Similar sheaf representations for Riesz spaces can be found in \cite{di2021sheaf}.
The study of the spectrum of $\ell$-groups attracted much attention (see \cite[Chapitre 10]{BKW}, \cite{MR1675803,MR1409232,MR1296582}).  A major advance was made by Panti in \cite{MR1707667}, where he completely characterized the prime ideals of finitely generated free $\ell$-groups and Riesz spaces as the ideals of functions vanishing along some ``direction''.
A major breakthrough in the study of the topology of spectra of $\ell$-groups was recently made by Wehrung.  In \cite[Corollary~9.2, Section~10]{MR3894048}  he showed that: $(i)$ second countable spectra of $\ell$-groups with the Zariski topology are exactly the generalized spectral spaces whose specialization order forms a root system and $(ii)$ the lattices of compact-open sets of the spectra of $\ell$-groups do not admit a first-order characterization.

In this paper we extend Baker-Beynon duality beyond the semisimple case (\Cref{thm:duality}). The main tools to obtain such an extension are:  the results on general affine adjunctions established in \cite{CMS} (which we briefly recall in \Cref{sec:prelim:adjunction}), and the existence, for any cardinal $\alpha$, of an ultrapower $\U$ of $\mathbb{R}$ containing a copy of each linearly ordered $\ell$-group and Riesz space with cardinality smaller than $\alpha$ (\Cref{thm:ultrap-embedding}). In \Cref{s:BB-comparison} we compare these dualities with Baker-Beynon duality. In \Cref{sec:first-properties} we show that these extensions induce generalized spectral topologies on $\U^\kappa$ minus the origin (\Cref{t:Uzkappa generalized spectral}) and that spectra of $\ell$-groups and Riesz spaces are retracts of their closed subsets (\Cref{c:spec-and-V}). These results provide the coordinatization of spectra mentioned above.  In \Cref{sec:archim} we dually characterize the existence of weak and strong order-units (\Cref{t:duals-of-units}). We also describe the subsets of $\U^\kappa$ corresponding to the intersection of all maximal ideals (\Cref{p:intersection of max ideals dual}) and the ideal of nilpotents (\Cref{t:dual-of-N}). This description yields a dual characterization of semisimplicity and Archimedeanity. We postpone to \cref{appendix} a dual description of  an $\omega$-generated $\ell$-group or Riesz space whose ideal of nilpotents does not coincide with the intersection of all maximal ideals.
In \Cref{sec:irreducible} we use tools from non-standard analysis to give a geometrical description of the irreducible closed subsets of $\U^{n}$ minus the origin, with $n \in \N$ (\Cref{lem:clx-coneix,lem:clx-coneix-lgroups}). We show that such subsets are in correspondence with some finite sequences of vectors in $\mathbb{R}^n$ that we call \emph{indexes}. This correspondence affords a canonical choice of the coordinatization of the spectra.
In \Cref{sec:Panti} we use the results of \Cref{sec:irreducible} to obtain an alternative proof of the characterization of prime ideals in finitely generated free $\ell$-groups and Riesz spaces given in \cite{MR1707667} (see \Cref{c:panti}).

We assume the reader is familiar with the background on $\ell$-groups and Riesz spaces  given in the Introduction.  The interested reader is kindly referred to \cite{anderson1988free,BKW,Darnel,fuchs2011partially,kopytov2013theory,RS1,RS2} for further details.

%%%%%%%%%%% %%%%%%%%%%% SECTION %%%%%%%%%%%%%%%%%%%%%% 
\section{Affine adjunctions}
\label{sec:prelim:adjunction}
%%%%%%%%%%%%%%%%%%%%%%%%%%%%%%%%%%%%%%%%%%%%%%%%%% 
In \cite{CMS}, the authors present a categorical abstraction of the classical affine adjunction of algebraic geometry. In particular, in \cite[Section 4]{CMS} they  show that, for any variety of algebras $\Va$, any choice of $A\in \Va$ induces a contravariant adjunction between $\Va$ and the category of subsets of $A^{\kappa}$ ---with $\kappa$ ranging over all cardinals--- and morphisms defined as follows.  A function $f\colon A^\kappa\to A$ is called  \emph{definable} if there exists a term $t$ in the language of $\Va$ such that $f(x)=t(x)$ for all $x\in A^\kappa$. The definition easily generalizes to functions from $S\seq A^{\mu}$ into $S'\seq A^{\nu}$, with $\mu$ and $\nu$ cardinals, indeed one says that $f\colon S\to S'$ is definable if there exists a family of terms $\{t_{\beta}\}_{\beta<\nu}$ such that 
\[f(x)=(t_{\beta}(x))_{\beta<\nu}\qquad \text{for any } x\in S.\]
In the context of $\ell$-groups (Riesz spaces, respectively) we call these maps $\Z$-definable ($\R$-definable, respectively).

In the rest of this section we recall some results of \cite{CMS} in the particular cases of $\ell$-groups or Riesz spaces. Thus, henceforth $\Va$ will be used to indicate either the variety of $\ell$-groups or Riesz spaces,  $\free_{\kappa}$ will denote the free $\kappa$-generated algebra in $\Va$,  $\A$ will be a fixed but arbitrary algebra in $\Va$, and $\kappa, \mu,\nu$ will always be cardinals. 

\begin{definition}\label{def:I and V for U}
 For any  $T \subseteq \free_{\kappa}$ and $S \subseteq \A^{\kappa}$, we define the following operators.
\begin{align*}
\V_{\A}(T)\df&\{ x \in {\A}^{\kappa} \mid t(x)=0 \mbox{ for all } t \in T \},\\
\C_{\A}(S)\df&\{ t \in \free_{\kappa} \mid t(x)=0 \mbox{ for all } x \in S \}.
\end{align*}
\end{definition} 
When $T=\{t\}$ and $S= \{ x \}$, we simply write $\V_{\A}(t)$ and $\C_{\A}(x)$. 
It is straightforward to see that $\C_{\A}(S)$ is an ideal of $\free_{\kappa}$ for every $S \subseteq \A^{\kappa}$.

\begin{lemma}[Basic Galois connection, \mbox{\cite[Lemma 2.3]{CMS}}]\label{l:galois} 
For every  $S\subseteq\A^{\kappa}$ and $T\subseteq\free_{\kappa}$, \[T\subseteq\C_{\A}{(S)}\tsse S\subseteq\V_{\A}{(T)}.\]
In words, the functions $\C_{\A}$ and $\V_{\A}$ form a Galois connection.
\end{lemma}
An immediate consequence of \Cref{l:galois} is that the compositions $\V_{\A} \circ \C_{\A}$ and $\C_{\A} \circ \V_{\A}$ are closure operators. Notice also that:
\[
\V_{\A}(T)= \bigcap \{ \V_{\A}(t) \mid t \in T \} \quad \mbox{and} \quad \C_{\A}(S)= \bigcap \{ \C_{\A}(x) \mid x \in S \}.
\]
We will tacitly assume that to each algebra $G\in\Va$ is associated some cardinal $\kappa$ and an ideal $J$ of $\free_{\kappa}$ such that $G$ is isomorphic to $\free_{\kappa}/J$.
The above Galois pair always induces a pair of contravariant functors $\CC$ and $\VV$ between the category  $\Va$ and the category of subsets of ${\A}^{\kappa}$ with definable maps between them.  For the objects we define, for any  subset $S\subseteq {\A}^{\kappa}$,
\[
\CC{(S)}\df\free_{\kappa}/\C_{\A}{(S)}
 \]
 and, for any algebra $\free_{\kappa}/J$,
\[
\VV{(\free_{\kappa}/J})\df\V_{\A}{(J)}.
\]

 Given $S\subseteq {\A}^{\kappa}$ and $T\subseteq {\A}^{\nu}$, let $\lambda \colon S \to T$
be a definable map, and let $(t_{\beta})_{\beta< \nu}$ be a $\nu$-tuple of defining terms for $\lambda$. Then there is an induced function $\CC{(\lambda)}\colon \CC{(T)} \to \CC{(S)}$ which acts for each $s \in \F_{\nu}$ as follows:
\[
\frac{s}{\C_{\A}{(T)}} \in \CC{(T)} \hspace{.3cm}\overset{\CC{(\lambda)}}{\longmapsto}\hspace{.3cm} \frac{s[(t_{\beta})_{\beta<\nu}]}{\C_{\A}{(S)}} \in \CC{(S)}.\]
Let $h \colon \F_{\kappa}/J_{1}\to \F_{\nu}/J_{2}$ be a $\Va$-homomorphism. 
For each $\alpha<\mu$, let  $\pi_{\alpha}$ be the projection term on the $\alpha^{\text{th}}$ coordinate, and let $\pi_{\alpha}/J_{1}$ denote 
the equivalence class of $\pi_{\alpha}$ modulo $J_{1}$. Fix, for each $\alpha$, an arbitrary $f_\alpha \in \F_{\nu}$ such that $f_\alpha/J_2=h(\pi_\alpha/J_1)$. For any $(x_{\beta})_{\beta< \nu}\in\V_{\A}{(J_{2})}$, set
\[
\VV{(h)}((x_{\beta})_{\beta<\nu})\df\big(\, f_{\alpha}(\, (x_{\beta})_{\beta<\nu}\,)\,\big)_{\alpha<{\kappa}}.
\]
\begin{theorem}[{\cite[Corollary 4.8]{CMS}}] \label{t:adjunction} 
The functors $\CC$ and $\VV$ form a contravariant adjunction.
\end{theorem}

 As a general fact, any adjunction restricts to an equivalence between the full subcategories of fixed points of the endofunctors obtained by composition.   This motivates the study of the fixed points of $\VV\circ\CC$ and $\CC\circ\VV$.  Notice that immediate calculations show that  $\VV\circ\CC(S)$ is just $\V_{\A}\C_{\A}(S)$ and $\CC\circ\VV(\F_{\mu}/J)$ is $\F_{\mu}/\C_{\A}\V_{\A}(J)$.
The fixed points of $\C_{\A}\circ\V_{\A}$ have an explicit characterization, provided by the following lemma.
\begin{lemma}[\mbox{\cite[Lemma 4.11 and Theorem 4.15]{CMS}}]\label{l:null-cong}
Let $J$ be an ideal of $\free_{\kappa}$.
\begin{enumerate}
\item\label{l:null-subdirect}  $\C_{\A} \V_{\A}(J)=J$ if and only if  $J=\C_{\A}(S)$
for some subset $S$ of $A^\kappa$.
\item\label{p:cong-embed} The algebra $\free_{\kappa}/J$ embeds into $\A$ if and only if there exists $x\in \A^{\kappa}$ such that $J=\C_{\A} (x)$.
\end{enumerate}
\end{lemma}
\begin{remark}\label{r:x-not-unique-for-prime-ideals}
It is worthwhile to notice that if $J$ is a prime ideal of $\free_{\kappa}$, then the element $x\in \A^{\kappa}$ such that $J=\C_{\A}(x)$ is not necessarily unique. This corresponds to the fact that $\free_{\kappa}/J$ might be embedded into $\A$ in different ways.
\end{remark}

So, the ideals fixed by the composition $\C_{\A} \circ \V_{\A}$ are exactly the ones that can be obtained as an intersection of ideals $J$ such that $\free_{\kappa}/J$ embeds into $\A$.

Baker-Beynon duality can be obtained as a particular case of the above setting by choosing $A=\R$; in \Cref{s:BB-comparison} we provide a more precise comparison. In fact, although Baker-Beynon duality is stated only for Archimedean finitely generated structures, the general framework outlined above already provides an extension to all semisimple algebras. 

Finally, we record here the following observation for further reference.
\begin{lemma}\label{l:PWL-I(C)}
For any subset $S\seq \A^{\kappa}$, $\free_{\kappa}/\C_{\A}(S)$ is isomorphic to the algebra of definable maps from $S$ to $\A$, with operations defined pointwise.
\end{lemma}
\begin{proof}
Consider the map $f$ that sends $t\in \free_{\kappa}$ to the definable map associated to $t$ restricted to $S$; it is routine to check that $f$ factors through $\free_{\kappa}/\C_{\A}(S)$ and is onto.  Furthermore, notice that $f(t)=0$ if and only if $t(s)=0$ for any $s\in S$, if and only if $t\in \C_{\A} (S)$. The statement then follows from the first isomorphism theorem.
\end{proof}

%%%%%%%%%%% %%%%%%%%%%% SECTION %%%%%%%%%%%%%%%%%%%%%% 
\section{The extended dualities for \texorpdfstring{$\ell$}{l}-groups and Riesz spaces}\label{sec:duality}

\Cref{l:null-cong} readily implies that an algebra of $\Va$ is fixed ---up to an iso--- by $\CC\circ \VV$ if and only if it is a subdirect product of algebras that embed into the chosen algebra $\A$. Therefore, the composition $\CC\circ \VV$  fixes every algebra in $\Va$ if and only if all the subdirectly irreducible algebras of $\Va$ embed into $\A$.  Unfortunately, such an $A$ does not exist as it can be easily shown that there exist subdirectly irreducible $\ell$-groups and Riesz spaces of arbitrarily large cardinality.  Nevertheless, as we prove in the next theorem, for any cardinal $\alpha$ one can find an algebra that contains isomorphic copies of all linearly ordered algebras in $\Va$ of cardinality smaller than $\alpha$. Since all subdirectly irreducible $\ell$-groups and Riesz spaces are linearly ordered, with this choice of $\A$ the composition $\CC\circ \VV$  fixes every algebra in $\Va$ with cardinality smaller than $\alpha$.  Before stating the theorem we briefly recall some standard tools in model theory.

Recall that two first-order structures for the same language $\mathcal{L}$ are called  \emph{elementarily equivalent} if they satisfy the same first-order $\mathcal{L}$-sentences. 
Let $A$ be a first-order structure, $I$ a set, and $F$ an ultrafilter of the powerset $\mathscr{P}(I)$.  The relation on $A^I$ defined by
\[ f\equiv_F g \tsse \{i\in I \mid f(i)=g(i)\}\in F \]
is a congruence and the quotient $A^I/{\equiv_F}$ is said to be an \emph{ultrapower} of $A$. 

The following result should be compared with the celebrated Hahn's embedding theorem for ordered groups \cite{Hahn} (and Hausner and Wendel's similar result for Riesz spaces \cite{hausner1952ordered}, see also \cite{Clifford}).
\begin{theorem}\label{thm:ultrap-embedding}
Let $\alpha$ be a cardinal.
\begin{enumerate}
\item\label{thm:ultrap-embedding:item1} There exists an ultrapower of $\R$ such that every linearly ordered group of cardinality less than $\alpha$ embeds in it. 
\item\label{thm:ultrap-embedding:item2} 
There exists an ultrapower of $\R$  such that every linearly ordered Riesz space of cardinality less than $\alpha$ embeds in it.
\end{enumerate}
Furthermore, if $\alpha \ge 2^{\omega}$, the ultrapower of \cref{thm:ultrap-embedding:item2} has also the property expressed in  \cref{thm:ultrap-embedding:item1}.
\end{theorem}
\begin{proof}
Without loss of generality we can assume $\alpha$ to be infinite.

\Cref{thm:ultrap-embedding:item1}.  It is known (\cite[Lemme 1.6.8 and Proposition 1.6.9]{BKW}) that any linearly ordered group $G$ can be embedded into a linearly ordered divisible group of cardinality $\max\{\omega,|G|\}$. By \cite[Theorem 4.3.2]{R77}, any non-trivial linearly ordered divisible group is elementarily equivalent to the additive group of the real numbers.  Hence, by \cite[Theorem 4.3.12]{CK}, there exists an ultrapower $\U$ of $\R$ such that all linearly ordered groups of cardinality less or equal to $\alpha$ embed into $\U$.

\Cref{thm:ultrap-embedding:item2}. The proof for Riesz spaces is analogous. We only prove that all non-trivial linearly ordered Riesz spaces are elementarily equivalent. By \cite[Chapter 1, Corollary 7.8]{dries} their theory has quantifier elimination, and hence it is model complete (see, e.g., \cite[Theorem 3.5.1(iv)]{CK}). Furthermore, $\R$ is an \emph{algebraically prime} model, i.e., $\R$ embeds into every non-trivial Riesz space. Hence, the theory of linearly ordered Riesz spaces is complete (see, e.g., \cite[Proposition 3.5.11]{CK}) and the claim follows as in the previous item.

Finally the last part of the statement follows from \cite[Exercise 4.3.25]{CK}.
\end{proof}

Let $\alpha$ be a cardinal. We recall that an algebraic structure $A$ is \emph{$\alpha$-generated} if there exists a family $\{ a_i \mid i\in I\}$ of elements of $A$ that generates $A$ and $\alpha$ is the cardinality of $I$.  

\begin{corollary}\label{cor:existence-U-generators}
Let $\alpha$ be a cardinal. 
\begin{enumerate}
\item\label{cor:existence-U-generators:item1} There exists an ultrapower of $\R$ such that any $\kappa$-generated linearly ordered  $\ell$-group with $\kappa < \alpha$ embeds in it.
\item\label{cor:existence-U-generators:item2} There exists an ultrapower of $\R$ such that any $\kappa$-generated linearly ordered Riesz space with $\kappa < \alpha$ embeds in it.
\item\label{cor:existence-U-generators:item3} Any ultrapower of $\R$ given by a non-principal ultrafilter on a countable set of indexes has the property that all finitely generated linearly ordered $\ell$-groups and Riesz spaces embed into it.
\end{enumerate}
\end{corollary}
\begin{proof}
To prove \cref{cor:existence-U-generators:item1,cor:existence-U-generators:item2} it is sufficient to observe that any $\kappa$-generated  $\ell$-group has cardinality at most $\alpha+\omega \le \alpha +2^{\omega}$ and that any $\kappa$-generated Riesz space has cardinality at most $\alpha+2^{\omega} $. So, it is enough to apply \Cref{thm:ultrap-embedding} to the cardinal $\alpha+2^{\omega}$. \Cref{cor:existence-U-generators:item3} follows from \cite[Theorems 5.1.12 and 6.1.1]{CK} and \cite[Theorem 2.3 and Lemma 2.4]{vanAlten}.
\end{proof}

\begin{assumption}\label{ass:U-embeds}
For the rest of the section, we fix a cardinal $\alpha$ and denote by $\U$ any linearly ordered group (Riesz space) such that all linearly ordered groups (Riesz spaces) generated by less than $\alpha$ elements embed into $\U$. The existence of a linearly ordered $\ell$-group (Riesz space) $\U$ with this property is guaranteed by \cref{cor:existence-U-generators}.
\end{assumption}

We instantiate the framework of \Cref{sec:prelim:adjunction} by setting $A=\U$. Since it will be clear from the context whether we are working in the setting of $\ell$-groups or Riesz spaces, we use the notation $\CU$ and $\VU$ for the operators in both settings. 
For the most part of this paper the proof strategy is the same for both $\ell$-groups and Riesz spaces. Consequently, we prove all results in general and specify when they only apply to one of the two cases or they need additional assumptions.

Let $\kappa < \alpha$ be a cardinal. We write $\orig$ for the origin of $\U^{\kappa}$, i.e., the $\kappa$-tuple whose entries are all equal to $0$.

\begin{lemma}\label{t:null-U}
Let $T\seq\free_{\kappa}$ and $x\in \U^{\kappa}$. 
\begin{enumerate}
\item\label{t:null-U:item0} $\free_{\kappa}=\CU(x)$ if and only if $x=\orig$.
\item\label{t:null-U:item1} $T$ is a prime ideal if and only if $T=\CU(x)$ for some $x \in \U^{\kappa}\setminus \{\mathbf{0}\}$.
\item\label{t:null-U:item2} $\CU\VU(T)$ is the ideal generated by $T$.
\item\label{t:null-U:item3} $T$ is an ideal if and only if $T=\CU\VU(T)$ if and only if $T=\CU(S)$ for some $S \subseteq \U^{\kappa}$.
\end{enumerate}
\end{lemma}

\begin{proof}
\Cref{t:null-U:item0}. Notice that $\CU(\mathbf{0})=\free_\kappa$ because each $t \in \free_\kappa$ induces a homogeneous function from $\U^\kappa$ to $\U$. Conversely, since every projection $\pi_\gamma$ is in $\free_\kappa=\CU(x)$, then $\pi_\gamma(x)=0$. Thus,  $x= \mathbf{0}$.

\Cref{t:null-U:item1}.
An ideal $T$ is prime if and only if  $\free_{\kappa}/T$ is non-trivial and linearly ordered.  In turn, by \Cref{ass:U-embeds}, the latter is equivalent to the fact that $\free_{\kappa}/T$ is non-trivial and embeds into $\U$. Hence, by \Cref{l:null-cong}\eqref{p:cong-embed} and \Cref{t:null-U:item0}, $T$ is prime if and only if $T=\CU(x)$ for some $x \in \U^{\kappa}\setminus\{\orig\}$.

\Cref{t:null-U:item2}. Since $\CU$ and $\VU$ form a Galois connection,
\[
\CU\VU(T)=\CU\left(\bigcup \{\{x\} \mid x \in \VU(T) \}\right)=\bigcap \{\CU(x) \mid x \in \VU(T) \}=\bigcap \{\CU(x) \mid T \subseteq \CU(x) \}.
\]
Thus, it follows from \Cref{t:null-U:item1} that $\CU\VU(T)$ is the intersection of all the prime ideals containing $T$. By \cite[4.2.4]{BKW} and \cite[Theorem 33.5]{RS1}, the latter coincides with the ideal generated by $T$.

\Cref{t:null-U:item3}. It follows from \Cref{t:null-U:item2} that  $T$ is an ideal if and only if $T=\CU\VU(T)$.
If $T=\CU\VU(T)$, then $T=\CU(S)$, where $S=\VU(T)$. For the other implication it is sufficient to recall that $T=\CU(S)$ is an ideal for all $S \subseteq \U^\kappa$.
\end{proof}

\begin{definition}
We say that a subset of $\U^\kappa$ is $\Z$-definable ($\R$-definable, respectively) if it is of the form $\VU(T)$ for some subset $T$ of the free $\ell$-group $\freek$ (free Riesz space $\freersk$, respectively). 
\end{definition}
We are now ready to state the main result of this section.  The names of categories introduced in the table below will only be used, for the sake of clarity, in the statement of next theorem.
\begin{center}
\begin{tabular}{|c| l| l|}
\hline
Name & Objects & Morphisms \\
\hline
$\apal$ & $\kappa$-generated $\ell$-groups, with $\kappa<\alpha$ & $\ell$-group homomorphisms\\
$\vpal$ & $\kappa$-generated Riesz spaces, with $\kappa<\alpha$ & Riesz space homomorphisms\\
$\sdefzal$ & $\Z$-definable subsets of $\U^\kappa$, with $\kappa < \alpha$ & $\Z$-definable maps\\
$\sdefral$ & $\R$-definable subsets of $\U^\kappa$, with $\kappa < \alpha$ & $\R$-definable maps\\
\hline
\end{tabular}
\end{center}

\begin{theorem} \label{thm:duality}
Under the \cref{ass:U-embeds}, the dual adjunction of \Cref{t:adjunction}, with $A\df\U$, restricts to dual equivalences between:
\begin{enumerate}
\item\label{t:duality:item1} $\apal$ and $\sdefzal$.
\item\label{t:duality:item2} $\vpal$ and $\sdefral$.
\end{enumerate}
\end{theorem}

\begin{proof} 
On the one hand, if $\free_\kappa/J$ is a $\kappa$-generated  $\ell$-group, then $\VV(\free_\kappa/J) = \VU(J)$ is a $\Z$-definable subset of $\U^{\kappa}$. On the other hand, if $C\seq\U^{\kappa}$ is $\Z$-definable, then $\CC(C) = \free_\kappa/\CU(C)$ is $\kappa$-generated. Moreover, if $\free_\kappa/J$ if a $\kappa$-generated $\ell$-group, then $J=\CU \VU(J)$ by \Cref{t:null-U}\eqref{t:null-U:item3}, and if $C\seq\U^{\kappa}$ is $\Z$-definable, then $C=\VU \CU(C)$ by definition.
This proves \Cref{t:duality:item1}. An analogous argument yields \Cref{t:duality:item2}.
\end{proof}

\begin{remark}\label{r:duality finitely generated}
Let $\U$ be an ultrapower of $\R$ given by a non-principal ultrafilter on a countable set of indexes $I$. By \Cref{cor:existence-U-generators}(\ref{cor:existence-U-generators:item3}), $\U$ satisfies  \cref{ass:U-embeds} with $\alpha = \omega$. Thus, \cref{thm:duality} yields a duality between the categories of finitely generated $\ell$-groups (Riesz Spaces, respectively) and of $\Z$-definable ($\R$-definable, respectively) subsets of $\U^n$ with $n \in \N$.
\end{remark}

\subsection{Comparison with the Baker-Beynon duality}\label{s:BB-comparison}
In this section we compare Baker-Beynon duality with the one of \Cref{thm:duality}. To this end, we need to add to \Cref{ass:U-embeds} the additional requirement that $\R$ be a subalgebra of $\U$. 
Note that this is not a strong requirement. Indeed, in any non-trivial Riesz space the set of scalar multiples of a non-zero element forms a subalgebra that is isomorphic to $\mathbb{R}$. In addition, if $\U$ is a linearly ordered $\ell$-group satisfying \Cref{ass:U-embeds} with $\alpha \geq 2^{\omega}$, then $\mathbb{R}$ embeds into $\U$.

So, in the remainder of this section, $\R^\kappa$ is a subalgebra of $\U^\kappa$ and thus it makes sense to  apply the operator $\CU$ to subsets of $\mathbb{R}^\kappa$ and compare the pair $\CI$ and $\VI$ with $\C$ and $\V$.

\begin{lemma}\label{l:facts-C-V}
Let $k<\alpha$,  $S \subseteq \R^{\kappa}$ and $T \subseteq \free_\kappa$.
\begin{enumerate}
\item\label{l:facts-C-V:item1} $\CU(S)=\CI(S)$.
\item\label{l:facts-C-V:item5} $\VI(T)=\VU(T) \cap \R^{\kappa}$.
\item\label{l:facts-C-V:item4} $\VI \CI(S)=\VU \CU(S)\cap \R^\kappa$.
\end{enumerate}
\end{lemma}

\begin{proof}
\Cref{l:facts-C-V:item1}. Since $\R$ is a subalgebra of $\U$, we have that $\CU(x)=\CI(x)$ for each $x \in \R^{\kappa}$. It follows that $\CU(S)=\bigcap \{ \CU(x) \mid x \in S \}=\bigcap \{ \CI(x) \mid x \in S \}=\CI(S)$.

\Cref{l:facts-C-V:item5} is an immediate consequence of the definitions of $\VI$ and $\VU$. 

\Cref{l:facts-C-V:item4}. It follows from \Cref{l:facts-C-V:item1} and \Cref{l:facts-C-V:item5} that  $\VI \CI(S) = \VI \CU(S) = \VU \CU(S)\cap \R^\kappa$.
\end{proof}

\begin{proposition}\label{p:null-U-max}
Let $k<\alpha$ and $T \subseteq \free_{\kappa}$.
\begin{enumerate}
\item\label{p:null-U-max:item1} $T$ is a maximal ideal if and only if $T=\CU(x)$ for some $x \in \R^{\kappa}\setminus \{\mathbf{0}\}$.
\item\label{p:null-U-max:item2} $\CU\VI(T)$ is the intersection of all maximal ideals containing $T$.
\item\label{p:null-U-max:item3} The following are equivalent:
\begin{enumerate}[label=(\roman*) , ref = \roman*]
\item\label{p:null-U-max:itemi} $T$ is an intersection of maximal ideals.
\item\label{p:null-U-max:itemii} $T=\CU\VI(T)$.
\item\label{p:null-U-max:itemiii} $T=\CU(S)$ for some $S \subseteq \R^{\kappa}$.
\item\label{p:null-U-max:itemiv} $T$ is an ideal and $\VU(T)= \VU \CU \VI(T)$.
\end{enumerate}
\end{enumerate}  
\end{proposition}

\begin{proof}
\Cref{p:null-U-max:item1}. As mentioned in \Cref{sec:intro}, $T$ is a maximal ideal if and only if $\freek/T$ is non-trivial and embeds into $\R$. Hence, by \Cref{l:null-cong}\eqref{p:cong-embed} and \Cref{t:null-U}(\ref{t:null-U:item0}), $T$ is maximal if and only if there exists $x\in \R^{\kappa}\setminus\{\orig\}$ such that $T=\CI(x)$, which is the same as $\CU(x)$ by \Cref{l:facts-C-V}(\ref{l:facts-C-V:item1}).

\Cref{p:null-U-max:item2}. Observe that $\CU \VI(T)= \CI \VI (T) = \bigcap \{ \CI (x) \mid x\in \VI (T) \}$. Moreover, $x \in \VI(T)$ if and only if $T \subseteq \CI(x)$. Thus, $\CI\VI(T)$ is the intersection of all the ideals of the form $\CI(x)$ containing $T$. By \Cref{p:null-U-max:item1}, this is exactly the intersection of all maximal ideals containing $T$.

\Cref{p:null-U-max:item3}.
The equivalence between \eqref{p:null-U-max:itemi} and \eqref{p:null-U-max:itemii} is an immediate consequence of \Cref{p:null-U-max:item2}. Condition \eqref{p:null-U-max:itemii} clearly implies \eqref{p:null-U-max:itemiii}. If \eqref{p:null-U-max:itemiii} holds, then obviously $T$ is an ideal and
\begin{align*}
\VU(T)&=\VU\CU(S) & \text{By hypothesis}\\
&=\VU\CI(S)& \text{By \Cref{l:facts-C-V}\eqref{l:facts-C-V:item1}}\\
&=\VU\CI\VI\CI(S)& \text{Because $\CI$ and $\VI$ are a Galois connection}\\
&=\VU\CU\VI\CU(S)& \text{By \Cref{l:facts-C-V}\eqref{l:facts-C-V:item1}}\\
&=\VU\CU\VI(T)& \text{By hypothesis}.
\end{align*}
 If \eqref{p:null-U-max:itemiv} holds, then $T=\CU \VU(T)= \CU \VU \CU \VI(T)=\CU\VI(T)$, and so \eqref{p:null-U-max:itemii} holds. 
\end{proof}
Let $\lr$ denote the set of functions $f\colon\R^n \to \R$ defined by homogeneous linear polynomials and let $\lz$ be the subset of $\lr$ in which the polynomials have integer coefficients.
A \emph{half-space} is a set of the form $\{x\in\R^{n}\mid f(x)\ge 0\}$ for $f\in \lr$.  A half-space is called \emph{rational} if the defining $f$ is in $\lz$.
A (\emph{rational}) \emph{polyhedral cone} is a finite union of finite intersections of (rational) half-spaces.

\begin{theorem}\label{teo:dualitySSeFP}
Let $J$ be an ideal of $\free_{\kappa}$.
\begin{enumerate}
\item\label{cor:lin-ord=closure-of-point}  $\free_{\kappa}/J$ is linearly ordered if and only if $J=\CU(x)$  for some $x\in\U^\kappa$.
\item\label{cor:simple=closure-of-real-point}  $\free_{\kappa}/J$ is simple if and only if $J=\CU(x)$ for some $x\in\R^\kappa \setminus \{ \orig \}$.
\item\label{cor:ss=closure-of-enlargement} $\free_{\kappa}/J$ is semisimple if and only if $J=\CU(S)$  for some  $S\seq \R^\kappa$.
\item\label{pro:FinPres=hperpoly1} If $\kappa$ is finite then the Riesz space $\free^{v}_{\kappa}/J$ is finitely presented if and only if $J=\CU(P)$ for some polyhedral cone $P\seq\R^\kappa$.
\item\label{pro:FinPres=hperpoly2} If $\kappa$ is finite then the $\ell$-group $\free^{\ell}_{\kappa}/J$ is finitely presented if and only if $J=\CU(P)$ for some rational polyhedral cone $P\seq\R^\kappa$.
\end{enumerate} 
\end{theorem}
\begin{proof}
\Cref{cor:lin-ord=closure-of-point}. Recall that $\free_{\kappa}/J$ is linearly ordered if and only if $J$ is prime or $J=\free_\kappa$. By \Cref{t:null-U}(\ref{t:null-U:item0},\ref{t:null-U:item1}) $J$ is prime or $J=\free_\kappa$ if and only if there exists $x\in \U^\kappa$ such that $J=\CU(x)$. Note that $\free_{\kappa}/J$ is trivial exactly when $x$ is the origin.

\Cref{cor:simple=closure-of-real-point}. By \cite[Corollaire 2.6.7]{BKW}, $\free_{\kappa}/J$ is simple if and only if $J$ is maximal.  By \Cref{p:null-U-max}(\ref{p:null-U-max:item1}), $J$ is maximal if and only if there exists $x\in \R^\kappa \setminus \{ \orig \}$ such that $J=\CU(x)$.

\Cref{cor:ss=closure-of-enlargement}. By standard results of universal algebra (see, e.g., \cite[Theorem II.6.20]{BS}), $\free_{\kappa}/J$ is semisimple if and only if $J$ is an intersection of maximal ideals. In turn, by  \Cref{p:null-U-max}\eqref{p:null-U-max:item3} this is equivalent to $J= \CU (S)$ for some $S\seq \R^\kappa$. 

\Cref{pro:FinPres=hperpoly1}.
By Baker-Beynon duality, principal ideals are intersection of maximal ideals, and \cite[Lemma~3.2]{baker68} yields that polyhedral cones are exactly the sets of the form 
$\VI(t)$ for $t \in \freersn$. If $\genid{t}$ is the ideal of $\freersn$ generated by $t$, then $\VI(\genid{t}) = \VI(t)$.
Now, $G$ is finitely presented if and only if $G\simeq \free^{v}_n /\genid{t}$. By \Cref{p:null-U-max}(\ref{p:null-U-max:item3}),  $\genid{t}= \CU \VI(t)$ and $\VI(t)$ is a polyhedral cone by the above considerations. Conversely, assume  $J=\CU(P)$ for some  polyhedral cone $P$, then there exists $t$ such that $P=\VI(\genid{t})$. It follows  that $J=\CU (\VI(\genid{t}))$. By \Cref{p:null-U-max}(\ref{p:null-U-max:item3}), the latter equals $\genid{t}$, settling the claim.

\Cref{pro:FinPres=hperpoly2}. The proof is analogous to the one of \Cref{pro:FinPres=hperpoly1}, upon noticing that rational polyhedral cones are exactly the sets of the form $\VI(t)$ for $t\in \free^{\ell}_{n}$ (see \cite[Corollary~2 to Theorem~3.1]{beynon77}).
\end{proof}

\begin{remark}
The correspondences of \Cref{teo:dualitySSeFP} induce categorical dualities. More specifically, there is a dual equivalence between the full subcategory of linearly ordered (simple, semisimple and finitely presented, respectively) $\kappa$-generated $\ell$-groups  and the full subcategory of $\Z$-definable subsets of $\U^{\kappa}$ for some $\kappa < \alpha$ whose objects are the $\VU\CU$-closures of the points in some $\U^\kappa$ (points of $\R^\kappa \setminus \{ \orig \}$, subsets of $\R^\kappa$, polyhedral cones in $\R^\kappa$, respectively) for some $\kappa < \alpha$. Similar dualities can be obtained, in the case of Riesz spaces, by considering analogous full  subcategories.
\end{remark}

Recall from the introduction that a closed cone is a subset of $\R^\kappa$ that is closed in the Euclidean topology and is closed under multiplication by non-negative scalars.

\begin{proposition}\label{p:iso-closed-cones-and-closures}
\mbox{}\begin{enumerate}
\item\label{p:iso-closed-cones-and-closures:item1} The composition $\VU\CU$ induces  an isomorphism between the category of closed cones in $\R^\kappa$ and piecewise linear maps with integer coefficients, with $\kappa<\alpha$, and the category of $\VU\CU$-closures of subsets of $\R^\kappa$ with $\Z$-definable maps between them, with $\kappa<\alpha$.
\item\label{p:iso-closed-cones-and-closures:item2} The composition $\VU\CU$ induces an isomorphism between the category of  closed cones in $\R^\kappa$ and  piecewise linear maps, with $\kappa<\alpha$, and the category of  $\VU\CU$-closures of subsets of $\R^\kappa$ with $\R$-definable maps between them, with $\kappa<\alpha$.
\end{enumerate}
\end{proposition}

\begin{proof}
Assume $A=\VU\CU(S)$ for some $S\seq \R^\kappa$. We map $A$ to the set $A\cap \R^\kappa$, which is a closed cone because $A \cap \R^\kappa =\VU \CU(S) \cap \R^\kappa = \VI \CU(S) = \VI \CI (S)$ by \Cref{l:facts-C-V}. Vice versa, we map each closed cone $C \subseteq \R^\kappa$ to the subset $\VU \CU(C) \subseteq \U^\kappa$. We show that these maps are inverses of each other. Since we have already observed that $\VU \CU(S) \cap \R^\kappa = \VI \CI (S)$, it follows that 
\[
\VU \CU (\VU \CU(S) \cap \R^\kappa) = \VU \CU \VI \CI (S) = \VU \CI \VI \CI (S) = \VU \CI (S) = \VU \CU (S).
\]
On the other hand, if $C \subseteq \R^\kappa$ is a closed cone, then \Cref{l:facts-C-V} implies that
\[
\VU \CU (C) \cap \R^\kappa =  \VI \CI (C) = C.
\]
Since each $\Z$-definable map $f\colon \VU\CU(S_1)\to \VU\CU(S_2)$ is induced by a tuple of terms in the language of  $\ell$-groups, the map $f_{|\VU\CU(S_1) \cap \R^\kappa}:\VU\CU(S_1) \cap \R^\kappa \to \VU\CU(S_2) \cap \R^\kappa$ is a well-defined $\Z$-definable map between closed cones. Vice versa, if $g\colon C_1 \to C_2$ is a $\Z$-definable map between closed cones, then we can consider the map $f\colon \VU \CU(C_1) \to \VU \CU(C_2)$ defined by the same terms that define $g$. It is then straightforward to see that this correspondence gives rise to the required isomorphism. The proof of \Cref{p:iso-closed-cones-and-closures:item2} is analogous.
\end{proof}

It is a straightforward consequence of \Cref{l:facts-C-V} that the compositions of the functors of the dual equivalence of \Cref{thm:duality} with the isomorphisms of \Cref{p:iso-closed-cones-and-closures} coincide with the functors of Baker-Beynon duality. Thus, the duality of \Cref{thm:duality} extends Baker-Beynon duality.

%%%%%%%%%%% %%%%%%%%%%% SECTION %%%%%%%%%%%%%%%%%%%%%% 
\section{The Zariski topology on \texorpdfstring{$\Uz^\kappa$}{U0k}}
\label{sec:first-properties}
%%%%%%%%%%%%%%%%%%%%%%%%%%%%%%%%%%%%%%%%%%%%%%%%%%%%%% 
In this section we study the topology induced  by the closure operator $\VU \CU$. We start by noticing that $\VU \CU$ is not a topological closure operator on the whole $\U^{\kappa}$. Indeed, $\VU \CU(\emptyset)= \VU(\free_\kappa)= \{ \orig\}$.  This issue can easily be  fixed by removing the origin from $\U^\kappa$. We set $\Uz^\kappa \df \U^\kappa \setminus \{ \orig \}$ and $\VUst(T)\df\VU(T) \setminus \{ \orig \}$,  for each $T \subseteq \free_\kappa$. Note that $\CU(S)=\CU(S \setminus \{\orig\})$ for any $S \subseteq \U^\kappa$ because $t(\mathbf{0})=0$ for all $t\in \free_{\kappa}$. It is an immediate consequence of \Cref{l:galois} that $\VUst$ and $\CU$ form a Galois connection between the powersets of $\Uz^\kappa$ and of $\free_\kappa$.

\begin{proposition}\label{p:VC-is-topological}
$\VUst \CU$ is a topological closure operator on $\Uz^\kappa$.
\end{proposition}
\begin{proof}
Since $\VUst$ and $\CU$ form a Galois connection, the composition $\VUst \CU$ is a closure operator on $\Uz^\kappa$. Also, $\VUst\CU(\emptyset)=\VU\CU(\emptyset) \setminus \{ \orig \}=\emptyset$.  It remains to show that if $C,D \subseteq \Uz^\kappa$, then $\VUst \CU(C \cup D) = \VUst \CU(C) \cup \VUst \CU(D)$. 
It suffices to show that, for any $J_1, J_2$ ideal of $\free_\kappa$, 
\begin{align}\label{eq:VC-top}
\VU(J_1\cap J_2)=\VU(J_1)\cup \VU(J_2).
\end{align}
The inclusion $\VU(J_1)\cup \VU(J_2)\subseteq \VU(J_1\cap J_2)$ is trivial. To prove the converse inclusion assume that $x\notin \VU(J_1)\cup \VU(J_2)$. Then there must exist $t_{1}\in J_1$ and $t_{2}\in J_2$ such that $t_{1}(x)\neq 0$ and $t_{2}(x)\neq 0$.  Also, $(t_{1}\wedge t_{2})(x)=\min \{t_{1}(x),t_{2}(x)\}\neq 0$ because $\U$ is linearly ordered. Since $t_{1}\wedge t_{2}\in J_1\cap J_2$, it follows that $x\notin \VU(J_1\cap J_2)$. 
\end{proof}

As a consequence of \Cref{p:VC-is-topological}, the subsets of $\Uz^\kappa$ of the form $\VUst(J)$, for $J$ an ideal of $\free_{\kappa}$, are the closed subsets of a topology on $\Uz^\kappa$ that we call \emph{Zariski topology}. It is readily seen that the subsets of the form $\VUst(t)$ with $t \in \free_\kappa$ form a basis of closed sets for  this topology.

\begin{remark}\label{rem:FinerTop}
Since each Riesz space is also an $\ell$-group, it makes sense to compare the Zariski topologies on $\Uz^\kappa$ given by the two different languages.  Every $\ell$-group term is also a Riesz term, therefore the topology induced by the Riesz terms is finer than the one relative to $\ell$-group terms; in fact, in \Cref{ex:finerTop} we show that in general it is strictly finer. When $\mathbb{R}$ is a subalgebra of $\U$, one can consider the topologies on $\mathbb{R}^\kappa \setminus \{ \orig \}$ as subspace of $\Uz^\kappa$. The closed subsets of $\mathbb{R}^\kappa \setminus \{ \orig \}$ are exactly the sets in the range of $\VI$ with the origin removed, i.e. the closed cones minus the origin.
So, in contrast with $\Uz^\kappa$, the subspace topologies induced on $\mathbb{R}^\kappa \setminus \{ \orig \}$ by $\ell$-group and  Riesz terms  coincide.
\end{remark}

\begin{lemma}\label{l:compactOpen=cBC}
The compact opens of $\Uz^\kappa$ are exactly the complements of the subsets of the form $\VUst(t)$ with $t \in \free_\kappa$. 
\end{lemma}

\begin{proof}
We first show that if $t \in \free_\kappa$, then $\Uz^\kappa \setminus \VUst(t)$ is compact. Suppose that $\Uz^\kappa \setminus \VUst(t)$ is contained in a union of opens, which, without loss of generality, can be assumed to be basic. Hence there exists $T \subseteq \free_\kappa$ such that
\[
\Uz^\kappa \setminus \VUst(t) \subseteq \bigcup \{ \Uz^\kappa \setminus \VUst(s) \mid s \in T  \}.
\]
By taking complements of both sides, we obtain 
\begin{align*}
    \VUst(t) &\supseteq\bigcap \{ \VUst(s) \mid s \in T \} = \VUst(T),
\end{align*}
and applying $\CU$ yields
\begin{align*}
\genid{t} &= \CU\VUst(t) \subseteq \CU \VUst (T) = \genid{T},
\end{align*}
where $\genid{T}$ denotes the ideal of $\free_\kappa$ generated by $T$.
By \cite[Proposition~2.2.3]{BKW} and \cite[p.~96]{RS1}, there exist $s_1, \ldots, s_n \in T$ such that $\lvert t \rvert \le \lvert s_1 \rvert + \dots + \lvert s_n \rvert$. Since
\begin{align}\label{eq: intersection V}
\begin{split}
\VUst(\lvert s_1 \rvert + \cdots +\lvert s_n\rvert ) &= \{ x \in \Uz^\kappa \mid  \lvert s_1(x) \rvert + \cdots +\lvert s_n(x)\rvert =0 \} \\
&=\{ x \in \Uz^\kappa \mid s_1(x)=\cdots=s_n(x)=0\}\\
&=\VUst(s_1) \cap \cdots \cap \VUst(s_n),    
\end{split}
\end{align}
we obtain $\VUst(s_1) \cap \cdots \cap \VUst(s_n) \subseteq \VUst(t)$. Thus, 
\[
\Uz^\kappa \setminus \VUst(t) \subseteq (\Uz^\kappa \setminus \VUst(s_1)) \cup \cdots \cup (\Uz^\kappa \setminus  \VUst(s_n)).
\]
This shows that $\Uz^\kappa \setminus \VUst(t)$ is compact for each $t \in \free_\kappa$. 

Suppose that $O$ is a compact open subset of $\Uz^\kappa$. Since $O$ is open, there exists $T \subseteq \free_\kappa$ such that $O= \Uz^\kappa \setminus \VUst(T)= \bigcup \{ \Uz^\kappa \setminus \VUst(t) \mid t \in T \}$. Since $O$ is compact, there exist $t_1, \ldots, t_n \in T$ such that $O= (\Uz^\kappa \setminus \V(t_1)) \cup \cdots \cup (\Uz^\kappa \setminus \VUst(t_n)) = \Uz^\kappa \setminus (\VUst(t_1) \cap \cdots \cap \VUst(t_n))$. 
Thus, $O=\Uz^\kappa \setminus \VUst(\lvert t_1 \rvert + \cdots +\lvert t_n\rvert)$ holds because $\VUst(t_1) \cap \cdots \cap \VUst(t_n)= \VUst(\lvert t_1 \rvert + \cdots +\lvert t_n\rvert )$.
\end{proof}

\begin{proposition}\label{p:compactOpen=base}
The set of compact opens of $\Uz^\kappa$ is a basis of open sets for the Zariski topology and it is closed under non-empty finite intersections.
\end{proposition}

\begin{proof}
Since $\{ \VUst(t) \mid t \in \free_\kappa \}$ is a basis of closed subsets of $\Uz^\kappa$, \Cref{l:compactOpen=cBC} implies that the compact open subsets form a basis of open subsets of the Zariski topology of $\Uz^\kappa$.

To conclude the proof it is enough to show that the basic closed sets $\VUst(t)$ are closed under binary unions and then apply \Cref{l:compactOpen=cBC}.
Let $t_1,t_2 \in \free_\kappa$, then
\begin{align*}
\VUst(t_1) \cup \VUst(t_2)&= \{ x \in \Uz^\kappa \mid t_1(x)=0 \mbox{ or } t_2(x)=0 \}\\ 
&= \{ x \in \Uz^\kappa \mid \lvert t_1(x)\rvert \wedge \lvert t_2(x)\rvert=0 \}\\
&=\VUst(\lvert t_1\rvert \wedge \lvert t_2\rvert )
\end{align*}
because $\U$ is linearly ordered. 
\end{proof}

Recall that a closed subset $C$ of a topological space is called  \emph{irreducible} if $C \subseteq D \cup E$ implies $C \subseteq D$ or $C \subseteq E$, for all closed subsets $D,E$.

\begin{proposition}\label{p:irrClosed=prime}
The operators $\VUst$ and $\CU$ restrict to order isomorphisms between the poset of irreducible closed subsets of $\Uz^\kappa$ ordered by inclusion and $\spec(\free_\kappa)$ ordered by reverse inclusion.
\end{proposition}
\begin{proof}
Since $\VUst$ and $\CU$ form a Galois connection, \cref{t:null-U}\eqref{t:null-U:item3} implies that $\VUst$ and $\CU$ restrict to an order isomorphism between the lattice of closed subsets of $\Uz^\kappa$ ordered by inclusion and the lattice of ideals of $\free_\kappa$ ordered by reverse inclusion. To conclude, it is enough to check that under this isomorphism the irreducible closed sets correspond to prime ideals. On the one hand, the irreducible closed sets are join-prime elements of the lattice of closed subsets of $\Uz^\kappa$. On the other hand, it is well known (see, e.g., \cite[Théorème~2.4.1]{BKW} and \cite[Theorem~33.2]{RS1}) that an ideal is meet-prime (with respect to inclusion) if and only if it is a prime ideal. Therefore, in the order given by reverse inclusion, $\spec(\free_\kappa)$ is the set of join-prime elements. 
\end{proof}

\begin{proposition}\label{p:sober}
Each non-empty irreducible closed subset of $\Uz^\kappa$ is the closure of a point.
\end{proposition}
\begin{proof}
Let $C$ be a non-empty irreducible closed subset of $\Uz^\kappa$. 
By \Cref{t:null-U}(\ref{t:null-U:item1}) and \Cref{p:irrClosed=prime}, there exists $x \in \Uz^\kappa$ such that $\CU(C)=\CU( x )$. Thus, $C=\VUst\CU( x )$, which means that $C$ is the closure of $x$ in $\Uz^\kappa$.
\end{proof}

Notice that the Zariski topology on $\Uz^\kappa$ is not  $T_0$. Indeed, if $t\in \free_\kappa$ and $x\in \Uz^\kappa$, since $t$ preserves multiplication by positive scalars, $t(x)=0$ implies $t(2x)=2t(x)=0$.  Hence, $x$ and $2x$ cannot be separated by an open set. Therefore, it will be convenient to consider the $T_0$-reflection of $\Uz^\kappa$. 
The $T_0$-reflection of a topological space $X$ is obtained by quotienting $X$ over the relation $\sim$ that collapses all points of $X$ that have the same closure (or, equivalently, that cannot be separated by an open set).
Since $\VUst\CU(x)$ is the closure of $\{ x \}$ for any $x \in \Uz^\kappa$, the relation $x \sim y$ holds if and only if $\VUst \CU( x ) = \VUst \CU( y )$. 

The following result is folklore. For lack of a reference, we sketch its proof.

\begin{lemma}\label{l:lattice-isomorphism}
If $X$ is a topological space, then the complete lattices of open sets of $X$ and of $X / \simo$ are isomorphic.
\end{lemma}
\begin{proof}
Let $\pi$ be the quotient map induced by $\sim$.  By definition of quotient space, a set $O$ is open in $X/\simo$ if and only if $\pi^{-1}[O]$ is open in $X$. Hence, $\pi^{-1}$ is a function from the lattice of open sets of $X/\simo$ into the lattice of open sets of $X$. It is straightforward to check that $O_1\seq O_2$ if and only if $\pi^{-1}[O_1]\seq \pi^{-1}[O_2]$, which also entails that $\pi^{-1}$ is injective. Since every bijection between lattices that preserves and reflects the order is a lattice isomorphism, it only remains to show that $\pi^{-1}$ is onto.  Let $A$ be any open subset of $X$, we prove that $A=\pi^{-1}[\pi[A]]$ --- i.e. $A$ is saturated. We show the non-trivial inclusion. If $x\in\pi^{-1}[\pi[A]]$, then $x\sim y$ for some $y\in A$. Hence, by the definition of $\sim$, there exists no open set that separates $x$ from $y$, and hence $x\in A$.
\end{proof}

A topological space is called \emph{generalized spectral} if it is $T_0$, sober, and the compact open subsets are closed under binary intersections and form a basis of the space.

\begin{theorem}\label{t:Uzkappa generalized spectral}
$\Uz ^\kappa/{\sim}$ is a generalized spectral space.
\end{theorem}
\begin{proof}
The space $\Uz^\kappa/{\sim}$ is $T_0$ by construction. By  \Cref{l:lattice-isomorphism},  the lattices of open subsets of $\Uz^\kappa$ and $\Uz^\kappa/\simo$ are isomorphic. 
This isomorphism maps compact opens of $\Uz^\kappa$ to compact opens of $\Uz^\kappa/\simo$ because an open subset of a topological space is compact if and only if it is a compact element of the lattice of open subsets.
\Cref{l:compactOpen=cBC,p:compactOpen=base} imply that the compact opens of $\Uz^\kappa/\simo$ are exactly the complements of $\VUst(t)/\simo$ for some $t \in \free_\kappa$. Therefore, the compact opens of $\Uz^\kappa/\simo$ are closed under binary intersections and form a basis.

Finally, let $C$ be a closed subset of $\Uz^\kappa/{\sim}$ and $O$ its complement. The set $C$ is irreducible if and only if
$O$ is meet-prime in the lattice of open subsets of $\Uz^\kappa$. 
Since isomorphisms of lattices preserve join-prime elements,
$\pi^{-1}[C]=\Uz^\kappa \setminus \pi^{-1}[O]$ is irreducible if and only if so is $C$.
\end{proof}

\begin{definition}
The spectrum of prime ideals of an $\ell$-group $G$ is classically considered with its hull-kernel topology: the closed subsets of  $\spec(G)$ are the sets $\mathtt{h}(J)=\{ P \in \spec(G) \mid J \subseteq P \}$ where $J$ is an ideal of $G$. It is straightforward to see that the closed sets $\mathtt{h}(a) = \{ P \in \spec(G) \mid a \in P \}$ for $a \in G$ form a basis of the topology. Similarly, one defines the hull-kernel topology on the prime spectrum of any Riesz space.
\end{definition}

\begin{theorem}\label{t:spec-closed}
$\Uz^\kappa/{\sim}$ is homeomorphic to $\spec(\free_\kappa)$.
\end{theorem}
\begin{proof}
Define a map $\UU\colon\Uz^\kappa \to \spec(\free_\kappa)$ that sends $x \in \Uz^\kappa$ to $\CU(x)$, which is a prime ideal by \Cref{t:null-U}(\ref{t:null-U:item1}). Since $\CU$ and $\VUst$ form a Galois connection, the definition of the Zariski topology on $\Uz^\kappa$ yields
\[
\UU(x)=\UU(y) \Longleftrightarrow \CU(x)=\CU(y)  \Longleftrightarrow \VUst\CU(x)=\VUst\CU(y)  \Longleftrightarrow  x \sim y.
\]
By  \Cref{t:null-U}\eqref{t:null-U:item1}, $\UU$ is onto.
Therefore, $\UU$ induces a bijection $\UU'\colon\Uz^\kappa/{\sim} \to \spec(\free_\kappa)$. To prove that $\UU'$ is continuous, let $t \in \free_\kappa$. Observe that 
\begin{align*}
(\UU')^{-1}[\mathtt{h}(t)] &= \{ [x]_{\simo} \in \Uz^\kappa/\simo \mid \UU(x) \in \mathtt{h}(t) \}=\{ [x]_{\simo} \in \Uz^\kappa/\simo \mid \CU(x) \in \mathtt{h}(t) \} \\
& = \{ [x]_{\simo} \in \Uz^\kappa/\simo \mid t \in \CU(x) \} = \{ [x]_{\simo} \in \Uz^\kappa/\simo \mid x \in \VUst(t) \}=\VUst(t)/\simo,
\end{align*}
and hence $\UU'$ is continuous.
Since $\UU'$ is a bijection, $\UU'(\VUst(t)/{\sim})=\mathtt{h}(t)$, and hence $\UU'$ is a closed map. Therefore, $\UU'$ is a homeomorphism.
\end{proof}

Let $\mathscr{E}\colon \spec(\free_{\kappa}) \to \Uz^{\kappa}$ be a function defined by choosing for each $P \in \spec(\free_{\kappa})$ a point $\mathscr{E}(P) \in \Uz^\kappa$ such that $P=\CU(\mathscr{E}(P))$.
It follows immediately from the definitions of $\mathscr{E}$ and $\UU$ that the composition ${\UU} \circ \mathscr{E}$ is the identity function on $\spec(\free_{\kappa})$.

\begin{proposition}\label{p:E embedding image dense}
$\mathscr{E}\colon \spec(\free_{\kappa}) \to \Uz^{\kappa}$ is a topological embedding whose image is a dense subset of $\Uz^{\kappa}$.
\end{proposition}
\begin{proof}
Let $\UU'\colon\Uz^\kappa/{\sim} \to \spec(\free_\kappa)$ be the homeomorphism from the proof of \Cref{t:spec-closed}. Since $\UU'$ is a homeomorphism, to prove the claim it is sufficient to show that $\mathscr{E} \circ \UU' \colon \Uz^\kappa/{\sim} \to \Uz^{\kappa}$ is an embedding whose image is a dense subset of $\Uz^{\kappa}$. We have that $\mathscr{E} \UU'([x]_{\simo}) = \mathscr{E} ( \CU(x))$ is a point $y \in \Uz^\kappa$ such that $\CU(y)=\CU(x)$, and hence $y \in [x]_{\simo}$. Thus, $\mathscr{E} \circ \UU'$ maps each equivalence class of $\sim$ to one of its representatives. Since $\Uz^\kappa/{\sim}$ is the $T_0$-reflection of $\Uz^\kappa$, it is then straightforward to see that $\mathscr{E} \circ \UU'$ is an embedding whose image is a dense subset of $\Uz^{\kappa}$.
\end{proof}

If $G\simeq\free_\kappa/J$, by the Correspondence Theorem \cite[Theorem 6.20]{BS} there is a bijection between $\mathtt{h}(J)=\{ P\in \spec(\free_{\kappa}) \mid J\subseteq P\}$ and $\spec(G)$ that can be easily proved to be a homeomorphism. With an abuse of notation, we identify $\spec(G)$ with $\mathtt{h}(J)$.

\begin{theorem}\label{c:spec-and-V}
Let $G\simeq \free_{\kappa}/J$ be an $\ell$-group or a Riesz space.
\begin{enumerate}
\item \label{pro:spec} $\VUst(J)/{\sim}$ is homeomorphic to $\spec(G)$. Moreover, $\spec(G)$ embeds into $\VUst(J)$ as a dense subset.
\item \label{pro:retract} $\spec(G)$ is a retract of $\VUst(J)$.
\item \label{c:PWL-spec} $G$ is isomorphic to the algebra of definable maps restricted to $\mathscr{E}[\spec(G)]$ with operations defined pointwise.
\end{enumerate}
\end{theorem}
\begin{proof}
\Cref{pro:spec}. Let us prove that we can restrict and corestrict the functions $\UU\colon \Uz^{\kappa}\to \spec(\free_{\kappa})$ and $\mathscr{E}\colon \spec(\free_{\kappa})\to \Uz^{\kappa}$ to the sets $\VUst(J)$ and $\spec(G)$.  Indeed, on the one hand, by the properties Galois connections it follows that
\[
x\in \VUst(J) \Longleftrightarrow J\subseteq \CU(x) \Longleftrightarrow J\subseteq \UU(x) \Longleftrightarrow \UU(x)\in \spec(G).
\]
On the other hand, let $\mathscr{E}(P)=x$, then the definition of $\mathscr{E}$ implies $P=\CU(x)$. Hence,
\[
P\in \spec(G) \Longleftrightarrow J\subseteq P \Longleftrightarrow \VUst(P)\subseteq \VUst(J) \Longleftrightarrow \VUst\CU( x)\subseteq \VUst(J) \Longleftrightarrow x\in \VUst(J),
\]
where the latter equivalence follows from the fact that $\VUst\CU( x)$ is the closure of $x$.
By arguing as in \Cref{t:spec-closed,p:E embedding image dense}, $\UU$ induces an homeomorphism between $\VUst(J)/\simo$ and $\spec(G)$, and $\mathscr{E}$ restricts to an embedding $\mathscr{E}\colon \spec(G)\to \VUst(J)$ whose image is dense in $\VUst(J)$. 

\Cref{pro:retract}. The composition $\mathscr{U} \circ \mathscr{E}$ is the identity on $\spec(G)$. Indeed, $\mathscr{U}( \mathscr{E}(P))=\CU(\mathscr{E}(P))=P$ by definition of $\mathscr{E}$. 

\Cref{c:PWL-spec}. By \Cref{l:PWL-I(C)}, it is enough to prove that $J= \CU(\mathscr{E}[\spec(G)])$. Since $\mathscr{E}[\spec(G)]$ is dense in $\VUst(J)$,
\[
\CU(\mathscr{E}[\spec(G)])=\CU \VUst \CU (\mathscr{E}[\spec(G)])= \CU\VUst(J)=J.\qedhere
\]
\end{proof}

Since $\mathscr{E}$ embeds $\spec(G)$ into $\Uz^\kappa$, we call it the \emph{coordinatization} of the spectrum. 
\Cref{c:spec-and-V}\eqref{c:PWL-spec} shows that it is possible to reconstruct, up to isomorphism, the whole $G$ using only $\spec(G)$ and the coordinatization. Thus, the duality of \Cref{thm:duality} can be thought of as being induced by $\spec$.
We remark that the coordinatization of $\spec(G)$ via $\mathscr{E}$ is not unique, as it depends on the choices made in the definition of $\mathscr{E}$. In \Cref{rem:coordinatization,rem:coordinatization-lgroups} 
we will describe, in the finitely generated case, a canonical way to define the embedding $\mathscr{E}$.

\section{Order-units, semisimplicity, and Archimedeanity}\label{sec:archim}

In this section we dually characterize the existence of order-units, the intersection of all maximal ideals, and the ideal of nilpotents of an $\ell$-group or Riesz space $G$. A dual characterization of semisimplicity and Archimedeanity will follow.

Recall that an element $0<a\in G$ is called a \emph{strong order-unit} if for each $b \in G$ there exists $n \in \mathbb{N}$ such that $b \le na$. We say that $a>0$ is a \emph{weak order-unit} of $G$ if $a \wedge \lvert b \rvert=0$ implies $b=0$ for each $b \in G$.

\begin{remark}\label{rem:facts units}
A strong order-unit is always a weak order-unit (see, e.g.,~\cite[Lemma XIII.11.4]{Bir79}). It is straightforward to see that $a$ is a strong order-unit of $G$ if and only if the ideal generated by $a$ is the whole $G$.  Moreover, if $G$ is finitely generated, then it has a strong order-unit. Indeed, if $a_1, \ldots , a_n$ generate $G$, then $\lvert a_1\rvert + \dots + \lvert a_n\rvert$ is a strong order-unit.
\end{remark}

It is well known that $G$ has a strong order-unit if and only if $\spec(G)$ is compact (see, e.g., \cite[Proposition~10.1.6]{BKW} and \cite[Theorem~37.1(i)]{RS1}). 
In the following theorem we show that analogous results for strong and weak order-units hold in our setting. Recall from \Cref{sec:first-properties} that $\Uz^\kappa$ is a topological space with the Zariski topology. Thus, if $J$ is an ideal of $\free_\kappa$, then $\VUst(J)$ can be equipped with the subspace topology inherited from $\Uz^\kappa$.

\begin{theorem}\label{t:duals-of-units}
Let $G=\free_{\kappa}/J$ be non-trivial and $t \in \free_\kappa$ such that $t/J > 0$ in $G$.
\begin{enumerate}
\item\label{t:duals-of-units:item1} $t/J$ is a strong order-unit of $G$ if and only if $\VUst(J)\cap\VUst(t)=\emptyset$.
\item\label{t:duals-of-units:item2} $G$ has a strong order-unit if and only if $\VUst(J)$ is compact.
\item\label{t:duals-of-units:item3} $t/J$ is a weak order-unit of $G$ if and only if $\VUst(J) \cap \VUst(t)$ has empty interior in $\VUst(J)$.
\item\label{t:duals-of-units:item4} $G$ has a weak order-unit if and only if $\VUst(J)$ contains a dense compact open subset.
\end{enumerate}
\end{theorem}

\begin{proof}
\Cref{t:duals-of-units:item1}. Let $I/J$ be the ideal of $G$ generated by $t/J$. Thus, $I$ is the ideal $\genid{\{ t \} \cup J}$ of $\free_\kappa$ generated by $\{ t \} \cup J$, and hence \Cref{t:null-U}(\ref{t:null-U:item2}) yields
\[
\VUst(I)=\VUst(\genid{\{ t \} \cup J})=\VUst(\{ t \} \cup J)=\VUst(J) \cap \VUst(t). 
\]
As mentioned above, $t/J$ is a strong order-unit of $G$ if and only if $I/J=G$, which holds if and only if $I=\free_\kappa$. Since $I=\free_\kappa$ is equivalent to $\VUst(I)=\emptyset$, it follows that $t/J$ is a strong order-unit if and only if $\VUst(J) \cap \VUst(t) = \emptyset$.

\Cref{t:duals-of-units:item2}. Suppose that $s/J$ is a strong order-unit of $G$. By \Cref{t:duals-of-units:item1}, $\VUst(J) \subseteq \Uz^\kappa \setminus \VUst(s)$, and the latter is compact by \Cref{l:compactOpen=cBC}. Thus, $\VUst(J)$ is compact because it is closed in a compact subset. Vice versa, suppose that $\VUst(J)$ is compact. Since 
\[
\emptyset = \VUst(\free_\kappa) \cap \VUst(J)= \bigcap \{ \VUst(s) \mid s \in \free_\kappa\}\cap \VUst(J)=  \bigcap \{ \VUst(s)\cap \VUst(J) \mid s \in \free_\kappa\},
\] 
there are $s_1, \ldots, s_n \in \free_\kappa$ such that $\VUst(J) \cap \VUst(s_1) \cap \dots \cap \VUst(s_n) = \emptyset$. Hence,
$\VUst(J) \cap \VUst(\lvert s_1 \rvert + \dots + \lvert s_n \rvert) = \emptyset$ by \cref{eq: intersection V}. It follows that $(\lvert s_1 \rvert + \dots + \lvert s_n \rvert)/J > 0$. Indeed, if $\lvert s_1 \rvert + \dots + \lvert s_n \rvert \in J$, then $\VUst(J)=\emptyset$, and so $J=\free_\kappa$. Since $G$ is non-trivial, this is impossible. Therefore, $(\lvert s_1 \rvert + \dots + \lvert s_n \rvert)/J$ is a strong order-unit of $G$ by \Cref{t:duals-of-units:item1}.

\Cref{t:duals-of-units:item3}. By definition, $t/J$ is a weak order-unit of $G$ if and only if 
\begin{equation}\label{t:duals-of-units:item3:eq1}
t \wedge \lvert s \rvert \in J\text{ implies }s \in J, \text{ for each }s \in \free_\kappa.
\end{equation}
Since $\VUst(t \wedge \lvert s \rvert) = \VUst(t) \cup \VUst(\lvert s \rvert) = \VUst(t) \cup \VUst(s)$, condition \eqref{t:duals-of-units:item3:eq1} is equivalent to
\begin{equation}\label{t:duals-of-units:item3:eq2}
\VUst(J)  \subseteq \VUst(t) \cup \VUst(s)\text{ implies }\VUst(J)  \subseteq \VUst(s), \text{ for each }s \in \free_\kappa.
\end{equation}
In turn, condition \eqref{t:duals-of-units:item3:eq2} is equivalent to 
\begin{equation} \label{t:duals-of-units:item3:eq3}
\VUst(J)  \setminus \VUst(s) \subseteq \VUst(J)\cap \VUst(t)  \text{ implies }\VUst(J)  \setminus \VUst(s) = \emptyset, \text{ for each }s \in \free_\kappa.
\end{equation}
By \Cref{l:compactOpen=cBC} and \Cref{p:compactOpen=base}, the family of subsets of the form $\VUst(J)  \setminus \VUst(s)$ is a basis of open subsets of $\VUst(J)$. Thus, \eqref{t:duals-of-units:item3:eq3} amounts to saying that $ \VUst(t)\cap \VUst(J)$ does not contain any non-empty open subset of $\VUst(J)$. We conclude that $t/J$ is a weak order-unit if and only if $ \VUst(t)\cap \VUst(J)$ has empty interior in $\VUst(J) $.

\Cref{t:duals-of-units:item4}. By \Cref{t:duals-of-units:item3}, $G$ has a weak order-unit if and only if there exists $s/J > 0$ such that $\VUst(J) \setminus \VUst(s)$ is dense in $\VUst(J)$.
It is then sufficient to prove that the non-empty compact open subsets of $\VUst(J)$ are exactly the ones of the form $\VUst(J) \setminus \VUst(s)$ for some $s \in \free_\kappa$ such that $s/J>0$. On the one hand, let $U = \VUst(J) \setminus \VUst(s)$ with $s/J>0$. Since $s/J \neq 0$, it follows that $\VUst(J) \nsubseteq \VUst(s)$, and so $U$ is not empty. We have that $U=(\Uz^\kappa \setminus \VUst(s))\cap \VUst(J)$, and so it is a closed subset of $\Uz^\kappa \setminus \VUst(s)$, which is compact by \Cref{l:compactOpen=cBC}. We conclude that $U$ is compact. Moreover, $U$ is open in $\VUst(J)$ because $\Uz^\kappa \setminus \VUst(s)$ is open in $\Uz^\kappa$. On the other hand, suppose $U$ is a non-empty compact open subset of $\VUst(J)$. By definition of the Zariski topology, there is $S \subseteq \free_k$ such that $U=(\Uz^\kappa \setminus \VUst(S))\cap \VUst(J)$. We have that $U = \bigcup \{ (\Uz^\kappa \setminus \VUst(s)) \cap \VUst(J) \mid s \in S \}$ and each $(\Uz^\kappa \setminus \VUst(s)) \cap \VUst(J)$ is open in $U$. Since $U$ is compact, there exist $s_1, \dots, s_n \in \free_\kappa$ such that 
\[
U =\big( \Uz^\kappa \setminus \VUst(s_1)\cup\dots\cup \Uz^\kappa \setminus \VUst(s_n)\big) \cap \VUst(J)=\VUst(J) \setminus (\VUst(s_1) \cap \dots \cap \VUst(s_n)).
\]
Since $\VUst(s_1) \cap \dots \cap \VUst(s_n) = \VUst(\lvert s_1 \rvert + \dots + \lvert s_n \rvert)$, we obtain 
\[
U = \VUst(J) \setminus \VUst(\lvert s_1 \rvert + \dots + \lvert s_n \rvert),
\] 
with $(\lvert s_1 \rvert + \dots + \lvert s_n \rvert)/J > 0$ because $U$ is not empty.
\end{proof}

The next corollary is a direct consequence of \Cref{t:duals-of-units}(\ref{t:duals-of-units:item2}) because any finitely generated $\ell$-group or Riesz space has a strong order-unit as observed in \Cref{rem:facts units}.

\begin{corollary}\label{c:fin-dim-compact}
If $n$ is a positive integer, then $\Uz^n$ is compact.
\end{corollary}

It is well known that $G$ is semisimple if and only if the intersection of all the maximal ideals of $G$ is the trivial ideal. We now give a dual characterization of the intersection of all maximal ideals, which yields a dual characterization of the $\ell$-groups and Riesz spaces that are semisimple.

\begin{proposition}\label{p:intersection of max ideals dual}
Let $G=\free_{\kappa}/J$ and $I/J$ be the intersection of all maximal ideals of $G$.
\begin{enumerate}
\item\label{p:intersection of max ideals dual:item1} $I=\CU(\VU(J) \cap \mathbb{R}^\kappa)$.
\item\label{rem:dual intersection maximals} $G$ is semisimple if and only if $\VUst(J)$ is the closure of $\VUst(J) \cap \mathbb{R}^\kappa$ in $\Uz^\kappa$.
\end{enumerate}
\end{proposition}

\begin{proof}
\Cref{p:intersection of max ideals dual:item1}.
It follows from Lemmas \ref{l:facts-C-V}\eqref{l:facts-C-V:item5} and \ref{p:null-U-max}\eqref{p:null-U-max:item2}  that 
\begin{align*}
I &= \bigcap \{ M \mid M \text{ is maximal and } J \subseteq M \} = \CU\VI(J) = \CU(\VU(J) \cap \mathbb{R}^\kappa).
\end{align*}

\Cref{rem:dual intersection maximals}. We have that $G$ is semisimple if and only if $J=I$, which is equivalent to $\VUst(J)=\VUst(I)$. Since $\CU(S)=\CU(S \setminus \{ \orig \})$ for any $S \subseteq \U^\kappa$, \Cref{p:intersection of max ideals dual:item1} implies that $I=\CU(\VUst(J) \cap \mathbb{R}^\kappa)$. Thus, $\VUst(I)=\VUst\CU(\VUst(J) \cap \mathbb{R}^\kappa)$, which is the closure of $\VUst(J) \cap \mathbb{R}^\kappa$ in $\Uz^\kappa$. Therefore, $G$ is semisimple if and only if $\VUst(J)$ is the closure of $\VUst(J) \cap \mathbb{R}^\kappa$.
\end{proof}

We say that $a \in G$ is \emph{nilpotent} if there exists $b \in G$ such that $n\lvert a \rvert \le b$ for each $n \in \N$.
It is well known that if $G$ has a strong order-unit $u$, then $a$ is nilpotent if and only if $n\lvert a \rvert \le u$ for each $n \in \N$.

\begin{remark}\label{rem:facts nilpotents}
A straightforward calculation yields that the set of nilpotents of $G$ form an ideal. This ideal is always contained in the intersection of all maximal ideals of $G$. The ideal of nilpotents and the intersection of all maximal ideals coincide when $G$ has a strong order-unit \cite[Theorem~1]{YF}. An example of a Riesz space in which the two ideals differ can be found in \cite[Section 3]{YF}. Clearly, any such example is necessarily not finitely generated by \Cref{rem:facts units}.
An $\ell$-group or Riesz space is Archimedean exactly when the ideal of nilpotents is the trivial ideal. Thus, semisimplicity always implies Archimedeanity and, in the presence of a strong order-unit, the two notions are equivalent. 
\end{remark}

\Cref{t:dual-of-N} will provide a dual description of the ideal of nilpotents of $G$. 
To state the theorem we first need to introduce some notation.
Since $\kappa$ is a cardinal, the elements of $\kappa$ are exactly the ordinals smaller than $\kappa$. Let $S= \{ \gamma_1, \dots, \gamma_n \}$ be a finite subset of $\kappa$ with $\gamma_1< \dots< \gamma_n$. We denote by $\varphi_S \colon \free_n \to \free_\kappa$
the unique map that sends the $i^{\text{th}}$ generator of $\free_n$ to the $\gamma_i^{\text{th}}$ generator of $\free_\kappa$ for each $i=1, \ldots, n$.
Moreover, let $\pi_S \colon \U^\kappa \to \U^n$ be the projection that maps $(z_\gamma)_{\gamma < \kappa}$ to $(z_{\gamma_1}, \ldots, z_{\gamma_n})$. It is straightforward to see that $\varphi_S$ is the injective homomorphism dual to $\pi_S$ under the duality of \Cref{thm:duality}. We have the following technical lemma.

\begin{lemma}\label{l:IcapF-and-NcapF}
Let $J$ be an ideal of $\free_{\kappa}\free_{\kappa}$ and $C=\VU(J)$. The following proerties hold.
\begin{enumerate}
\item\label{l:IcapF-and-NcapF:item1} $\varphi_S^{-1}[J] = \CU(\pi_S[C])$.
\item\label{l:IcapF-and-NcapF:item2} If $H \subseteq \free_n$, then $\V(\varphi_S[H]) = \pi_S^{-1}[\VU(H)]$.
\end{enumerate}
Moreover, if $N_S/\varphi_S^{-1}[J]$ is the ideal of nilpotents of $\free_n/\varphi_S^{-1}[J]$, then 
\begin{enumerate}[resume]
\item\label{l:IcapF-and-NcapF:item3i} $N_S = \CU(\VU\CU(\pi_S[C]) \cap \mathbb{R}^n)$.
\item\label{l:IcapF-and-NcapF:item3iii} $\VU(\varphi_S[N_S]) = \pi_S^{-1}[\VU\CU(\VU\CU(\pi_S[C]) \cap \mathbb{R}^n)]$.
\item\label{l:IcapF-and-NcapF:item3ii} If $t \in \varphi_S[\free_n]$, then $t \in \varphi_S[N_S]$ if and only if there exists $u \in \varphi_S[\free_n]$  such that $m \lvert t/J \rvert \le u/J$ for each $m \in \N$.
\end{enumerate}
\end{lemma}

\begin{proof}
\Cref{l:IcapF-and-NcapF:item1}.
Since $\varphi_S$ is a homomorphism and $J$ is an ideal, $\varphi_S^{-1}[J]$ is an ideal of $\free_n$. It follows from the definitions of $\varphi_S$ and $\pi_S$ that $\varphi_S (t)(x)=t(\pi_S(x))$ for all $t \in \free_n$ and $x\in \U^\kappa$. Thus,
\begin{align*}
t \in \varphi_S^{-1}[J] & \Longleftrightarrow \varphi_S(t) \in J = \CU(C) \Longleftrightarrow \varphi_S(t)(x)=0 \mbox{ for all } x \in C\\ 
&\Longleftrightarrow t(\pi_S(x))=0 \mbox{ for all } x \in C \Longleftrightarrow t \in \CU(\pi_S[C]).
\end{align*}
Therefore, $\varphi_S^{-1}[J] = \CU(\pi_S[C])$.

\Cref{l:IcapF-and-NcapF:item2}. Since $\varphi_S (t)(x)=t(\pi_S(x))$ for all $t \in \free_n$ and $x\in \U^\kappa$, we have that $\VU(\varphi_S(t))=\pi_S^{-1}[\VU(t)]$. Therefore,
\begin{align*}
\VU(\varphi_S[H]) & = \bigcap \{\VU(\varphi_S(t)) \mid t \in H \} = \bigcap \{\pi_S^{-1}[\VU(t)] \mid t \in H \}\\
& = \pi_S^{-1} \left[ \bigcap \{\VU(t) \mid t \in H \} \right] = \pi_S^{-1}[\VU(H)].
\end{align*}

\Cref{l:IcapF-and-NcapF:item3i}. $\free_n/\varphi_S^{-1}[J]$ has a strong order-unit because it is finitely generated. Thus, 

\Cref{rem:facts nilpotents,p:intersection of max ideals dual,l:IcapF-and-NcapF:item1} imply that
\begin{align*}
N_S = \CU (\VU(\varphi_S^{-1}[J]) \cap \mathbb{R}^n) = \CU (\VU\CU(\pi_S[C]) \cap \mathbb{R}^n).
\end{align*}

\Cref{l:IcapF-and-NcapF:item3iii}. It follows from \Cref{l:IcapF-and-NcapF:item2,l:IcapF-and-NcapF:item3i} that
\[
\VU(\varphi_S[N_S]) = \pi_S^{-1}[\VU(N_S)] = \pi_S^{-1}[\VU\CU (\VU\CU(\pi_S[C]) \cap \mathbb{R}^n)].
\]

\Cref{l:IcapF-and-NcapF:item3ii}. Let $F = \varphi_S[\free_n]$.
Since the co-restriction $\varphi_S \colon \free_n \to F$ is an isomorphism, it induces an isomorphism between $\free_n/\varphi_S^{-1}[J]$ and $F /(F \cap J)$. It follows that $t \in \varphi_S[N_S]$ if and only if $t/(F \cap J)$ is nilpotent in $F/(F \cap J)$. Thus, $t \in \varphi_S[N_S]$ if and only if there exists $u \in F$ such that $m \lvert t/(F \cap J) \rvert \le u/(F \cap J)$ for each $m \in \mathbb{N}$, which is equivalent to $(m \lvert t \rvert - u ) \vee 0 \in F \cap J$. Since $t,u \in F$ and $F$ is a subalgebra of $\free_\kappa$, the last condition is equivalent to $(m \lvert t \rvert - u ) \vee 0 \in J$ for each $m \in \mathbb{N}$. Therefore, $t \in \varphi_S[N_S]$ if and only if there exists $u \in F=\varphi_S[\free_n]$ such that $m \lvert t/J \rvert \le u/J$ for each $m \in \mathbb{N}$. 
\end{proof}

The following theorem shows that the ideal of nilpotents of $G$ corresponds to the intersection of a family of subsets of $\U^\kappa$ indexed by the set of all finite subsets of ordinals smaller than $\kappa$. Let $\pfink$ be the set of all finite subsets of $\kappa$. If $S \in \pfink$, we denote the cardinality of $S$ by $\lvert S \rvert$.

\begin{theorem}\label{t:dual-of-N}
Let $G=\free_{\kappa}/J$ and $C=\VU(J)$. If $N/J$ is the ideal of nilpotents of $G$, then
\[
\VU(N) = \bigcap_{S \in \pfink} \pi_S^{-1} [\VU\CU(\VU\CU(\pi_S[C]) \cap \mathbb{R}^{\lvert S \rvert})].
\]
\end{theorem}

\begin{proof}
Let $S \in \pfink$ and $N_S/\varphi_S^{-1}[J]$ be the ideal of nilpotents of $\free_n/\varphi_S^{-1}[J]$. By \Cref{l:IcapF-and-NcapF}(\ref{l:IcapF-and-NcapF:item3ii}), $\varphi_S[N_S]$ is the set of all the $t \in \varphi_S[\free_n]$ for which there exists $u \in \varphi_S[\free_n]$ such that $m \lvert t/J \rvert \le u/J$ for each $m \in \N$. It follows that $t/J$ is a nilpotent of $G$ for any $t \in \varphi_S[N_S]$, and hence $\varphi_S[N_S] \subseteq N$. Conversely, let $t \in N$. Then $t \in \free_\kappa$ and $t/J$ is a nilpotent of $G$, which means that there exists $u \in \free_\kappa$ such that $m \lvert t/J \rvert \le u/J$ for each $m \in \N$. Since every term contains a finite number of variables and $\varphi_S[\free_n]$ is the subset of $\free_\kappa$ consisting of all the equivalence classes of terms whose variables are indexed by elements of $S$, there exists $S \in \pfink$ such that $t,u \in \varphi_S[\free_n]$. It follows that $t \in \varphi_S[N_S]$.
Therefore,
$N = \bigcup \{ \varphi_S[N_S] \mid S \in \pfink \}$. Then \Cref{l:IcapF-and-NcapF}(\ref{l:IcapF-and-NcapF:item3iii}) imples that
\begin{align*}
\VU(N)  = \bigcap_{S \in \pfink} \VU(\varphi_S[N_S]) = \bigcap_{S \in \pfink} \pi_S^{-1} [\VU\CU(\VU\CU(\pi_S[C]) \cap \mathbb{R}^{\lvert S \rvert})].
\end{align*}
\end{proof}

\begin{corollary}\label{c:inclusion-subsets-C}
Let $C \subseteq \U^\kappa$ such that $C=\VU \CU(C)$. Then
\[
\VU \CU(C \cap \mathbb{R}^\kappa) 
\subseteq 
\bigcap_{S \in \pfink} \pi_S^{-1} [\VU\CU(\VU\CU(\pi_S[C]) \cap \mathbb{R}^{\lvert S \rvert})].
\]
If in addition $C\setminus \{ \orig \}$ is a compact subspace of $\Uz^\kappa$, then the inclusion above is an equality.
\end{corollary}

\begin{proof}
Let $J=\CU(C)$. Then $J$ is an ideal of $\free_\kappa$ and $C=\VU(J)$. 
Let $N/J$ be the ideal of nilpotents of $\free_\kappa/J$ and $I/J$ the intersection of all maximal ideals of $\free_\kappa/J$.
By \Cref{rem:facts nilpotents}, $N/J \subseteq I/J$, and so $N \subseteq I$. Thus, \Cref{p:intersection of max ideals dual}\eqref{rem:dual intersection maximals} and \Cref{t:dual-of-N} imply that
\begin{align*}
\VU \CU(C \cap \mathbb{R}^\kappa) = \VU(I) \subseteq \VU(N) = \bigcap_{S \in \pfink} \pi_S^{-1} [\VU\CU(\VU\CU(\pi_S[C]) \cap \mathbb{R}^{\lvert S \rvert})].
\end{align*}
If $C \setminus \{ \orig \} = \VUst(J)$ is compact, then $\free_\kappa/J$ has a strong order-unit by \Cref{t:duals-of-units}(\ref{t:duals-of-units:item2}). Therefore, $I=N$ by \Cref{rem:facts nilpotents}, and the inclusion becomes an equality. 
\end{proof}

\begin{remark}\label{rem:directed intersection}
\mbox{}\begin{enumerate}
\item\label{rem:directed intersection:item1} 
If $S,S' \in \pfink$ with $S \subseteq S'$, then
\[
\pi_{S'}^{-1} [\VU\CU(\VU\CU(\pi_{S'}[C]) \cap \mathbb{R}^{\lvert S' \rvert})]
\subseteq
\pi_S^{-1} [\VU\CU(\VU\CU(\pi_S[C]) \cap \mathbb{R}^{\lvert S \rvert})].
\]
Indeed, \cref{l:IcapF-and-NcapF}\eqref{l:IcapF-and-NcapF:item3ii} implies $\varphi_S[N_S] \subseteq \varphi_{S'}[N_{S'}]$, and hence $\VU(\varphi_{S'}[N_{S'}]) \subseteq \VU(\varphi_S[N_S])$. It is then sufficient to apply \cref{l:IcapF-and-NcapF}\eqref{l:IcapF-and-NcapF:item3iii}.
As a consequence, the sets
\[
\pi_S^{-1} [\VU\CU(\VU\CU(\pi_S[C]) \cap \mathbb{R}^{\lvert S \rvert})],
\]
where $S \in \pfink$, form a downward directed family of subsets of $\U^\kappa$.

\item\label{rem:directed intersection:item2} If $\kappa=\omega$, then it follows from \cref{rem:directed intersection:item1} that we can replace the equality of \cref{t:dual-of-N} with
\[
\VU(N) = \bigcap_{n=0}^\infty \pi_{\{0, \ldots, n\}}^{-1} [\VU\CU(\VU\CU(\pi_{\{0, \ldots, n\}}[C]) \cap \mathbb{R}^{n+1})],
\]
where the subsets $\pi_{\{0, \ldots, n\}}^{-1} [\VU\CU(\VU\CU(\pi_{\{0, \ldots, n\}}[C]) \cap \mathbb{R}^{n+1})]$ of $\U^\omega$ form a decreasing sequence of subsets of $\U^\omega$.
\end{enumerate}
\end{remark}

In the appendix we will describe a set $C \subseteq \U^\omega$ for which the inclusion of \Cref{c:inclusion-subsets-C} is strict. This will yield a presentation of an $\ell$-group (or Riesz space) $\free_\omega/\CU(C)$ that is $\omega$-generated and whose ideal of nilpotents does not coincide with the intersection of its maximal ideals. The construction of $C$ will require some basic techniques of non-standard analysis that will be introduced in \Cref{sec:irreducible}.

%%%%%%%%%%% %%%%%%%%%%% SECTION %%%%%%%%%%%%%%%%%%%%%% 
\section{Irreducible closed subsets of \texorpdfstring{$\Uz^{n}$}{U0n}}
\label{sec:irreducible}
%%%%%%%%%%%%%%%%%%%%%%%%%%%%%%%%%%%%%%%%%%%%%%%%%% 

By \Cref{p:irrClosed=prime}, prime ideals of $\free_\kappa$ correspond to irreducible closed subsets of $\Uz^\kappa$.
In this section we use tools from non-standard analysis to give a geometrical description of the irreducible closed subsets of $\Uz^{n}$ when $\U$ is an ultrapower of $\mathbb{R}$ and $n \in \N$.  These tools are available only in finite-dimensional settings, therefore in this section we only work in finite dimension. We briefly recall below the basic definitions and results in non-standard analysis needed in our proofs; see, e.g., \cite{Goldblatt,rob} for more details. 

Throughout this section, $\U$ will be an ultrapower of the form $\R^{\mathbb{N}}/\mathcal{F}$, where $\mathcal{F}$ is a non-principal ultrafilter on $\mathbb{N}$. By \Cref{cor:existence-U-generators}\eqref{cor:existence-U-generators:item3}, such an ultrapower satisfies \Cref{ass:U-embeds} with $\alpha = \omega$.

The elements of $\U$, called \emph{hyperreal numbers}, are equivalence classes of $\N$-indexed sequences of real numbers $[(r_{i})_{i\in \N}]$.
There is a canonical elementary embedding of $\R$ into $\U$ defined by sending $r \in \R$ to $[(r_i)] \in \U$, with $r_i=r$ for all $i\in \N$.
This embedding allows us to identify $\R$ with its isomorphic copy inside $\U$. In order to recall the transfer principle, which will be one of our main tools, we need to introduce the notion of enlargement (see, e.g., \cite[Sections~3.9, 3.11, and 3.14]{Goldblatt}).

\begin{definition}\label{d:enlargement}
If $P\subseteq \R^n$, its \emph{enlargement} $\en{P}\seq\U^n$ is defined as follows.  For any $(r_i^1), \ldots ,(r_i^n)\in\R^\mathbb{N}$, 
\[
\left( [(r_i^1)], \ldots, [(r_i^n)] \right)\in \en{P} \text{ if and only if } \{i\in \mathbb{N}\mid (r_i^1, \ldots, r_i^n) \in P\}\in \mathcal{F}.
\]
Moreover, if $A \subseteq \R^n$ and $f \colon A \to \R$, then the enlargement $\en{f}\colon \en{A}\to \U$ of $f$ is given by
\[ \en{f}([(r_i^1)], \ldots, [(r_i^n)])\df[(f(r_i^1,\dots,r_i^n))].\]
\end{definition}

If $A \subseteq \R^n$ is infinite, its enlargement $\en{A} \subseteq \U^n$ must contain some elements outside $\R^n$ (see, e.g., \cite[Theorem~3.9.1]{Goldblatt}).
On the other hand, if $A$ is finite, then $A=\en{A}$.

Let $\mathscr{L}$ be a first-order language and $(\R, (P_\alpha), (f_\alpha)))$ an $\mathscr{L}$-structure, where the $P_\alpha$'s and $f_\alpha$'s are the interpretations of the predicate and function symbols of $\mathscr{L}$ in $\R$. Then $(\U, (\en{P_\alpha}), (\en{f_\alpha})))$ is also an $\mathscr{L}$-structure. We are now ready to state one of our main tools: the \emph{transfer principle}, which is a consequence of {\L}o\'s theorem (see, e.g., {\cite[Theorem 4.1.9]{CK}}). 
\begin{theorem}[Transfer principle]\label{t:transfer}
Let $\varphi$ be a first-order $\mathscr{L}$-sentence. Then $\varphi$ is true in $(\R, (P_\alpha), (f_\alpha))$ if and only $\varphi$ is true in $(\U, (\en{P_\alpha}), (\en{f_\alpha}))$.
\end{theorem}

\begin{remark}
\mbox{}\begin{enumerate}
\item An immediate consequence of the transfer principle is that $\U$ is a lattice-ordered field with the operations that are enlargements of the operations of the lattice-ordered field $\R$. In addition, $\R$ is an ordered subfield of $\U$. To keep the notation simple, we will denote the enlargements of these operations and the order without the suffix $\en{}$.
\item Since $\U$ is a field, $\U^n$ has a natural structure of $\U$-vector space. Moreover, $\U^n$ is also an $\R$-vector space containing $\R^n$ as a subspace because $\R$ is a subfield of $\U$. In particular, $\U^n$ is a Riesz space.
\item We will often transfer first-order properties of subsets $P \subseteq \R^n$ or functions $f \colon A \to \R$ to their enlargements $\en{P} \subseteq \U^n$ and $\en{f} \colon \en{A} \to \U$ and vice versa. 
This can be done by adding predicate and function symbols to $\mathscr{L}$ and interpret them in $\R$ as the relations or functions to which we want to apply the transfer principle.
\end{enumerate}
\end{remark}

Archimedeanity is not a first-order property and it does not transfer to $\U$. In fact, $\U$ cannot be Archimedean because it satisfies \Cref{ass:U-embeds}.
A hyperreal $\eps \in \U$ is called \emph{infinitesimal} if $-r <\eps < r$ for all $0< r\in \R$. We say that $\alpha \in \U$ is \emph{limited} if there exist $r,s\in \R$ such that $r< \alpha< s$, otherwise $\alpha$ is called \emph{unlimited}. 
Since $\U$ is an ultrapower over a non-principal ultrafilter on $\N$, it contains non-zero infinitesimal and unlimited elements (see~\cite[Section 3.8]{Goldblatt}).
An $x\in\U^{n}$ is called infinitesimal or limited if all its components are infinitesimal or limited, respectively.

We say that $\alpha, \beta \in \U$ are \emph{infinitely close} if $\alpha-\beta$ is infinitesimal. Every limited hyperreal $\alpha$ is infinitely close to exactly one real number (see, e.g., \cite[p.~57]{rob} and \cite[Theorem 5.6.1]{Goldblatt}), that is called its \emph{standard part} and is denoted by $\st(\alpha)$. Notice that infinitesimals in $\U$ are exactly the elements whose standard part is zero. The function $\st$ extends to any limited element of $\U^{n}$ by acting pointwise. The following theorem is one of the fundamental results of non-standard analysis (see, e.g., \cite[Theorem 4.5.2]{rob} and \cite[Theorem 7.1.1]{Goldblatt}).

\begin{theorem}\label{l:zero-ns->zero-s}
A function $f\colon\R^n\to \R$ is continuous if and only if $f(\st(x))=\st(\en{f}(x))$ for each limited $x\in \U^n$. 
\end{theorem}

We gather here some facts about the operations on $\U$. For a list of similar facts see, e.g., \cite[Section 5.2]{Goldblatt}.

\begin{lemma}\label{c:basic-fact-infinitesimal}
Let $\alpha, \beta, \varepsilon, \delta \in \U$.
\begin{enumerate}
\item\label{c:basic-fact-infinitesimal:item1} All ring and Riesz space operations on  $\U$ commute with $\st$.
\item\label{c:basic-fact-infinitesimal:item2} If $\alpha,\beta$ are limited, then $\alpha\le \beta$ implies $\st(\alpha)\le \st(\beta)$.
\item\label{c:basic-fact-infinitesimal:item3} If $\varepsilon$ and $\delta$ are infinitesimals, $\alpha$ is limited, and $\beta$ is unlimited, then
\begin{enumerate}[label=\alph*., ref=\alph*]
\item\label{c:basic-fact-infinitesimal:item3a} $\varepsilon+\delta$ and $\varepsilon \delta$ are infinitesimal.
\item\label{c:basic-fact-infinitesimal:item3b} $\alpha+ \varepsilon$ is limited and $\alpha \varepsilon$ is infinitesimal.
\item\label{c:basic-fact-infinitesimal:item3c} $\alpha+\beta$ is unlimited and if $\alpha$ is not infinitesimal, then $\alpha \beta$ is unlimited.
\item\label{c:basic-fact-infinitesimal:item3d} $1/\beta$ is infinitesimal and $1/\varepsilon$ is unlimited if $\varepsilon \neq 0$.
\end{enumerate}
\end{enumerate}
\end{lemma}

\begin{proof}
\Cref{c:basic-fact-infinitesimal:item1} follows from \Cref{l:zero-ns->zero-s} and the fact that all the ring and Riesz space operations on $\mathbb{R}$ are continuous.

\Cref{c:basic-fact-infinitesimal:item2}. If $\alpha\le \beta$, then $\alpha=\alpha\wedge \beta$. By \Cref{c:basic-fact-infinitesimal:item1}, $\st(\alpha)=\st(\alpha\wedge \beta)=\st(\alpha)\wedge\st(\beta)$, and hence $\st(\alpha)\le \st(\beta)$.

\Cref{c:basic-fact-infinitesimal:item3}. We prove \Cref{c:basic-fact-infinitesimal:item3a}, the other proofs are similar. If $\varepsilon$ and $\delta$ are infinitesimals, then $\st(\varepsilon)=\st(\delta)=0$. By \Cref{c:basic-fact-infinitesimal:item1}, $\st(\varepsilon+\delta)=\st(\varepsilon)+\st(\delta)=0+0=0$ and $\st(\varepsilon\cdot\delta)=\st(\varepsilon)\cdot\st(\delta)=0\cdot0=0$. Thus, $\varepsilon+\delta$ and $\varepsilon\cdot \delta$ are infinitesimal. 
\end{proof}

Before moving to the characterization of irreducible closed subset of $\U^n$, we describe some immediate consequences of the transfer principle relative to the framework developed in \Cref{sec:duality}.

\begin{remark}\label{rem:enlargement piecewise linear}
By \cite[Theorem 2.4]{baker68}, the free Riesz space $\freersn$ is isomorphic to the Riesz space of continuous piecewise linear functions $\R^n \to \R$ with pointwise operations. This isomorphism is obtained by mapping the equivalence class of a term $t$ to the corresponding definable function. If $t$ is a term, let $\varphi$ be the first-order sentence ``for all $x$, $f(x)=t(x)$'' in the language of Riesz spaces with an additional unary function symbol $f$. The transfer principle applied to $\varphi$ yields that a function $f \colon \R^n \to \R$ is defined by the term $t$ if and only if $\en{f} \colon \U^n \to \U$ is defined by $t$. Thus, the definable functions from $\U^n$ to $\U$ are exactly the enlargements of continuous piecewise linear functions from $\R^n$ to $\R$. If $S \subseteq \U^n$, it then follows from \Cref{l:PWL-I(C)} that $\freersn/\CU(S)$ is isomorphic to the Riesz space of the enlargements of continuous piecewise linear functions restricted to $S$. By \cite[Section 7]{baker68}, an analogous result holds for $\ell$-groups, by considering only piecewise linear functions with integer coefficients.
\end{remark}

\begin{remark}\label{rem:*V_I(f)=V_U(f)}
For any $t \in \free_n$ the transfer principle yields
\begin{align*}
\en{(\VI(t)}) & = \en{\{ x \in \R^n \mid t(x)=0\}} = \{ x \in \U^n \mid t(x)=0\} = \VU(t).
\end{align*}
However, if $T \subseteq \free_n$, then $\en{(\VI(T))}$ and $\VU(T)$ are two subsets of $\U^n$ that do not necessarily coincide. For example, 
let $n$ be a positive integer and $t_n \in \free_2$ be the equivalence class of the term $0 \wedge x \wedge y \wedge (x-ny)$. Set $T = \{t_n \mid n \ge 1 \}$. Then  $\VI(t_n)=\{(x,y)\in \R^2 \mid x,y \ge 0, \ ny\le x\}$ and $\VI(T)=\{(x,y)\in \R^2 \mid  y=0,\ x\ge 0\}$. Thus, $\en{(\VI(T))}=\{(x,y)\in \U^2 \mid  y=0,\ x\ge 0\}$ by the transfer principle, whereas
\[
\VU(T)=\bigcap_n \VU(t_n)=\{(x,y)\in \U^2 \mid  x,y \ge 0 \text{ and } y/x \text{ is infinitesimal}\},
\]
which is larger than $\en{(\VI(T))}$. Nevertheless, the inclusion $\en{(\VI(T))} \subseteq  \VU(T)$ holds for any $T \subseteq \free_n$. Indeed, $\VI(T)\subseteq \VI(t)$ yields $\en{(\VI(T))}\subseteq \en{(\VI(t))}=\VU(t)$ for any $t\in T$. Consequently, $\en{(\VI(T))}\subseteq \bigcap_{t\in T}\VU(t)=\VU(T)$.
\end{remark}

The following example shows that the topologies on $\Uz^\kappa$ depend on whether we are working with Riesz spaces or $\ell$-groups, as mentioned in \Cref{rem:FinerTop}.

\begin{example}\label{ex:finerTop}
Let $r$ be an irrational real number and $A=\{(x,y) \in \Uz^2 \mid y=rx \}$. Then $A=\VUst(s)$, where $s \in \free^{v}_2$ is the equivalence class of the term $rx-y$ in the language of Riesz spaces. It follows that $A$ is closed in the Zariski topology on $\Uz^2$ relative to Riesz spaces. We show that $A$ is not closed in the Zariski topology on $\Uz^2$ relative to $\ell$-groups.
If $A \subseteq \VUst(t)$ for some $t \in \free^\ell_2$, then 
\[
\{(x,y) \in \mathbb{R}^2 \mid y=rx \} = \{ \orig \} \cup (A \cap \mathbb{R}^2) \subseteq \VI(t).
\]
Since $\VI(t)$ is a rational polyhedral cone, there exist $p,q \in \Q$ such that $p < r < q$ and $\{1\} \times [p,q] \subseteq \VI(t)$. By the transfer principle and \Cref{rem:*V_I(f)=V_U(f)}, for every $\alpha \in \U$ such that $p \le \alpha \le q$ we have that $(1, \alpha) \in \en{(\VI(t))}=\VU(t)$. In particular, if $\varepsilon \neq 0$ is an infinitesimal, then $(1, r+\varepsilon) \in \VU(t)$. Since this is true for any $t \in \free^\ell_2$ such that $A \subseteq \VUst(t)$, it follows that $(1,r+\varepsilon) \in \VUst\CU(A)$. Thus, $A$ is not closed in the Zariski topology on $\Uz^2$ relative to $\ell$-groups because $(1,r+\varepsilon) \notin A$. Indeed, it can be shown that 
\[
\VUst\CU(A) = \{ (x,y) \in \Uz^2 \mid x \neq 0 \text{ and } y/x-r \text{ is infinitesimal} \}.
\]
This example can be easily generalized to show that the two Zariski topologies do not coincide for any $n > 1$. Nonetheless, both topologies induce the same subspace topology on $\R^n \setminus \{ \mathbf{0} \}$ (see \Cref{rem:FinerTop}).
\end{example}

The results in the remainder of this section depend on whether we work with Riesz Spaces or $\ell$-groups. Hence, we shall first deal with the case of Riesz Spaces, and then use it to discuss the case of $\ell$-groups.

\subsection{The case of Riesz spaces}

The goal of this subsection is to show that irreducible closed subsets of $\Uz^n$ in the Zariski topology relative to Riesz spaces are in bijective correspondence with tuples of orthonormal vectors of $\R^n$. It will follow that such tuples are in 1-1 correspondence with the prime ideals of the free $n$-generated Riesz space, as shown in \cite{MR1707667}. Following \cite[Definition 2.1]{BM}, where similar tuples are introduced to study prime ideals in free MV-algebras, we call such tuples indexes.

\begin{definition}
An \emph{index} is a tuple $(v_1, \ldots, v_k)$ of orthonormal vectors of $\R^n$. Let $\indu,\indv$ be indexes in $\mathbb{R}^n$. We write $\indu \le \indv$ if $\indu$ is a truncation of $\indv$, i.e., if $\indv=(v_1, \ldots, v_k)$ and $\indu=(v_1, \ldots, v_j)$ with $j \le k$.
\end{definition}

Next theorem allows us to associate an index to any point of $\Uz^n$.

\begin{theorem}[{\cite[Proposition 5.2.1 and Theorem 5.3.1]{NAP}}]\label{t:decomposition}
If $x\in \Uz^n$, then $x=\alpha_1v_1+\dots + \alpha_k v_k$ where $\alpha_1, \ldots , \alpha_k \in \U$ are strictly positive, $\alpha_{i+1}/\alpha_{i}$ is infinitesimal for each $i < k$, and $v_1, \ldots , v_k\in \R^n$ are orthonormal vectors. Furthermore, this decomposition is unique.
\end{theorem}
We call the representation of $x$ as $\alpha_1v_1+\dots + \alpha_k v_k$, given in \Cref{t:decomposition}, the \emph{orthogonal decomposition} of $x$.
We denote by $\imm(x)$ the tuple of vectors $(v_1 , \dots, v_k)$ that appear in the orthogonal decomposition of $x$. Notice that $\imm(x)$ is an index, which we call the  \emph{index of $x$}.

\begin{lemma}\label{l:orthdec-standardpart}
Let $x \in \Uz^n$ with orthogonal decomposition $x=\alpha_1 v_1 + \cdots + \alpha_k v_k$.
\mbox{}\begin{enumerate}
\item\label{l:orthdec-standardpart:item1} $k \le n$.
\item\label{l:orthdec-standardpart:itema} If $i < j$, then $\alpha_j/\alpha_i$ is an infinitesimal.
\item\label{l:orthdec-standardpart:item2} $x$ is limited if and only if $\alpha_1$ is limited. In this case, $\st(x)=\st(\alpha_1) v_1$.
\item\label{l:orthdec-standardpart:item3} $x$ is infinitesimal if and only if $\alpha_1$ is infinitesimal.
\end{enumerate}
\end{lemma}

\begin{proof}
\Cref{l:orthdec-standardpart:item1} follows from the fact that $v_1, \ldots, v_k$ are orthogonal vectors of $\mathbb{R}^n$, and so are linearly independent.

\Cref{l:orthdec-standardpart:itema}. Writing $\alpha_j/\alpha_i = (\alpha_{i+1}/\alpha_i) \cdots (\alpha_j/\alpha_{j-1})$, \Cref{c:basic-fact-infinitesimal}\eqref{c:basic-fact-infinitesimal:item3} guarantees that $\alpha_j/\alpha_i$ is infinitesimal, as product of infinitesimals.

\Cref{l:orthdec-standardpart:item2}. Suppose $\alpha_1$ is limited. By \Cref{l:orthdec-standardpart:itema}, $\alpha_i/\alpha_1$ is infinitesimal for each $i>1$. By \Cref{c:basic-fact-infinitesimal}\eqref{c:basic-fact-infinitesimal:item3}, $\alpha_i= \alpha_1 (\alpha_i/\alpha_1)$ is infinitesimal, and so $\alpha_i v_i$ is infinitesimal for each $i>1$. Since $\alpha_1v_1$ is limited and $x=\alpha_1 v_1 + \cdots + \alpha_k v_k$, \Cref{c:basic-fact-infinitesimal} entails that $x$ is limited and
\begin{align*}
\st(x) &= \st(\alpha_1 v_1 + \dots + \alpha_k v_k) = \st(\alpha_1 v_1) + \dots + \st(\alpha_k v_k)\\
 &= \st(\alpha_1 v_1)=\st(\alpha_1) \st(v_1) =\st(\alpha_1) v_1.
\end{align*}
For the other implication we reason contrapositively. Assume that $\alpha_1$ is unlimited. Notice that $x\alpha_1^{-1} = v_1 + (\alpha_2/\alpha_1) v_2 + \cdots + (\alpha_k/\alpha_1) v_k$, where the right-hand side of the equation is not infinitesimal by \Cref{c:basic-fact-infinitesimal}\eqref{c:basic-fact-infinitesimal:item3} because $v_1$ is not infinitesimal. Thus, $x\alpha_1^{-1}$ is not infinitesimal, and so  $x\alpha_1^{-1}\alpha_1=x$ is unlimited. 

\Cref{l:orthdec-standardpart:item3}. We have that $x$ is infinitesimal if and only if $\st(x)=\orig$. By \Cref{l:orthdec-standardpart:item2} this is equivalent to $\st(\alpha_1) v_1=\orig$. Since $v_1$ is non-zero, $x$ is infinitesimal if and only if $\st(\alpha_1)=0$, which means that $\alpha_1$ is infinitesimal.
\end{proof}

The next lemma shows that the Zariski topology on $\Uz^n$ relative to Riesz spaces is generated by the enlargements of half-spaces. Once more, we write $\lr$ for the set of functions $f\colon\R^n \to \R$ defined by homogeneous linear polynomials. A family of closed subsets of a topological space is said to form a subbasis if every closed subset can be written as an intersection of finite unions of elements of the family.

\begin{lemma}\label{r:basis-hyper}
The sets $S_f\df\{ x \in \Uz^n \mid \en{f}(x) \ge 0 \}$ with $f \in \lr$ form a subbasis of closed subsets for the Zariski topology of $\Uz^n$ relative to Riesz spaces.
\end{lemma}
\begin{proof}
Let $f \in \lr$. The linear polynomial $p$ defining $f$ is in fact a term in the language of Riesz spaces. If $t \in \freersn$ is the equivalence class of the term $p \wedge 0$, then the transfer principle and \Cref{rem:*V_I(f)=V_U(f)} entail that $S_f=\en{(\VI(t))} \setminus \{ \orig \} = \VUst(t)$. Thus, $S_f$ is a closed subset of $\Uz^n$.
By \cite[Lemma 3.2]{baker68}, for every $t\in \free_n$ there exist $l_{ij}\in\lr$, with $i=1, \ldots ,k$ and $j=1, \ldots, m$, such that 
\[
\VI(t)=\bigcap_{i=1}^k\bigcup_{j=1}^m\{ x \in \R^n \mid l_{ij}(x) \ge 0 \}.
\]
Hence, \Cref{rem:*V_I(f)=V_U(f)} yields that 
\[
\VUst(t)=\en{(\VI(t))}\setminus\{\mathbf{0}\}=\bigcap_{i=1}^k\bigcup_{j=1}^m\{ x \in \Uz^n \mid \en{l_{ij}}(x) \ge 0 \}
\]
because the enlargement commutes with finite intersections and unions by the transfer principle. The claim follows because the sets  $\VUst(t)$ as $t$ ranges in $\free_n$ form a basis of closed subsets for the Zariski topology.
\end{proof}

\begin{remark}\label{l:enlarge-linear}
If $f \in \lr$, then $\en{f} \colon \U^n \to \U$ is a linear map of $\U$-vector spaces by the transfer principle. 
Since $\en{f}$ agrees with $f$ on $\R^n$, we will write $f(v)$ instead of $\en{f}(v)$ when $v \in \R^n$.
\end{remark}

The following lemma shows that if $f \colon \R^n \to \R$ is a linear map and $x \in \Uz^n$, then the sign of $\en{f}(x)$ is completely determined by the values of $f$ on the vectors of $\imm(x)$.

\begin{lemma}\label{l:zero-index}
Let $x \in \Uz^n$ with $\imm(x)=(v_1, \ldots, v_k)$ and $f \in \lr$. 
\begin{enumerate}
\item \label{l:zero-index:item1} $\en{f}(x) = 0$ if and only if $f(v_i)=0$ for each $i=1, \ldots, k$. 
\item \label{l:zero-index:item2} $\en{f}(x) > 0$ if and only if there exists $i$ such that $f(v_j)=0$ for each $j<i$ and $f(v_i) > 0$. 
\item \label{l:zero-index:item3} $\en{f}(x) < 0$ if and only if there exists $i$ such that $f(v_j)=0$ for each $j<i$ and $f(v_i) < 0$. 
\end{enumerate}
\end{lemma}
\begin{proof}
Since $\imm(x)=(v_1, \ldots, v_k)$, there exists an orthogonal decomposition $x=\alpha_1v_1+\dots + \alpha_k v_k$.

\Cref{l:zero-index:item1}. By \Cref{l:enlarge-linear}, $\en{f}$ is a $\U$-linear map extending $f$, and so $\en{f}(x)=\alpha_1 f(v_1) + \cdots + \alpha_k f(v_k)$. Thus, the right-to-left implication is straightforward. For the other implication, assume that $\en{f}(x)=0$. Then
\begin{align}\label{eq:zero-index1}
f(v_1) + \frac{\alpha_2}{\alpha_1}f(v_2)+\cdots + \frac{\alpha_k}{\alpha_1} f(v_k)=0,
\end{align}
with $f(v_1)\in \R$. Since each ${\alpha_i}/{\alpha_1}$ is infinitesimal by \Cref{l:orthdec-standardpart}(\ref{l:orthdec-standardpart:itema}), 
we have
\[
f(v_1)=\st\left(f(v_1)\right)=-\st\left(\frac{\alpha_2}{\alpha_1}f(v_2)+\cdots + \frac{\alpha_k}{\alpha_1} f(v_k)\right)=0.
\]
Thus, \eqref{eq:zero-index1} implies that
\[
f(v_2) + \frac{\alpha_3}{\alpha_2}f(v_3)+ \dots + \frac{\alpha_k}{\alpha_2} f(v_k)=0.
\]
Since $\alpha_i/\alpha_j$ is infinitesimal for any $i>j$, we can repeat the same argument and deduce that $f(v_i)=0$ for each $i$.

\Cref{l:zero-index:item2}. Assume $\en{f}(x) > 0$. Since $\en{f}(x)=\alpha_1 f(v_1) + \cdots + \alpha_k f(v_k)$, there exists $i$ such that $f(v_i) \neq 0$ and $f(v_j)=0$ for each $j < i$. Since $\alpha_i >0$, we have that 
\[
f(v_i) + \frac{\alpha_{i+1}}{\alpha_i}f(v_{i+1})+ \cdots + \frac{\alpha_k}{\alpha_i} f(v_k) = \frac{\en{f}(x)}{\alpha_i} > 0.
\]
It follows from \Cref{c:basic-fact-infinitesimal}\eqref{c:basic-fact-infinitesimal:item2} that
\[
f(v_i) = \st \left( f(v_i) + \frac{\alpha_{i+1}}{\alpha_i}f(v_{i+1})+ \cdots + \frac{\alpha_k}{\alpha_i} f(v_k) \right) \ge 0,
\]
which implies that $f(v_i) > 0$ because $f(v_i) \neq 0$.

Vice versa, assume there exists $i$ such that $f(v_j)=0$ for each $j < i$ and $f(v_i) > 0$. Then
\[
\frac{1}{\alpha_i} \en{f}(x)=f(v_i) + \frac{\alpha_{i+1}}{\alpha_i}f(v_{i+1})+ \cdots + \frac{\alpha_k}{\alpha_i} f(v_k).
\]
Thus, $\st(\alpha_i^{-1} \en{f}(x))=f(v_i) > 0$ which, again by \Cref{c:basic-fact-infinitesimal}\eqref{c:basic-fact-infinitesimal:item2}, implies that $\alpha_i^{-1}\en{f}(x) > 0$. Therefore, $\en{f}(x) > 0$ because $\alpha_i >0$.

\Cref{l:zero-index:item3}. Apply \Cref{l:zero-index:item2} to $-f$.
\end{proof}

Given a vector subspace $V$ of $\R^n$, we denote by $\ort{V}$ its \emph{orthogonal complement}. We will need the following simple linear algebra result.

\begin{lemma}\label{l:existence-linear-maps}
Let $V$ be a vector subspace of $\mathbb{R}^n$.
\mbox{}\begin{enumerate}
\item\label{l:existence-linear-maps:item1} If $w \in V^\perp$ is a unit vector, then there exists $f \in \lr$ such that $f(w)<0$, and $f(v)=0$ for all $v \in V$.
\item\label{l:existence-linear-maps:item2} If $u,w \in V^\perp$ are two distinct unit vectors, then there exists $f \in \lr$ such that $f(u)>0$, $f(w)<0$, and $f(v)=0$ for all $v \in V$.
\end{enumerate}
\end{lemma}

\begin{proof}
Let $v_1, \ldots, v_k$ be a basis of $V$.

\Cref{l:existence-linear-maps:item1}. Since $w \in V^\perp$, the vectors $v_1, \ldots, v_k, w$ are linearly independent. By a standard argument from linear algebra, there exists $f \in \lr$ such that $f(v_i)=0$ for each $i$ and $f(w)=-1$.  Therefore, $f(w)<0$, and $f(v)=0$ for all $v \in V$.

\Cref{l:existence-linear-maps:item2}. If $u=-w$, then by \Cref{l:existence-linear-maps:item1} there exists  $f \in \lr$ such that $f(v)=0$ for each $v \in V$ and $f(w)<0$, and so $f(u)>0$. Suppose $u \neq -w$. Since $u,w$ are distinct unit vectors, they are linearly independent. Since $u,w \in V^\perp$, it follows that $v_1, \ldots, v_k, u, w$ are linearly independent vectors of $\mathbb{R}^n$. Thus, there exists $f \in \lr$ such that $f(v_i)=0$ for each $i$, $f(u)=1$, and $f(w)=-1$. Therefore, $f(u)>0$, $f(w)<0$, and $f(v)=0$ for all $v \in V$.
\end{proof}

We are now able to characterize the closures of points of $\Uz^n$ using indexes.

\begin{theorem}\label{l:closure-index}
$x \in \VUst \CU ( y )$ if and only if $\imm(x) \le \imm(y)$.
\end{theorem}
\begin{proof}
By \Cref{r:basis-hyper}, $x \in \VUst \CU ( y )$ if and only if $y \in S_f$ implies $x \in S_f$ for each $f \in \lr$. Thus, $x \in \VUst \CU ( y )$ if and only if $\en{f}(y)\ge 0$ implies $\en{f}(x)\ge 0$ for each $f\in \lr$. By \Cref{l:zero-index}(\ref{l:zero-index:item1},\ref{l:zero-index:item2}), if $\imm(x) \le \imm(y)$ and $\en{f}(y) \ge 0$, then $\en{f}(x) \ge 0$. Therefore, the right-to-left implication holds.

Conversely, suppose $\imm(x) \nleq \imm(y)$. Let $\imm(x)=(v_1, \ldots ,v_k)$ and $\imm(y)=(w_1, \ldots ,w_t)$.
We first consider the case $\imm(y) < \imm(x)$, which means that $t<k$ and $v_i=w_i$ for any $i \le t$. Let $V$ be the span of $v_1, \ldots, v_t$ in $\R^n$. Then $v_{t+1} \in V^\perp$ and \Cref{l:existence-linear-maps}(\ref{l:existence-linear-maps:item1}) implies that there exists $f \in \lr$ such that $f(v_{t+1})<0$ and $f(v_i)=0$ for each $i \le t$. By \Cref{l:zero-index}(\ref{l:zero-index:item1},\ref{l:zero-index:item3}), $\en{f}(y)=0$ and $\en{f}(x) < 0$. Thus, $y \in S_f$ and $x \notin S_f$, and so $x \notin \VUst \CU (y)$.
If $\imm(y)$ and $\imm(x)$ are incomparable, then there exists $i \le \min(k,t)$ such that $v_j=w_j$ for $j < i$ and $v_i\neq w_i$. Let $V$ be the span of $v_1, \ldots, v_{i-1}$. 
Since $w_i, v_i \in V^\perp$, \Cref{l:existence-linear-maps}(\ref{l:existence-linear-maps:item2}) implies that there exists $f \in \lr$ such that $f(w_i) > 0$, $f(v_i) < 0$, and $f(v_j)=f(w_j)=0$ for each $j < i$. By \Cref{l:zero-index}(\ref{l:zero-index:item2},\ref{l:zero-index:item3}), $\en{f}(y)>0$ and $\en{f}(x) < 0$. Thus $y \in S_f$ and $x \notin S_f$, and so $x \notin \VUst \CU ( y )$. 
\end{proof}

\begin{corollary}
$\VUst \CU (x)= \VUst \CU (y)$ if and only if $\imm(x)= \imm(y)$.
\end{corollary}

\begin{definition}
Let $\indv$ be an index. We set $\cone(\indv) := \{ x \in \Uz^n \mid \imm(x) \le \indv \}$.
\end{definition}

As an immediate consequence of \Cref{l:closure-index} we obtain the following theorem characterizing the closure of points of $\Uz^n$, which are exactly the irreducible closed subsets of $\Uz^n$ by \cref{p:sober}.

\begin{theorem}\label{lem:clx-coneix}
$\VUst \CU (y) = \cone(\imm(y))$ for each $y \in \Uz^n$.
\end{theorem}

\begin{example}\label{ex:cone(v)}
Let $e_1=(1,0)$ and $e_2=(0,1)$ be the standard basis of $\mathbb{R}^2$ and $\indv$ the index $(e_1,e_2)$. Then
\[
\cone(\indv) = \{ (x_1, x_2) \in \Uz^2 \mid x_1 > 0, \ x_2 \ge 0, \text{ and } x_2/x_1 \text{ is infinitesimal} \}.
\]
If $\varepsilon \in \U$ is a positive infinitesimal and $y=(1, \varepsilon) \in \Uz^2$, then $y=1 e_1 + \varepsilon e_2$ is the orthogonal decomposition of $y$. Thus, $\imm(y)=\indv$. By \cref{lem:clx-coneix}, $\cone(\indv)=\VUst \CU (y)$. Let $T = \{t_n \mid n \ge 1 \} \subseteq \free_2$ be defined as in \Cref{rem:*V_I(f)=V_U(f)}. We have that $\cone(\indv)=\VUst(T)$, and hence $\CU(\cone(\indv))$ is the prime ideal $\langle T \rangle$ of $\free_2$.

Let $\indu=(e_1)$ be the index obtained by truncating $\indv$ and $z=(1,0) \in \Uz^2$. Then $\imm(z)=\indu$ and so $\VUst\CU(z)=\cone(\indu) = \{ (x_1,0) \in \Uz^2 \mid x_1 > 0 \}$, which is contained in $\cone(\indv)$.
\end{example}

\begin{corollary}\label{cor:cone-irr-order-iso}
Let $\cone$ be the map that associates $\cone(\indv)$ to any index $\indv$.
\begin{enumerate}
\item\label{cor:cone-irr-order-iso:item1} $\cone$ is an order isomorphism between the set of indexes ordered by truncation and the set of irreducible closed subsets of $\Uz^n$ ordered by inclusion.
\item\label{cor:cone-irr-order-iso:item2} ${\CU} \circ {\cone}$ is an order isomorphism between the set of indexes ordered by truncation and $\spec(\freersn)$ ordered by reverse inclusion.
\end{enumerate}
\end{corollary}

\begin{proof}
\Cref{cor:cone-irr-order-iso:item1}. Let $\indv=(v_1, \dots, v_k)$ be an index and $\varepsilon \in \U$ a positive infinitesimal. We have that $\indv=\imm(y)$, where $y=v_1 + \varepsilon v_2 + \dots + \varepsilon^{k-1} v_k \in \Uz^n$. By \Cref{lem:clx-coneix}, $\cone(\indv)=\cone(\imm(y))=\VUst \CU (y)$, which is an irreducible closed subset of $\Uz^n$. Thus, $\cone$ is a well-defined function. It is clear that $\indv \le \indw$ implies $\cone(\indv) \subseteq \cone(\indw)$. To show the reverse implication, let $\cone(\indv) \subseteq \cone(\indw)$. Take $y \in \Uz^n$ such that $\imm(y)=\indv$. Then $y \in \cone(\indv) \subseteq \cone(\indw)$, and so $\indv=\imm(y) \le \indw$. It remains to prove that $\cone$ is onto. Let $C \subseteq \Uz^n$ be an irreducible closed subset. By \Cref{p:sober}, $C=\VUst \CU (y)$ for some $y \in \Uz^n$. Therefore, $C=\cone(\imm(y))$ by \Cref{lem:clx-coneix}.

\Cref{cor:cone-irr-order-iso:item2}. By \cref{p:irrClosed=prime}, $\CU$ is an order isomorphism between the set of irreducible closed subsets of $\Uz^n$ ordered by inclusion and $\spec(\freersn)$ ordered by reverse inclusion. Thus, ${\CU} \circ {\cone}$ is an order isomorphism by \Cref{cor:cone-irr-order-iso:item1}.
\end{proof}

We now show that, under the assumptions made at the beginning of this section, there is a natural way to define the embedding $\mathscr{E} \colon \spec(\freersn) \to \Uz^n$ (see the end of \Cref{sec:first-properties}), which yields a coordinatization of $\spec(\freersn)$.

\begin{remark}\label{rem:coordinatization}
Fix a positive infinitesimal $\varepsilon \in \U$. By \Cref{cor:cone-irr-order-iso}\eqref{cor:cone-irr-order-iso:item2}, for each $P \in \spec(\freersn)$ there exists a unique index $\indv=(v_1, \dots, v_k)$ such that $\cone(\indv)=\VUst(P)$. We set $\mathscr{E}(P)=v_1 + \varepsilon v_2 + \dots + \varepsilon^{k-1} v_k$. 
Since $\imm(v_1 + \varepsilon v_2 + \dots + \varepsilon^{k-1} v_k)=\indv$, \Cref{lem:clx-coneix} yields that $\VUst \CU (\mathscr{E}(P))=\cone(\indv)=\VUst(P)$. Thus, $\CU (\mathscr{E}(P)) = P$.
It follows from \Cref{p:E embedding image dense} that $\mathscr{E}$ is an embedding of $\spec(\freersn)$ into $\Uz^n$. Thus, $\mathscr{E}$ can be thought of as a coordinatization of $\spec(\freersn)$, which is canonical modulo the choice of the positive infinitesimal $\varepsilon$.
The image of $\mathscr{E}$, given by 
\[
\mathscr{E}[\spec(\freersn)]=\{v_1 + \varepsilon v_2 + \dots + \varepsilon^{k-1} v_k \mid (v_1, \dots, v_k) \text{ is an index}\},
\]
is then a dense subspace of $\Uz^n$ that is homeomorphic to $\spec(\freersn)$. Observe that $\mathscr{E}[\spec(\freersn)] \cap \R^n$ coincides with the sphere $\mathbb{S}^{n-1} = \{ x \in \R^n \mid \lVert x \rVert =1 \}$. By \Cref{p:null-U-max}(\ref{p:null-U-max:item1}), $\mathscr{E}$ restricts to the well-known homeomorphism between the maximal spectrum $\Max(\freersn)$ of $\freersn$ and $\mathbb{S}^{n-1}$ (see \cite[Proof of Lemma 3.4]{Marra}). An immediate consequence of the definition of $\mathscr{E}$ is that each point in $\mathscr{E}[\spec(\freersn)]$ is infinitely close to a (necessarily unique) point of $\mathbb{S}^{n-1}$. If $P \in \spec(\freersn)$, it follows from \Cref{cor:cone-irr-order-iso}(\ref{cor:cone-irr-order-iso:item2}) that $\mathscr{E}(P)$ is infinitely close to $\mathscr{E}(M) \in \mathbb{S}^{n-1}$, where $M$ is the unique maximal ideal containing $P$.
Finally, if $J$ is an ideal of $\freersn$, then $\mathscr{E}$ restricts to an embedding of $\spec(\freersn/J)$ into $\VUst(J)$, whose image is $\mathscr{E}[\spec(\freersn)] \cap \VUst(J)$.
\end{remark}

\subsection{The case of \texorpdfstring{$\ell$}{l}-groups}
Building on the previous subsection, we now describe the more complex case of $\ell$-groups.

The next lemma, analogous to \Cref{r:basis-hyper}, shows that the closed subsets of the Zariski topology on $\Uz^n$ relative to $\ell$-groups are generated by the enlargements of rational half-spaces. Recall that we denote by $\lz$ the set of functions $f\colon\R^n \to \R$ defined by homogeneous linear polynomials with integer coefficients.

\begin{lemma}\label{r:basis-hyper-l-groups}
The sets $S_f\df\{ x \in \Uz^n \mid \en{f}(x) \ge 0 \}$ with $f \in \lz$ form a subbasis of closed subsets for the Zariski topology of $\Uz^n$ relative to $\ell$-groups.
\end{lemma}
\begin{proof}
Let $t\in \freen$. By \cite[Lemma~3.2, Section~7]{baker68}, there exist $l_{ij}\in\lz$, with $i=1, \ldots ,k$ and $j=1, \ldots, m$, such that $\VI(t)=\bigcap_{i=1}^k\bigcup_{j=1}^m\{ x \in \R^n \mid l_{ij}(x) \ge 0 \}$.
The rest of the proof is analogous to the one of \Cref{r:basis-hyper}.
\end{proof}

In order to characterize the indexes corresponding to the irreducible closed subsets of the Zariski topology relative to $\ell$-groups, we need to introduce the notion of rational subspaces of $\R^n$.

\begin{definition}
Let $V$ be a vector subspace of $\mathbb{R}^n$. We say that $V$ is \emph{rational} if it admits a linear basis consisting of vectors from $\Q^{n}$ or, equivalently, from $\Z^n$. 
\end{definition}

The following lemma is an immediate consequence of the fact that a linear system with rational coefficients admits a solution in $\R^n$ if and only if it admits a solution in $\Q^n$ (see, e.g, \cite[p.~15]{hoffman_linear_1971}).
\begin{lemma}\label{l:lin indep in Q and R}
The vectors $v_1, \ldots, v_m \in \mathbb{Q}^n$ are linearly independent as vectors of the $\mathbb{Q}$-vector space $\mathbb{Q}^n$ if and only if they are linearly independent as vectors of the $\mathbb{R}$-vector space $\mathbb{R}^n$.
\end{lemma}

Any $f \in \lz$ corresponds to a polynomial with integer coefficient, which in turn is a term in the language of $\ell$-groups. For this reason, we will denote the kernel of $f$ by $\VI(f)$.

\begin{lemma}\label{l:subspaces-intersections}
Let $V$ be a vector subspace of $\mathbb{R}^n$ of dimension $m$. The following conditions are equivalent.
\begin{enumerate}[label=(\roman*)]
\item\label{l:subspaces-intersections:item1} $V$ is rational.
\item\label{l:subspaces-intersections:itemA} $V \cap \mathbb{Q}^n$ is dense in $V$ with respect to the Euclidean topology.
\item\label{l:subspaces-intersections:item2} $V \cap \mathbb{Q}^n$ is a subspace of $\mathbb{Q}^n$ of dimension $m$ as a $\mathbb{Q}$-vector space.
\item\label{l:subspaces-intersections:item3} $V$ is the kernel of an onto linear map $f:\mathbb{R}^n \to \mathbb{R}^{n-m}$ with integer coefficients.
\item\label{l:subspaces-intersections:item4} $V = \bigcap_{i \in I} \VI(f_i)$ with $f_i \in \lz$ for each $i$.
\end{enumerate}
\end{lemma}

\begin{proof}
\ref{l:subspaces-intersections:item1} $\Rightarrow$ \ref{l:subspaces-intersections:itemA}. Let $v_1, \ldots, v_m$ be a basis of $V$ consisting of vectors from $\Q^n$. If $x \in V$, there exist $r_1, \ldots, r_m \in \mathbb{R}$ such that $x=r_1 v_1 + \cdots + r_m v_m$. We show that, given $r>0$, there exists $y \in V \cap \Q^n$ such that $\lVert x - y \rVert < r$. Let $s=\lVert v_1 \rVert + \dots + \lVert v_n \rVert$ and $q_1, \ldots, q_m \in \Q$ such that $|r_i-q_i| < r/s$. If $y=q_1 v_1 + \cdots + q_m v_m \in V \cap \Q^n$, then $\lVert x - y \rVert < r$. Therefore, $V \cap \Q^n$ is dense in $V$.

\ref{l:subspaces-intersections:itemA} $\Rightarrow$ \ref{l:subspaces-intersections:item2}. If $V \cap \mathbb{Q}^n$ has dimension strictly smaller than $m$, then the span $W$ of $V \cap \mathbb{Q}^n$ in $\mathbb{R}^n$ is a proper subspace of $V$ because its dimension is smaller than $m$. Thus, $V \cap \mathbb{Q}^n \subseteq W \subset V$ and $W$ is a proper closed subset of $V$. Thus,  $V \cap \mathbb{Q}^n$ is not dense in $V$.

\ref{l:subspaces-intersections:item2} $\Rightarrow$ \ref{l:subspaces-intersections:item1}. Let $v_1, \ldots, v_m$ be a basis of $V \cap \mathbb{Q}^n$. Then $v_1, \ldots, v_m$ are also linearly independent in $\mathbb{R}^n$ by \Cref{l:lin indep in Q and R}. Thus, they form a basis of $V$ because $V$ has dimension $m$.

\ref{l:subspaces-intersections:item2} $\Rightarrow$ \ref{l:subspaces-intersections:item3}. Let $v_1, \ldots, v_m$ be a basis of $V \cap \mathbb{Q}^n$. We can extend it to a basis $v_1, \ldots, v_m, v_{m+1}, \ldots, v_n$ of $\mathbb{Q}^n$. Define $g:\mathbb{Q}^n \to \mathbb{Q}^{n-m}$ by mapping $v_1, \ldots, v_m$ to the zero vector and $v_{m+1}, \ldots, v_n$ to the canonical basis of $\mathbb{Q}^{n-m}$. Then $g$ is defined by $n-m$ homogeneous linear polynomials with rational coefficients. We can assume that the coefficients of $g$ are all integers by multiplying them by the product of their denominators. Let $f:\mathbb{R}^n \to \mathbb{R}^{n-m}$ be the linear map defined by the same polynomials. Then $f$ maps $v_1, \ldots, v_m$ to the zero vector and $v_{m+1}, \ldots, v_n$ to a basis of $\mathbb{R}^{n-m}$. Therefore, $f$ is onto and its kernel is $V$.

\ref{l:subspaces-intersections:item3} $\Rightarrow$ \ref{l:subspaces-intersections:item2}. The map $f$ restricts to a linear function $g: \mathbb{Q}^n \to \mathbb{Q}^{n-m}$ because $f$ has integer coefficients. Since $V$ is the kernel of $f$, the kernel of $g$ is $V \cap \mathbb{Q}^n$. We have that $g$ maps the canonical basis of $\mathbb{Q}^n$ to a set of generators of $\mathbb{Q}^{n-m}$ because $f$ is onto.
By the rank-nullity theorem, the dimension of $V \cap \mathbb{Q}^n$, which is the kernel of $g$, is $n-(n-m)=m$.

\ref{l:subspaces-intersections:item3} $\Rightarrow$ \ref{l:subspaces-intersections:item4}. It is an immediate consequence of the fact that each component of $f$ is a map in $\lz$.

\ref{l:subspaces-intersections:item4} $\Rightarrow$ \ref{l:subspaces-intersections:item3}. If $m=n$, then $V=\R^n$ and the claim is clear. Thus, we can assume that $m < n$. We show that there are $i_1, \dots, i_{n-m} \in I$ such that $V= \bigcap_{j=1}^{n-m} \VI(f_{i_j})$. Since $m < n$, there is $i_1 \in I$ such that $\VI(f_{i_1}) \neq \R^n$, and hence $\VI(f_{i_1})$ has dimension $n-1$. Moreover, if we have found $i_1, \dots, i_k \in I$ such that $\bigcap_{j=1}^{k} \VI(f_{i_j})$ has dimension $n-k > m$, then there exists $i_{k+1} \in I$ such that $\bigcap_{j=1}^{k+1} \VI(f_{i_j})$ has dimension $n-k-1$. Indeed, otherwise $\bigcap_{j=1}^k \VI(f_{i_j}) = \bigcap_{i \in I} \VI(f_i) = V$, which is impossible because the dimension of $V$ is $m < n-k$. So, it is possible to find $i_1, \dots, i_{n-m} \in I$ such that the dimension $\bigcap_{j=1}^{n-m} \VI(f_{i_j})$ is $m$.
Thus, $V$ is the kernel of the linear map $f\colon\mathbb{R}^n \to \mathbb{R}^{n-m}$ given by $f(v)=(f_{i_1}(v), \ldots, f_{i_{n-m}}(v))$. By the rank-nullity theorem, the image of $f$ has dimension $n-m$. Therefore, $f$ is onto.
\end{proof} 

\begin{lemma}\label{lem:perp-is-rational}
If $V$ is a rational subspace of $\mathbb{R}^n$, then $V^\perp$ is a rational subspace. 
\end{lemma}

\begin{proof}
Suppose that $V$ has dimension $m$. By \Cref{l:subspaces-intersections}, $V \cap \Q^n$ has dimension $m$ as a $\Q$-vector space. Let $v_1, \dots, v_m$ be a basis of $V \cap \Q^n$ and $v_{m+1}, \dots, v_n$ a basis of $V^\perp \cap \Q^n$, which is the orthogonal complement of $V \cap \Q^n$ in $\Q^n$. By \cref{l:lin indep in Q and R}, the span of $v_{m+1}, \dots, v_n$ in $\R^n$ is a subspace of dimension $n-m$ contained in $V^\perp$. Since $V^\perp$ has dimension $n-m$, it follows that $v_{m+1}, \dots, v_n$ is a basis of $V^\perp$. Therefore, $V^\perp$ is rational.
\end{proof}

By \Cref{l:subspaces-intersections}, rational subspaces of $\R^n$ are exactly intersections of subspaces of the form $\VI(f)$ with $f \in \lz$. Therefore, rational subspaces are closed under arbitrary intersections. This allows us to consider the smallest rational subspace containing a given subset of $\R^n$.

\begin{definition}
Let $S$ be a subset of $\mathbb{R}^n$. We call the smallest rational vector subspace containing $S$ the \emph{rational envelope of $S$} and denote it by $\env{S}$.
\end{definition}

We are now ready to introduce the indexes that correspond to irreducible closed subsets of $\Uz^n$ in the setting of $\ell$-groups. Following the notation of \cite{MR1707667}, we call such indexes $\mathbb{Z}$-reduced.

\begin{definition}\label{def:reduced}
An index $\indv=(v_1, \ldots, v_k)$ is said to be \emph{$\mathbb{Z}$-reduced} if $v_i \in \ortp{v_j}$ for each $i \neq j$.
\end{definition}

\begin{remark}\label{rem:reduced-alternative-def}
\Cref{lem:perp-is-rational} yields that $v_i \in \ortp{v_j}$ if and only if the subspaces $\env{v_i}$ and $\env{v_j}$ are orthogonal to each other. Therefore, $v_i \in \ortp{v_j}$ if and only if $v_j \in \ortp{v_i}$. \Cref{lem:perp-is-rational} also implies that $\indv$ is $\mathbb{Z}$-reduced if and only if $v_1, \dots, v_{i-1} \in \ortp{v_i}$ for each $i>1$, which is equivalent to $v_i \in \ortp{v_1, \dots, v_{i-1}}$ for each $i>1$.
\end{remark}

We now show how to associate to any index a unique $\mathbb{Z}$-reduced index via a sort of Gram-Schmidt process. A similar process also appears in \cite[p.~188]{MR1707667}.

\begin{definition}\label{def:reduction}
Let $\indv=(v_1, \ldots, v_k)$ be an index. We define $w_1, \ldots, w_k \in \mathbb{R}^n$ by recursion as follows:
$w_1 \coloneqq v_1$ and $w_i$ is the projection of $v_i$ onto $\ortp{w_1, \dots, w_{i-1}}$ for each $1< i \le k$.
Let $(u_1,\dots,u_t)$ be the tuple of vectors obtained by discarding all the zero components of the tuple $(w_1,\dots,w_k)$ and then normalizing the remaining components.
We denote $(u_1, \ldots, u_t)$ by $\red(\indv)$.
\end{definition}

\begin{proposition}\label{p:Z-red-iff-equal-red}
Let $\indv$ be an index.
\begin{enumerate}
\item\label{p:Z-red-iff-equal-red:item1} $\red(\indv)$ is a $\mathbb{Z}$-reduced index.
\item\label{p:Z-red-iff-equal-red:item2} $\indv$ is $\mathbb{Z}$-reduced if and only if $\red(\indv)=\indv$.
\item\label{p:Z-red-iff-equal-red:item3} If $\indv'$ is an index such that $\indv \le \indv'$, then $\red(\indv) \le \red(\indv')$.
\end{enumerate}
\end{proposition}

\begin{proof}
\Cref{p:Z-red-iff-equal-red:item1}. Let $\indv=(v_1, \dots, v_k)$, $w_1, \dots, w_k$, and $u_1, \ldots, u_t$ as in \Cref{def:reduction}. Since $w_i \in \ortp{w_1, \dots, w_{i-1}}$ and $\env{w_j} \subseteq \env{w_1, \dots, w_{i-1}}$ for each $j < i$, we have that $w_i \in \ortp{w_j}$. It follows that $u_i \in \ortp{u_j}$ whenever $j < i$ because the $u_i$'s are obtained by normalizing the non-zero $w_i$'s. Therefore, $\red(\indv)$ is a $\mathbb{Z}$-reduced index. 

\Cref{p:Z-red-iff-equal-red:item2,p:Z-red-iff-equal-red:item3} are immediate consequences of \Cref{p:Z-red-iff-equal-red:item1} and the definition of $\red(\indv)$. 
\end{proof}

\begin{example}
Let $v_1,v_2,v_3 \in \mathbb{R}^3$ be the unit vectors 
\begin{align*}
v_1=\frac{1}{\sqrt{3}}(1,\sqrt{2},0), \qquad v_2=\frac{1}{2}(\sqrt{2},-1,1), \qquad v_3=\frac{1}{2\sqrt{3}}(-\sqrt{2}, 1, 3).
\end{align*}
Since $v_1,v_2,v_3$ are pairwise orthogonal, $\indv \coloneqq (v_1,v_2,v_3)$ is an index. By \Cref{def:reduction}, $w_1=v_1$. The irrationality of $\sqrt{2}$ yields that $\env{w_1}$ is the $xy$-plane in $\mathbb{R}^3$. Thus, $\ortp{w_1}$ is the $z$-axis, and then 
\[
w_2=\frac{1}{2} (0,0,1).
\]
It follows that $\langle w_1, w_2 \rangle = \R^3$, which implies $\ortp{w_1, w_2}=\{ \orig \}$ and $w_3=(0,0,0)$. Therefore, $\red(\indv)=(u_1,u_2)$ where $u_1=w_1=v_1$ and $u_2=(0,0,1)$. 
\end{example}

\begin{lemma}\label{rem:env vi = env ui}
Let $\indv=(v_1, \dots, v_k)$, $w_1, \dots, w_k$, and $u_1, \ldots, u_t$ as in \Cref{def:reduction}. If $\red((v_1, \dots, v_i)) = (u_1, \dots, u_{s_i})$ with $i \le k$ and $s_i \le t$, then
\begin{enumerate}[resume]
\item\label{rem:env vi = env ui:item1} $v_i-w_i \in \env{w_1, \ldots, w_{i-1}}$.
\item\label{rem:env vi = env ui:item2} $\env{v_1, \dots, v_i} = \env{w_1, \dots, w_i} = \env{u_1, \dots, u_{s_i}}$.
\end{enumerate}
\end{lemma}

\begin{proof}
\Cref{rem:env vi = env ui:item1}. Since $w_i$ is the projection of $v_i$ onto $\ortp{w_1, \dots, w_{i-1}}$, the vector $v_i-w_i$ is the projection of $v_i$ onto $\env{w_1, \dots, w_{i-1}}$. 

\Cref{rem:env vi = env ui:item2}. We show by induction on $i \ge 1$ that $\env{v_1, \dots, v_i} = \env{w_1, \dots, w_i}= \env{u_1, \dots, u_{s_i}}$.
By \cref{def:reduction}, $v_1=w_1=u_1$, and so the base case follows. Suppose the claim is true for $i \ge 1$. By \Cref{rem:env vi = env ui:item1} and the induction hypothesis, $v_{i+1}-w_{i+1} \in \env{w_1, \dots, w_i} = \env{v_1, \dots, v_i}$. 
Thus, $v_{i+1} \in \env{w_1, \dots, w_{i+1}}$ and $w_{i+1} \in \env{v_1, \dots, v_{i+1}}$, which implies that $\env{v_1, \dots, v_{i+1}} = \env{w_1, \dots, w_{i+1}}$. Since $u_1, \dots, u_{s_{i+1}}$ are obtained by normalizing the non-zero vectors among $w_1, \dots, w_{i+1}$, it follows that $\env{w_1, \dots, w_{i+1}} = \env{u_1, \dots, u_{s_{i+1}}}$.
\end{proof}

The following technical lemmas are the analogues of \cref{l:zero-index,l:existence-linear-maps} in the context of $\ell$-groups.

\begin{lemma}\label{lem:zeros-index-equiv-reduced}
Let $f \in \lz$ and $\indv=(v_1, \dots, v_k)$ an index with $\red(\indv)=(u_1, \dots, u_t)$. 
\begin{enumerate}
\item\label{lem:zeros-index-equiv-reduced:item1} $f(v_i)=0$ for each $i \le k$ if and only if $f(u_j)=0$ for each $j \le t$.
\item\label{lem:zeros-index-equiv-reduced:item2} The following conditions are equivalent:
\begin{enumerate}[label=\roman*., ref=\roman*]
\item\label{lem:zeros-index-equiv-reduced:item2a} There exists $i \le k$ such that $f(v_i)>0$ and $f(v_j)=0$ for each $j < i$.
\item\label{lem:zeros-index-equiv-reduced:item2b} There exists $s \le t$ such that $f(u_s)>0$ and $f(u_r)=0$ for each $r < s$.
\end{enumerate}
\end{enumerate}
\end{lemma}

\begin{proof}
\Cref{lem:zeros-index-equiv-reduced:item1}. We have that  $f(v_i)=0$ for each $i \le k$ if and only if $\{ v_1, \dots, v_k \} \subseteq \VI(f)$. By \Cref{l:subspaces-intersections}, $\VI(f)$ is a rational subspace of $\R^n$. Thus, $\{ v_1, \dots, v_k \} \subseteq \VI(f)$ if and only if $\env{v_1, \dots, v_k} \subseteq \VI(f)$. By \Cref{rem:env vi = env ui}(\ref{rem:env vi = env ui:item2}), $\env{v_1, \dots, v_k} = \env{u_1, \dots, u_t}$. It follows that $\env{v_1, \dots, v_k} \subseteq \VI(f)$ if and only if $\env{u_1, \dots, u_t} \subseteq \VI(f)$ which is equivalent to $f(u_j)=0$ for each $j \le t$ because $\VI(f)$ is a rational subspace.

\Cref{lem:zeros-index-equiv-reduced:item2}. Suppose that condition (\ref{lem:zeros-index-equiv-reduced:item2a}) holds. 
By \Cref{p:Z-red-iff-equal-red}(\ref{p:Z-red-iff-equal-red:item3}), there exists $s \le t$ such that $\red((v_1, \dots, v_i))=(u_1, \dots, u_s)$. We have that $v_i \notin \env{v_1, \ldots, v_{i-1}}$ because $\env{v_1, \ldots, v_{i-1}} \subseteq \VI(f)$ and $v_i \notin \VI(f)$. It follows, using the notation of \Cref{def:reduction}, that $w_i$ is non-zero. Thus, $\red((v_1, \dots, v_{i-1}))=(u_1, \dots, u_{s-1})$ and $u_s=w_i/\lVert w_i \rVert$. So, condition (\ref{lem:zeros-index-equiv-reduced:item2a}) and \Cref{lem:zeros-index-equiv-reduced:item1} imply that $f(u_r)=0$ for each $r < s$. By \Cref{rem:env vi = env ui}(\ref{rem:env vi = env ui:item1},\ref{rem:env vi = env ui:item2}), $v_i-w_i \in \env{v_1, \ldots, v_{i-1}}$, and hence $f(v_i)=f(w_i)$.
Therefore, 
\[
f(u_s)=f(w_i/\lVert w_i \rVert)=f(w_i)/\lVert w_i \rVert=f(v_i)/\lVert w_i \rVert >0.
\]
Suppose that condition (\ref{lem:zeros-index-equiv-reduced:item2b}) holds. It follows from \Cref{def:reduction} that there exists a smallest index $i \le k$ such that $\red((v_1, \dots, v_i))=(u_1, \dots, u_s)$. Since $i$ is the smallest index with this property, $\red((v_1, \dots, v_{i-1}))=(u_1, \dots, u_{s-1})$ and $u_s=w_i/\lVert w_i \rVert$. By arguing as above, we obtain that $f(v_j)=0$ for each $j < i$ and $f(v_i)=f(w_i)=\lVert w_i \rVert f(u_s) >0$.
\end{proof}

\begin{lemma}\label{lem:zero-index-zred}
Let $x \in \Uz^n$ with $\red(\imm(x))=(u_1, \ldots, u_t)$ and $f \in \lz$. 
\begin{enumerate}
\item \label{lem:zero-index-zred:item1} $\en{f}(x) = 0$ if and only if $f(u_i)=0$ for each $i=1, \ldots, t$. 
\item \label{lem:zero-index-zred:item2} $\en{f}(x) > 0$ if and only if there exists $i$ such that $f(u_j)=0$ for each $j<i$ and $f(u_i) > 0$. 
\item \label{lem:zero-index-zred:item3} $\en{f}(x) < 0$ if and only if there exists $i$ such that $f(u_j)=0$ for each $j<i$ and $f(u_i) < 0$. 
\end{enumerate}
\end{lemma}

\begin{proof}
\Cref{lem:zero-index-zred:item1,lem:zero-index-zred:item2} are immediate consequences of \Cref{l:zero-index,lem:zeros-index-equiv-reduced}. \Cref{lem:zero-index-zred:item3} follows from \Cref{lem:zero-index-zred:item2} applied to $-f$.
\end{proof}

\begin{lemma}\label{l:existence-linear-Zmaps}
Let $V$ be a rational subspace of $\mathbb{R}^n$.
\begin{enumerate}
\item \label{l:existence-linear-Zmaps:item1} If $w \in V^\perp$ is a unit vector, then there exists $g \in \lz$ such that $g(w)<0$, and $g(v)=0$ for all $v \in V$.
\item \label{l:existence-linear-Zmaps:item2} If $u,w \in V^\perp$ are two distinct unit vectors, then there exists $g \in \lz$ such that $g(u)>0$, $g(w)<0$, and $g(v)=0$ for all $v \in V$.
\end{enumerate}
\end{lemma}

\begin{proof}
\Cref{l:existence-linear-Zmaps:item1}. By \Cref{l:existence-linear-maps}(\ref{l:existence-linear-maps:item1}), there exists $f \in \lr$ such that $f(w)<0$, and $f(v)=0$ for all $v \in V$. Since $f$ is linear, there exists $y \in \mathbb{R}^n$ such that $f(x)=x \cdot y$ for each $x\in \mathbb{R}^n$, where $\cdot$ denotes the dot product in $\R^n$. Then $y \in V^\perp$.
By \Cref{lem:perp-is-rational}, $V^\perp$ is rational, and so ${V^\perp \cap \Q^n}$ is dense in $V^\perp$ by \Cref{l:subspaces-intersections}. Therefore, we can find $z \in {V^\perp \cap \Q^n}$ such that $\lVert z-y \lVert < -f(w)/\lVert w \rVert$. Let $g\colon \mathbb{R}^n \to \mathbb{R}$ be the linear map given by $g(x)=x \cdot z$ for each $x \in \R^n$. Since $f(w)=w \cdot y$, we have $g(w)=w\cdot (z-y)+f(w)$. Then  $g(w)<0$ because $w\cdot (z-y)\le\lVert w \rVert \lVert z-y \lVert < -f(w)$. Moreover, $g(v)=v \cdot z=0$ for all $v \in V$ because $z\in V^\perp$. The coefficients of the linear polynomial defining $g$ are rational numbers because $z \in \Q^n$. By multiplying its coefficients by the product of the absolute values of their denominators, we can assume $g \in \lz$.

\Cref{l:existence-linear-Zmaps:item2}. By \Cref{l:existence-linear-maps}(\ref{l:existence-linear-maps:item2}), there exists $f \in \lr$ such that $f(u)>0$, $f(w)<0$, and $f(v)=0$ for all $v \in V$. Because $f$ is linear, there is $y \in \mathbb{R}^n$ such that $f(x)=x \cdot y$ for each $x\in \mathbb{R}^n$. Since $V^\perp$ is a rational subspace, $V^\perp \cap \Q^n$ is dense in $V^\perp$ by \cref{l:subspaces-intersections}, so there exists $z \in V^\perp \cap \mathbb{Q}^n$ such that $\lVert z-y \lVert < \min(f(u)/\lVert u \rVert,-f(w)/\lVert w \rVert)$. By arguing as in \Cref{l:existence-linear-maps:item1}, we obtain $g \in \lz$ satisfying $g(u)>0$, $g(w)<0$, and $g(v)=0$ for all $v \in V$.
\end{proof}

We use indexes to characterize the closure operator of the Zariski topology on $\Uz^n$ relative to $\ell$-groups.

\begin{theorem}\label{lem:clx-redix}
Let $x,y \in \Uz^n$. Then $x \in \VUst \CU (y)$ if and only if $\red(\imm(x)) \le \red(\imm(y))$.
\end{theorem}

\begin{proof}
Assume that $\red(\imm(x)) \le \red(\imm(y))$. By \Cref{lem:zero-index-zred}, $\en{f}(y) \ge 0$ implies $\en{f}(x) \ge 0$ for each $f \in \lz$. Thus, $x \in \VUst \CU (y)$ by \Cref{r:basis-hyper-l-groups}.

Suppose $\red(\imm(x)) \nleq \red(\imm(y))$. Let $\red(\imm(x))=(v_1, \ldots ,v_k)$ and $\red(\imm(y))=(w_1, \ldots ,w_t)$.
We first consider the case that $\red(\imm(y)) < \red(\imm(x))$. Then $t<k$ and $v_i=w_i$ for any $i \le t$. Let $V=\env{v_1, \ldots, v_t}$. Since $\red(\imm(x))$ is a $\Z$-reduced index, $v_{t+1} \in V^\perp$. By \Cref{l:existence-linear-Zmaps}(\ref{l:existence-linear-Zmaps:item1}), there exists $f \in \lz$ such that $f(v_{t+1})<0$ and $f(v_i)=0$ for each $i \le t$. \Cref{lem:zero-index-zred} implies $\en{f}(y)=0$ and $\en{f}(x) < 0$. Thus, $x \notin \VUst \CU (y)$ by \Cref{r:basis-hyper-l-groups}.
If $\red(\imm(y))$ and $\red(\imm(x))$ are incomparable, then there exists $i \le \min(k,t)$ such that $v_j=w_j$ for $j < i$ and $v_i\neq w_i$. Let $V=\env{v_1, \ldots, v_{i-1}}$.
Since $w_i, v_i \in V^\perp$, \Cref{l:existence-linear-Zmaps}(\ref{l:existence-linear-Zmaps:item2}) implies that there exists $f \in \lz$ such that $f(w_i) > 0$, $f(v_i) < 0$, and $f(v_j)=f(w_j)=0$ for each $j < i$. By \Cref{lem:zero-index-zred}, $\en{f}(y)>0$ and $\en{f}(x) < 0$. Thus, $x \notin \VUst \CU (y)$ by \Cref{r:basis-hyper-l-groups}.
\end{proof}

\begin{lemma}\label{lem:truncation-red}
Let $\indu,\indv$ be indexes. If $\indu$ is $\Z$-reduced and $\red(\indv) \le \indu$, then there is an index $\indw$ such that $\indv \le \indw$ and $\red(\indw)=\indu$.
\end{lemma}

\begin{proof}
Let $\indv=(v_1, \ldots, v_k)$ and $\indu=(u_1, \ldots, u_t)$. Since $\red(\indv) \le \indu$, there is $s \le t$ such that $\red(\indv)=(u_1, \ldots, u_s)$. By \Cref{rem:env vi = env ui}, we have that $\env{v_1, \ldots, v_k}=\env{u_1, \ldots, u_s}$. Moreover, $u_{s+1}, \ldots, u_t \in \ortp{u_1, \ldots, u_s}$ because $\indu$ is a $\mathbb{Z}$-reduced index. Thus, $\indw\coloneqq(v_1, \ldots, v_k, u_{s+1}, \ldots, u_t)$ is an index such that $\indv \le \indw$ and $\red(\indw)=\indu$.
\end{proof}

The following theorem, which is one of the main results of the section together with \Cref{lem:clx-coneix}, gives a complete  description of the irreducible closed subsets of the Zariski topology relative to $\ell$-groups.

\begin{theorem}\label{lem:clx-coneix-lgroups}
Let $y \in \Uz^n$. Then $\VUst \CU (y) = \bigcup \left\lbrace \cone(\indv) \mid \red(\indv) = \red(\imm(y)) \right \rbrace$.
\end{theorem}

\begin{proof}
We first show that $\VUst \CU (y) = \bigcup \left\lbrace \cone(\indw) \mid \red(\indw) \le \red(\imm(y)) \right \rbrace$. If $x \in \VUst \CU (y)$, then  $\red(\imm(x)) \le \red(\imm(y))$ by \Cref{lem:clx-redix}. Thus, $x \in \cone(\imm(x))$ with $\red(\imm(x)) \le \red(\imm(y))$. On the other hand, if $x \in \cone(\indw)$ with $\red(\indw) \le \red(\imm(y))$, then $\imm(x) \le \indw$. By \Cref{p:Z-red-iff-equal-red}(\ref{p:Z-red-iff-equal-red:item3}),  $\red(\imm(x)) \le \red(\indw) \le \red(\imm(y))$.
\Cref{lem:clx-redix} yields that $x \in \VUst \CU (y)$. It remains to show that 
\begin{equation}\label{eq:equal Conel}
\bigcup \left\lbrace \cone(\indw) \mid \red(\indw) \le \red(\imm(y)) \right \rbrace=\bigcup \left\lbrace \cone(\indv) \mid \red(\indv) = \red(\imm(y)) \right \rbrace.    
\end{equation}
The right-to-left inclusion is immediate. To show the other inclusion assume that $x \in \cone(\indw)$ with $\red(\indw) \le \red(\imm(y))$. Then $\imm(x) \le \indw$ and \Cref{p:Z-red-iff-equal-red}(\ref{p:Z-red-iff-equal-red:item3}) implies that $\red(\imm(x)) \le \red(\indw) \le \red(\imm(y))$. By \Cref{lem:truncation-red}, there exists an index $\indv$ such that $\imm(x) \le \indv$ and $\red(\indv)=\red(\imm(y))$. Therefore, $x \in \cone(\indv)$ with $\red(\indv)=\red(\imm(y))$.
\end{proof}

\begin{corollary}\label{cor:cone-irr-order-iso-lgroups}
Let $\cone^\ell$ be the map that associates $\bigcup \left\lbrace \cone(\indv) \mid \red(\indv) = \indu \right \rbrace$ to any $\Z$-reduced index $\indu$. 
\begin{enumerate}
\item\label{cor:cone-irr-order-iso-lgroups:item1} $\cone^\ell$ is an order isomorphism between the set of $\Z$-reduced indexes ordered by truncation and the set of irreducible closed subsets of $\Uz^n$ ordered by inclusion.
\item\label{cor:cone-irr-order-iso-lgroups:item2} ${\CU} \circ {\cone^\ell}$ is an order isomorphism between the set of indexes ordered by truncation and $\spec(\free_n)$ ordered by reverse inclusion.
\end{enumerate}
\end{corollary}

\begin{proof}
\Cref{cor:cone-irr-order-iso-lgroups:item1}. The proof that $\cone^\ell$ is an onto well-defined map is analogous to the one of \Cref{cor:cone-irr-order-iso}(\ref{cor:cone-irr-order-iso:item1}) and uses \Cref{lem:clx-coneix-lgroups} instead of \Cref{lem:clx-coneix}. It remains to show that $\cone^\ell$ preserves and reflect the order. By \cref{eq:equal Conel}, $\cone^\ell(\indu) = \bigcup \left\lbrace \cone(\indv) \mid \red(\indv) \le \indu \right \rbrace$ for any $\Z$-reduced index $\indu$. Thus, $\indu \le \indw$ implies $\cone^\ell(\indu) \subseteq \cone^\ell(\indw)$. 
Suppose that $\cone^\ell(\indu) \subseteq \cone^\ell(\indw)$, where $\indu,\indw$ are $\Z$-reduced indexes. By arguing as in the proof of \cref{cor:cone-irr-order-iso}\eqref{cor:cone-irr-order-iso:item1}, we can find $y \in \Uz^n$ such that $\imm(y)=\indu$. Then $y \in \cone^\ell(\indu) \subseteq \cone^\ell(\indw)$, and hence $y \in \cone(\indv)$ for some index $\indv$ such that $\red(\indv)=\indw$. 
We have that $y \in \cone(\indv)$ if and only if $\indu=\imm(y) \le \indv$. Therefore, $\indu=\red(\indu) \le \red(\indv)=\indw$ by \Cref{p:Z-red-iff-equal-red}(\ref{p:Z-red-iff-equal-red:item3}).

The proof of \Cref{cor:cone-irr-order-iso-lgroups:item2} is analogous to the one of \Cref{cor:cone-irr-order-iso}(\ref{cor:cone-irr-order-iso:item2}).
\end{proof}

\begin{example}\label{ex:conel}
Let $y=(1,\sqrt{2}) \in \Uz^2$. Then $\imm(y)=((1,\sqrt{2})/\sqrt{3})$, which is clearly a $\Z$-reduced index. Since $\langle (1,\sqrt{2})/\sqrt{3} \rangle = \R^2$, there are exactly three indexes $\indv$ such that $\red(\indv)=\imm(y)$, which are $\imm(y)$ itself and
\[
\left(\frac{1}{\sqrt{3}}(1,\sqrt{2}), \pm\frac{1}{\sqrt{3}}(\sqrt{2},-1) \right).
\]
It follows from \Cref{lem:clx-coneix-lgroups} that
\begin{align*}
\VUst \CU (y) &= \cone^\ell(\imm(y)) \\
&= \{ \alpha (1,\sqrt{2}) + \beta (\sqrt{2},-1) \mid \alpha, \beta \in \U, \ \alpha>0 \text{ and } \beta/\alpha \text{ is infinitesimal} \}\\
&= \{ (\gamma, \gamma \sqrt{2} + \delta) \mid \gamma, \delta \in \U, \ \gamma>0 \text{ and } \delta/\gamma \text{ is infinitesimal} \},
\end{align*}
where the last equality follows from straightforward algebraic manipulations.
In this case $\cone(\imm(y))$ is strictly contained in $\cone^\ell(\imm(y))$ because
\[
\cone(\imm(y)) = \{ \alpha(1, \sqrt{2}) \mid 0 < \alpha \in \U\}.
\]
\end{example}

\begin{remark}\label{rem:Icone=Iconel}
As \Cref{ex:conel} shows, if $\indv$ is a $\Z$-reduced index, then $\cone(\indv)$ is strictly contained in $\cone^\ell(\indv)$ in general. However, we always have that $\CU(\cone(\indv))=\CU(\cone^\ell(\indv))$. Indeed, $\cone(\indv) \subseteq \cone^\ell(\indv)$ implies $\CU(\cone^\ell(\indv)) \subseteq \CU(\cone(\indv))$. To show the other inclusion, let $t \in \freen$ and suppose $t \in \CU(\cone(\indv))$. 
Take $y \in \Uz^n$ such that $\imm(y)=\indv$. Then $t(y)=0$, and so $t \in \CU(y)$. By \cref{lem:clx-coneix-lgroups}, $\cone^\ell(\indv)=\cone^\ell(\imm(y))=\VUst\CU(y) \subseteq \VUst(t)$. Therefore, $t \in \CU(\cone^\ell(\indv))$. This shows that $\CU(\cone(\indv)) \subseteq \CU(\cone^\ell(\indv))$. 
\end{remark}

We end this section with some observations analogous to the ones of \Cref{rem:coordinatization}, which yield a coordinatization of $\spec(\freen)$ in the setting of $\ell$-groups.

\begin{remark}\label{rem:coordinatization-lgroups}
Fix a positive infinitesimal $\varepsilon \in \U$. By \Cref{cor:cone-irr-order-iso-lgroups}\eqref{cor:cone-irr-order-iso-lgroups:item2}, for each $P \in \spec(\freen)$ there exists a unique $\Z$-reduced index $\indv=(v_1, \dots, v_k)$ such that $\cone^\ell(\indv)=\VUst(P)$. We then define $\mathscr{E}(P)=v_1 + \varepsilon v_2 + \dots + \varepsilon^{k-1} v_k$. 
\Cref{lem:clx-coneix-lgroups} allows us to argue as in \Cref{rem:coordinatization} and obtain that $\mathscr{E}$ is an embedding of $\spec(\freen)$ into $\Uz^n$. 
Observe that the image of $\mathscr{E}$ still contains $\mathbb{S}^{n-1}$ but it is in general a proper subset of the image of the map $\mathscr{E}$ defined in \Cref{rem:coordinatization}. 
\end{remark}

%%%%%%%%%%% %%%%%%%%%%% SECTION %%%%%%%%%%%%%%%%%%%%%% 
\section{Characterization of prime ideals in free \texorpdfstring{$\ell$}{l}-groups and Riesz spaces}
\label{sec:Panti}
%%%%%%%%%%%%%%%%%%%%%%%%%%%%%%%%%%%%%%%%%%%%%%%%%%
In this section we show that the results of \Cref{sec:irreducible} can be used to obtain an alternative proof of the characterization of prime ideals in finitely generated free $\ell$-groups and Riesz spaces given in \cite[Theorems~3.8 and 4.8]{MR1707667}.

\begin{definition}
Let $V$ be an $\mathbb{R}$-vector space or a $\U$-vector space. A subset $C \subseteq V$ is a \emph{cone} if it is closed under multiplication by nonnegative scalars. If $T$ is a subset of $V$, then the smallest cone containing $T$ is called the \emph{positive span} of $T$ and is given by the $x \in V$ for which there exist $v_1, \ldots, v_k \in T$ and non-negative scalars $s_1, \ldots, s_k$ such that $x=s_1v_1+\cdots + s_k v_k$.
\end{definition}

\begin{definition}\label{def:v-cone}
Let $\indv=(v_1, \ldots, v_k)$ be an index. We say that a cone $C \subseteq \mathbb{R}^n$ is a \emph{$\indv$-cone} if there exist real numbers $0< r_1, \ldots, r_k \in \mathbb{R}$ such that $C$ is the positive span of the set
\[
\left\{ \sum_{i=1}^j r_i v_i \mid j = 1, \ldots, k \right\}.
\]
In other words, $x \in C$ if and only if there exist $0 \le s_1, \ldots, s_k \in \mathbb{R}$ such that
\[
x= \sum_{j=1}^k \left( s_j \sum_{i=1}^j r_i v_i \right) = s_1(r_1 v_1)+s_2(r_1 v_1+r_2 v_2)+ \cdots + s_k(r_1 v_1 + r_2 v_2 + \cdots + r_k v_k ).
\]
\end{definition}

\begin{remark}\label{rem:en-v-cones}
Let $C$ be a $\indv$-cone given by $r_1, \ldots, r_k$ as in \Cref{def:v-cone}. By the transfer principle, $y \in \en{C}$ if and only if there exist $0 \le \beta_1, \ldots, \beta_k \in \U$ such that
\[
y = \sum_{j=1}^k \left( \beta_j \sum_{i=1}^j r_i v_i \right) = \beta_1(r_1 v_1)+\beta_2(r_1 v_1+r_2 v_2)+ \cdots + \beta_k(r_1 v_1 + r_2 v_2 + \cdots + r_k v_k ).
\]
\end{remark}

\begin{theorem}\label{lem:conev-as-intersection}
If $\indv$ is an index, then 
\[
\cone(\indv)= \bigcap \{ \en{C} \mid C \mbox{ is a $\indv$-cone} \} \setminus \{ \mathbf{0} \}.
\]
\end{theorem}

\begin{proof}
Let $\indv=(v_1, \ldots, v_k)$. If $x \in \cone(\indv)$, then $\imm(x) \le \indv$. Thus, the orthogonal decomposition of $x$ is of the form $\alpha_1 v_1 + \cdots + \alpha_t v_t$ where $t \le k$. For any $0 < r_1, \ldots, r_k \in \mathbb{R}$ we show that $x \in \en{C}$ where $C$ is the $\indv$-cone that is the positive span of $\{ \sum_{i\le j} r_i v_i \mid j \le k \}$. Let $\beta_i \in \U$ for each $i=1, \ldots, k$ be defined as follows: if $i > t$ set $\beta_i=0$, if $i=t$ set $\beta_t=\alpha_t/r_t$, and if $i < t$ set
\[
\beta_i= \frac{\alpha_i}{r_i} - \sum_{j=i+1}^t \frac{\alpha_j}{r_j}.
\]
It follows that
\begin{align*}
\beta_1(r_1 v_1) &+ \cdots + \beta_k(r_1 v_1+\cdots + r_k v_k) = \beta_1(r_1 v_1)+ \cdots + \beta_t(r_1 v_1+\cdots + r_t v_t) \\
& = \left(\frac{\alpha_1}{r_1}-\frac{\alpha_2}{r_2}- \cdots - \frac{\alpha_t}{r_t}\right)(r_1 v_1) + \cdots + \left(\frac{\alpha_t}{r_t}\right)(r_1 v_1+\cdots + r_t v_t)\\
& =\alpha_1 v_1 + \cdots + \alpha_t v_t=x.
\end{align*}
It remains to show that $\beta_i \ge 0$ for each $i$. We have that $r_i > 0$ for each $i$ and $\alpha_j/\alpha_i$ is infinitesimal for any $i < j \le t$ by \Cref{l:orthdec-standardpart}(\ref{l:orthdec-standardpart:itema}). So, if $i < j \le t$, then $\alpha_j/\alpha_i < r_j/(r_i(t-1))$ because $\alpha_j/\alpha_i$ is smaller than any positive real number, and hence
\[
\frac{\alpha_j}{r_j} <  \frac{\alpha_i}{r_i(t-1)}.
\]
Therefore, 
\[
\beta_i= \frac{\alpha_i}{r_i} - \sum_{j=i+1}^t \frac{\alpha_j}{r_j} > \frac{\alpha_i}{r_i} - (t-1) \frac{\alpha_i}{r_i(t-1)} = 0
\]
for each $i \le t$. If $i > t$, then $\beta_i =0$. Thus, $x \in \en{C}$ by \Cref{rem:en-v-cones}.

Conversely, let $x \in \Uz^n$ such that $x \in \en{C}$ for every $\indv$-cone $C$. Since $v_1, \ldots, v_k$ are linearly independent vectors in $\mathbb{R}^n$, there exist $v_{k+1}, \ldots, v_n$ such that $v_1, \ldots, v_n$ is a basis of $\mathbb{R}^n$.
It follows that $v_1, \ldots, v_n$ is also a basis of $\U^n$ as a $\U$-vector space. Thus, $x=\alpha_1 v_1+ \cdots+ \alpha_n v_n$ for some unique $\alpha_i \in \U$. By \Cref{rem:en-v-cones}, for every  $r_1, \ldots, r_k$ positive real numbers there exist $\beta_1, \ldots, \beta_k \ge 0$ in $\U$ such that $x=\beta_1(r_1 v_1)+ \cdots + \beta_k(r_1 v_1+\cdots + r_k v_k)$. The uniqueness of the $\alpha_i$'s implies that $\alpha_i=r_i \sum_{j=i}^k \beta_j$ if $i \le k$ and $\alpha_i=0$ if $i>k$.
Thus, $\alpha_i \ge 0$ for each $i$ and $\alpha_j \le \alpha_i$ for $i \le j$. Let $t$ be the largest index such that $\alpha_t \neq 0$. We have 
\[
\frac{\alpha_{i+1}}{\alpha_i} = \frac{r_{i+1}}{r_i} \frac{\sum_{j=i+1}^k \beta_j}{\sum_{j=i}^k \beta_j}  \le \frac{r_{i+1}}{r_i}
\]
for every $i < t$. For any $0 < r \in \mathbb{R}$ we can choose the $r_i$'s so that $r_{i+1}/r_i=r$. Thus, $\alpha_{i+1}/\alpha_i$ is positive and smaller than any positive real number, and hence infinitesimal, for every $i < t$. Therefore, $x = \alpha_1 v_1 + \dots + \alpha_t v_t \in \cone(\indv)$.
\end{proof}

By combining \Cref{lem:clx-coneix,lem:clx-coneix-lgroups,lem:conev-as-intersection} we obtain an additional description of the irreducible closed subsets of $\Uz^n$.

\begin{theorem}\label{teo:intersection-v-cones}
Let $y\in \Uz^n$.
\begin{enumerate}
\item In the Zariski topology of $\Uz^n$ relative to Riesz spaces we have
\[
\VUst \CU (y) = \bigcap \{ \en{C} \mid C \mbox{ is an $\imm(y)$-cone} \}\setminus \{ \mathbf{0} \}.
\]
\item In the Zariski topology of $\Uz^n$ relative to $\ell$-groups we have
\[
\VUst \CU (y) = \bigcup \left\lbrace \bigcap \{ \en{C} \mid C \mbox{ is a $\indv$-cone} \} \mid \red(\indv) = \red(\imm(y)) \right\rbrace \setminus \{ \mathbf{0} \}.
\] 
\end{enumerate}
\end{theorem}

\begin{example}
Let $e_1=(1,0)$ and $e_2=(0,1)$ be the standard basis of $\mathbb{R}^2$ and $\indv$ the index $(e_1,e_2)$. We saw in \Cref{ex:cone(v)} that
if $\varepsilon$ is a positive infinitesimal and $y=(1,\varepsilon) \in \Uz^2$, then the closure of $y$ in the topology of $\Uz^2$ relative to Riesz spaces is
\[
\VUst \CU(y)=\cone(\indv) = \{ (x_1, x_2) \in \Uz^2 \mid x_1 > 0, \ x_2 \ge 0, \text{ and } x_2/x_1 \text{ is infinitesimal} \}.
\]
By \Cref{def:v-cone}, $\indv$-cones are positive spans of sets $\{ e_1, e_1+re_2 \}$ for some $0 < r \in \R$, which are the cones $\{ (x_1,x_2) \in \R^2 \mid x_1 \ge 0 \text{ and } 0 \le x_2 \le r x_1 \}$.
By the transfer principle, the enlargements of the $\indv$-cones are the subsets of $\U^2$ of the form $\{ (x_1,x_2) \in \U^2 \mid x_1 \ge 0 \text{ and } 0 \le x_2 \le r x_1 \}$  with $0 < r \in \R$. By \cref{teo:intersection-v-cones}, we obtain that 
\[
\VUst\CU(x)=\cone(\indv)= \bigcap_{0 < r \in \R} \{ (x_1,x_2) \in \Uz^2 \mid x_1 \ge 0 \text{ and } 0 \le x_2 \le r x_1 \}.
\] 
It can also be shown that $\cone(\indv)=\cone^\ell(\indv)$ because $\red(\indw)=\indv$ implies $\indw=\indv$ for any index $\indw$.
\end{example}

\begin{lemma}\label{lem:zero-on-conev}
Let $\indv$ be an index and $t \in \free_n$. Then $\cone(\indv) \subseteq \VU(t)$ if and only $\VI(t)$ contains a $\indv$-cone.
\end{lemma}

\begin{proof}
We first prove the right-to-left implication. Suppose there exists a $\indv$-cone $C$ such that $C \subseteq \VI(t)$. Then $\en{C} \subseteq \en{(\VI(t))}$. By \Cref{lem:conev-as-intersection}, $\cone(\indv) \subseteq \en{C}$, and \Cref{rem:*V_I(f)=V_U(f)} implies that $\en{(\VI(t))} = \VU(t)$. Therefore, $\cone(\indv) \subseteq \VU(t)$.

We now prove the left-to-right implication. Let $\indv=(v_1, \dots, v_k)$. If $r \in \R$ and $\alpha \in \U$, denote by $S_{\R}(r)$ the positive span of $\left\{ v_1 + r v_2 + \dots + r^{j-1} v_j \mid j \le k \right\}$ in $\R^n$ and by $S_{\U}(\alpha)$ the positive span of $\left\{ v_1 + \alpha v_2 + \dots + \alpha^{j-1} v_j \mid j \le k \right\}$ in $\U^n$.
If $A$ is a subset of $\R^n$, then the transfer principle yields that there exists a positive $r \in \R$ such that $S_{\R}(r) \subseteq A$ if and only if there exists a positive $\alpha \in \U$ such that $S_{\U}(\alpha) \subseteq \en{A}$.
Indeed, the positive span of a subset can be defined by a first-order formula.
Assume $\cone(\indv) \subseteq \VU(t)$.
If $\varepsilon \in \U$ is a positive infinitesimal, then $v_1 + \varepsilon v_2 + \dots + \varepsilon^{j-1} v_j \in \cone(\indv)$ for each $j \le k$. It is straightforward to check that $\cone(\indv)$ is a cone in $\U^n$. Thus, $S_{\U}(\varepsilon) \subseteq \cone(\indv) \subseteq \VU(t)$.
Since $\VU(t)=\en{(\VI(t))}$ and $S_{\U}(\varepsilon) \subseteq \VU(t)$, there exists a positive $r \in \R$ such that $S_{\R}(r) \subseteq \VI(t)$. Therefore, $\VI(t)$ contains the $\indv$-cone $S_{\R}(r)$.
\end{proof}

Recall from \Cref{cor:cone-irr-order-iso,cor:cone-irr-order-iso-lgroups} that ${\CU} \circ {\cone}$ is a bijection between the set of indexes and $\spec(\freersn)$, and that ${\CU} \circ {\cone^\ell}$ is a bijection between the set of $\Z$-reduced indexes and $\spec(\freen)$. We provide an alternative description of these maps using only the language of Baker-Beynon duality.

\begin{theorem}\label{t:I(Cone(v))}
\mbox{}\begin{enumerate}
\item\label{t:I(Cone(v)):item1} If $\indv$ is an index, then 
\[
\CU(\cone(\indv))=\{ t \in \freersn \mid \VI(t) \mbox{ contains a $\indv$-cone}\}.
\]
\item\label{t:I(Cone(v)):item2} If $\indv$ is a $\Z$-reduced index, then 
\[
\CU(\cone^\ell(\indv))=\{ t \in \freen \mid \VI(t) \mbox{ contains a $\indv$-cone}\}.
\]
\end{enumerate}
\end{theorem}

\begin{proof}
\Cref{t:I(Cone(v)):item1}. Let $t \in \freersn$. Then $t \in \CU(\cone(\indv))$ if and only if $\cone(\indv) \subseteq \VU(t)$. By \Cref{lem:zero-on-conev}, $\cone(\indv) \subseteq \VU(t)$ if and only if $\VI(t)$ contains a $\indv$-cone.

\Cref{t:I(Cone(v)):item2} is a consequence of \Cref{t:I(Cone(v)):item1} because $\CU(\cone^\ell(\indv))=\CU(\cone(\indv))$ by \Cref{rem:Icone=Iconel}.
\end{proof}

Finally, we can deduce \cite[Theorems~3.8 and~4.8]{MR1707667} as a consequence of our results.

\begin{corollary}\label{c:panti}
\mbox{}\begin{enumerate}
\item The prime ideals of $\freersn$ are exactly the subsets of the form
\[
\{ t \in \freersn \mid \VI(t) \mbox{ contains a $\indv$-cone}\}
\]
for some index $\indv$.
\item The prime ideals of $\freen$ are exactly the subsets of the form
\[
\{ t \in \freen \mid \VI(t) \mbox{ contains a $\indv$-cone}\}
\]
for some $\Z$-reduced index $\indv$.
\end{enumerate}
\end{corollary}

\begin{proof}
Follows from \Cref{t:I(Cone(v))} because every prime ideal of $\freersn$ is of the form $\CU(\cone(\indv))$ for some index $\indv$ by \Cref{cor:cone-irr-order-iso} and every prime ideal of $\freen$ is of the form $\CU(\cone^\ell(\indv))$ by \Cref{cor:cone-irr-order-iso-lgroups}.
\end{proof}

In \Cref{sec:archim}, we explored the topological conditions that characterize the existence of a strong order unit in both lattice-ordered groups and Riesz spaces. Note that when these algebras are endowed with a designated strong order unit, they can no more be handled in algebraic setting of \Cref{sec:prelim:adjunction}. Nonetheless, both lattice-ordered groups and Riesz spaces with a strong order unit are categorically equivalent to varieties: the ones of MV-algebras and Riesz MV-algebras, respectively. A duality for semisimple MV-algebras, along the line of Baker-Beynon's, has been developed in \cite{MS-StuLog}. The tools developed here can be adapted to extend also this duality to the entire class of MV-algebras and Riesz MV-algebras, as we are currently doing.

\appendix

\section{}\label{appendix}

Let $\U$ be an ultrapower of $\mathbb{R}$ that satisfies \Cref{ass:U-embeds}  with respect to the cardinal $\omega_1$. We construct a $\VU\CU$-fixpoint $C \subseteq \U^{\omega}$ for which the inclusion of \Cref{c:inclusion-subsets-C} is strict (see \Cref{t:appendix}). 
Since there exist non-Archimedean $\ell$-groups (resp. Riesz spaces) that are finitely generated, they must embed into $\U$, and hence $\U$ contains infinitesimal elements (see \Cref{sec:irreducible}  for the definition of infinitesimals). Set $C = \{ \orig \} \cup \bigcup \{ C_n \mid n \ge 1 \}$, where 
\begin{align*}
C_1 = \{ & (x_i)_{i \in \omega} \mid x_0 \ge 0, \ x_1 > 0, \ x_0/x_1 \text{ is infinitesimal, and } x_i=0 \text{ for } i \ge 2 \},\\
C_n = \{ & (x_i)_{i \in \omega} \mid x_0 \ge 0, \ x_1=n x_0, \ x_n > 0, \ x_0/x_n \text{ is infinitesimal,} \text{ and } x_i=0\\
& \text{ for } i \neq 0,1,n \} \text{ for each } n \ge 2.
\end{align*}

We have that $\VU\CU(C)=C$ because a straightforward verification yields that $C = \VU(T)$, where $T \subseteq \free_\omega$ consists of the equivalence classes of following terms:
\begin{align*}
& x_i \wedge 0 &&\text{ for } i \ge 0,\\
&x_i \wedge x_j &&\text{ for } i,j \ge 2 \text{ and } i \neq j,\\
&\lvert x_i \rvert \wedge ((j x_0 - x_i) \vee 0) &&\text{ for } i \ge 2 \text{ and } j \ge 1,\\
&\lvert x_i \rvert \wedge \lvert i x_0 - x_1 \rvert &&\text{ for } i \ge 2,\\
& \bigg((ix_0-x_1) \wedge \bigwedge_{1 < k < i} (kx_0-x_k) \bigg) \vee 0 &&\text{ for } i \ge 2.
\end{align*}

\begin{proposition}\label{p:appendix}
Let $\varepsilon \in \U$ be a positive infinitesimal and $y =(y_i)_{i \in \omega} \in \U^\omega$ given by $y_0=\varepsilon$, $y_1=1$, and $y_i=0$ for $i \ge 2$.
\begin{enumerate}
\item\label{p:appendix:item1} $y \notin \VU\CU(C \cap \R^\omega)$.
\item\label{p:appendix:item2} $y \in \bigcap_{S \in \pfinom} \pi_S^{-1} [\VU\CU(\VU\CU(\pi_S[C]) \cap \mathbb{R}^{\lvert S \rvert})] $.
\end{enumerate}
\end{proposition}

\begin{proof}
\Cref{p:appendix:item1}. For any $n \ge 1$ we have that 
\[
C_n \cap \mathbb{R}^\omega = \{ (x_i)_{i \in \omega} \mid 0 < x_n \in \mathbb{R} \text{ and } x_i=0 \text{ for } i \neq n \}.
\]
Indeed, if $x_0, x_n \in \mathbb{R}$ and $x_0/x_n$ is infinitesimal, then $x_0=0$. Thus, the definition of $C$ yields  
\[
C \cap \mathbb{R}^\omega = \bigcup_{n \ge 1}\{ (x_i)_{i \in \omega} \mid 0 \le x_n \in \mathbb{R} \text{ and } x_i=0 \text{ for } i \neq n \}.
\]
If $t \in \free_\omega$ is the
projection term on the $0^\text{th}$ coordinate, then $t \in \CU(C \cap \R^\omega)$. However, $t(y)=\varepsilon \neq 0$, and hence $y \notin \VU\CU(C \cap \R^\omega)$.

\Cref{p:appendix:item2}. By \Cref{rem:directed intersection}\eqref{rem:directed intersection:item2}, it is sufficient to show for each $n \ge 2$ that
\[
(\varepsilon, 1, 0, \dots, 0) \in \VU\CU(\VU\CU(\pi_{\{0, \dots, n\}}[C]) \cap \mathbb{R}^{n+1})
\]
because $(\varepsilon, 1, 0, \dots, 0) = \pi_{\{0, \dots, n\}}(y)$.
We first prove that $(1/m, 1, 0, \dots, 0) \in \VU\CU(\pi_{\{0, \dots, n\}}[C]) \cap \mathbb{R}^{n+1}$  for each integer $m$ such that $m > n$. Let $z=(z_i)_{i \in \omega}$ be given by $z_0=1/m$, $z_1=1$, $z_m=1/\varepsilon$, and $z_i=0$ for $i \neq 0,1,m$. Since, $z \in C_m \subseteq C$ and $(1/m, 1, 0, \dots, 0)=\pi_{\{0, \dots, n\}}(z)$, we obtain that 
\begin{equation}\label{eq:appendix}
(1/m, 1, 0, \dots, 0) \in \pi_{\{0, \dots, n\}}[C] \cap \R^{n+1} \subseteq \VU\CU(\pi_{\{0, \dots, n\}}[C]) \cap \mathbb{R}^{n+1}.
\end{equation}
Let $t \in \CU(\VU\CU(\pi_{\{0, \dots, n\}}[C]) \cap \mathbb{R}^{n+1})$. By \Cref{eq:appendix}, $t(1/m, 1, 0, \dots, 0) =0$ for any $m > n$. If $f \colon \R \to \R$ is the piecewise linear map given by $f(x)=t(x,1,0, \dots, 0)$, then $f(1/m)=0$ for any $m>n$. Since $f$ is piecewise linear, it follows that there exists $0 < r \in \R$ such that $f(x)=0$ for any $x \in [0,r] \subseteq \R$. We can employ the transfer principle (see \cref{t:transfer}) because $\U$ is an ultrapower of $\R$. So, we obtain that $\en{f}(x)=0$ for any $x \in \U$ such that $0 \le x \le r$, where $\en{f} \colon \U \to \U$ is the enlargement of $f$. Therefore, $\en{f}(\varepsilon)=0$ because $\varepsilon$ is a positive infinitesimal. By \cref{rem:enlargement piecewise linear}, $\en{f}(x)=t(x,1,0, \dots, 0)$ for any $x \in \U$, and hence $t(\varepsilon,1,0, \dots, 0)=0$. Since this holds for all $t \in \CU(\VU\CU(\pi_{\{0, \dots, n\}}[C]) \cap \mathbb{R}^{n+1})$, it follows that
\[
(\varepsilon, 1, 0, \dots, 0) \in \VU\CU(\VU\CU(\pi_{\{0, \dots, n\}}[C]) \cap \mathbb{R}^{n+1}).
\]
\end{proof}

As an immediate consequence of \Cref{p:appendix}, we obtain:

\begin{theorem}\label{t:appendix}
$C$ is a $\VU \CU$-fixpoint such that
\[
\VU\CU(C \cap \mathbb{R}^\omega) \neq \bigcap_{S \in \pfinom} \pi_S^{-1} [\VU\CU(\VU\CU(\pi_S[C]) \cap \mathbb{R}^{\lvert S \rvert})].
\]
\end{theorem}

\begin{remark}
It can be shown that
\[
\VU\CU(C \cap \mathbb{R}^\omega) = \bigcup_{ n \ge 1} \{ (x_i)_{i \in \omega} \mid 0 \le x_n \in \U  \text{ and } x_i=0 \text{ for } i \neq n \}
\]
and
\[
\bigcap_{S \in \pfinom} \pi_S^{-1} [\VU\CU(\VU\CU(\pi_S[C]) \cap \mathbb{R}^{\lvert S \rvert})]= \VU\CU(C \cap \mathbb{R}^\omega) \cup C_1.
\]
\end{remark}

Consequently, \cref{p:intersection of max ideals dual,c:inclusion-subsets-C}  imply that $\free_\omega/\CU(C)$ is an $\omega$-generated $\ell$-group (resp. Riesz space) whose ideal of nilpotents does not coincide with the intersection of all maximal ideals. For a more direct algebraic description of $\free_\omega/\CU(C)$, observe that the list of terms provided at the beginning of the appendix gives a presentation of $\free_\omega/\CU(C)$. Indeed, since $C=\VU(T)$, where $T$ is the set of equivalence classes of terms in that list, we have $\CU(C)=\CU \VU(T)$ is the ideal of $\free_\omega$ generated by $T$.

\bibliographystyle{abbrv}

\end{document}